\documentclass{amsart}
\usepackage[utf8]{inputenc}
\usepackage{hyperref}
\usepackage{enumerate} 
\usepackage{upref}
\usepackage{verbatim} 
\usepackage{color}
\usepackage{graphicx}
\usepackage{bbm}
\usepackage[textsize=tiny]{todonotes}

\newcommand*{\mailto}[1]{\href{mailto:#1}{\nolinkurl{#1}}}
\newcommand{\arxiv}[1]{\href{http://arxiv.org/#1}{arXiv:#1}}
\usepackage{amssymb, amsmath, amsfonts, amsthm}
\usepackage[all]{xy}

\usepackage{marginnote}

\newcommand{\RN}[1]{%
  \textup{\uppercase\expandafter{\romannumeral#1}}%
}

\DeclareMathOperator{\id}{Id}
\DeclareMathOperator{\meas}{meas}
\DeclareMathOperator{\supp}{supp}
 
\newcommand{\dott}{\, \cdot\,}
\newcommand{\epsi}{\varepsilon}
\renewcommand{\P}{\ensuremath{\mathcal{P}}}

\newcommand{\U}{\ensuremath{\mathcal{U}}}

\newcommand{\N}{\ensuremath{\mathcal{N}}}

\newcommand{\D}{\ensuremath{\mathcal{D}}}
\newcommand{\G}{\ensuremath{\mathcal{G}}}
\newcommand{\F}{\ensuremath{\mathcal{F}}}
\newcommand{\C}{\ensuremath{\mathcal{C}}}
\newcommand{\R}{\ensuremath{\mathcal{R}}}

\newcommand{\inv}{{^{-1}}}
\newcommand{\abs}[1]{\left\vert#1\right\vert}

\newcommand{\Real}{\mathbb R}

\newcommand{\indicator}{\mathbb I}
\newcommand{\norm}[1]{\left\Vert#1\right\Vert}

\newcommand{\nn}{\nonumber}

\newtheorem{theorem}{Theorem}[section]

\newtheorem{lemma}[theorem]{Lemma}
\newtheorem{definition}[theorem]{Definition}

\newtheorem{remark}[theorem]{Remark}

\numberwithin{equation}{section}

\allowdisplaybreaks

\begin{document}

\title[Uniqueness for the HS equation]{Uniqueness of conservative solutions for the Hunter--Saxton equation}

\author[K. Grunert]{Katrin Grunert}
\address{Department of Mathematical Sciences\\ NTNU Norwegian University of Science and Technology\\ NO-7491 Trondheim\\ Norway}
\email{\mailto{katrin.grunert@ntnu.no}}
\urladdr{\url{https://www.ntnu.edu/employees/katrin.grunert}}

\author[H. Holden]{Helge Holden}
\address{Department of Mathematical Sciences\\
  NTNU Norwegian University of Science and Technology\\
  NO-7491 Trondheim\\ Norway}
\email{\mailto{helge.holden@ntnu.no}}
\urladdr{\url{https://www.ntnu.edu/employees/holden}}

\thanks{We acknowledge support by the grants {\it Waves and Nonlinear Phenomena (WaNP)} and {\it Wave Phenomena and Stability --- a Shocking Combination (WaPheS)}  from the Research Council of Norway. }  
\date{\today}

\begin{abstract}
We show that the Hunter--Saxton equation $u_t+uu_x=\frac14\big(\int_{-\infty}^x d\mu(t,z)-  \int^{\infty}_x d\mu(t,z)\big)$ and 
$\mu_t+(u\mu)_x=0$ has a unique, global, weak, and conservative solution $(u,\mu)$ of the Cauchy problem on the line.
\end{abstract}
\maketitle

\section{Introduction}\label{formal}
The Hunter--Saxton (HS) equation \cite{MR1135995} reads
\begin{align*}
u_t+uu_x&=\frac14\Big(\int_{-\infty}^x d\mu(t,z)-  \int^{\infty}_x d\mu(t,z)\Big),\\
\mu_t+(u\mu)_x&=0.
\end{align*}
Here $u$ is an $H^1(\Real)$ function for each time $t$, and $\mu(t)$ is a non-negative Radon measure. 
Derived in the context of modeling liquid crystals, the HS equation has turned out to have considerable interest mathematically.  It has, e.g.,  a geometric interpretation 
\cite{MR1978343z,MR2403320,MR2318260,MR2348278,MR3945049}, convergent numerical methods \cite{GNS,HolKarRis:sub05}, and a stochastic version \cite{MR4151173}, 
in addition to numerous extensions and generalizations \cite{MR2653251,MR2525162}, too many to mention here. The first comprehensive study appeared in \cite{MR1361013,MR1361014}.   While the HS equation was originally derived on differential form
\begin{equation*}
(u_t+uu_x)_x=\frac12 u_x^2,
\end{equation*}
where, in the case of smooth functions, $\mu$ equal to $u_x^2$ will automatically satisfy the second equation, we prefer to work on the integrated version.  Note that there are
several ways to integrate this equation, say
\begin{equation*}
u_t+uu_x=\frac12\int_{0}^x u_x^2(t,y)dy,
\end{equation*}
for which the uniqueness of conservative solutions on the half-line has been established in \cite{BZZ},
but we prefer the more symmetric form. 
For us it is essential to introduce a measure $\mu(t)$ on the line such that for almost all times  $d\mu=d\mu_{\rm ac}=u_x^2dx$. The times $t$ when
$d\mu \neq u_x^2 dx$ will precisely be the times when uniqueness can break down.   Our task here is to analyze this situation in detail, and restore uniqueness by carefully selecting  particular solutions called conservative solutions.
The aim of this paper is to show the following uniqueness result (Theorems~\ref{thm:main2} and~\ref{thm:main1}):   \\
{\em For any initial data $(u_0, \mu_0)\in \D$ the Hunter--Saxton equation has a unique global conservative weak solution $(u,\mu)\in\D$.}  Here $\D$ is given in Definition \ref{def:euler}.

In the case of the so-called dissipative solutions, where energy is removed exactly at the times when the measure $\mu$ ceases to be absolutely continuous, 
the uniqueness question has been addressed in \cite{MR2796054} by showing uniqueness of the characteristics. 

The problem at hand can be illustrated by the following explicit example \cite{CGH}. Consider the trivial case $u_0=0$, which clearly has $u(t,x)=0$ as one solution. However, as can be easily verified, also 
\begin{equation}\label{eq:counter}
u(t,x)=-\frac{\alpha}4 t\, \indicator_{(-\infty, -\frac{\alpha}8 t^2)}(x)+ \frac{2x}{t}\, \indicator_{(-\frac{\alpha}8 t^2, \frac{\alpha}8 t^2)}(x)+ \frac{\alpha}4 t\, \indicator_{(\frac{\alpha}8 t^2, \infty)}(x)
\end{equation}
is a solution for any $\alpha\ge0$, with $\mu(0)=\alpha \delta_0$ and $d \mu(t)=4 t^{-2} \indicator_{(-\alpha t^2/8,  \alpha t^2/8)}(x)dx$ for $t\not =0$.  Here  $\indicator_A$ is the indicator (characteristic) function of the set $A$.
  Thus the initial value problem is not well-posed without further constraints.
  
Furthermore, it turns out that the solution $u$ of the HS equation may develop singularities in finite time in the following sense:  Unless the initial
data is monotone increasing, we find
\begin{equation}\label{eq:blow-up}
  \inf(u_x)\to-\infty \text{  as  } t\uparrow t^*=2/\sup(-u_0^\prime).
\end{equation}
Past wave breaking there are at least two different classes of solutions, denoted conservative (energy is conserved) and dissipative (where energy is removed locally) solutions, respectively, and this dichotomy is the source of the interesting behavior of solutions of the equation (but see also \cite{GN,GT}).  We will in this paper consider the so-called conservative case where the associated energy is preserved. 

The natural approach to solve the HS equation is by the use of characteristics, i.e., to solve the equation
\begin{equation}\label{eq:char_def}
\check y_t(t,\xi)= u(t, \check y(t,\xi)), \quad \check y(0,\xi)= \check y_0(\xi).
\end{equation}
However, in this case the function $u=u(t,x)$ will in general only be H\"older and not Lipschitz continuous. This is the crux of the problem.  Thus we cannot expect uniqueness of solutions 
of this equation.  Indeed, it is precisely in the case where uniqueness fails that the HS equation encounters singularities.  See \cite{BC,BZZ,MR2796054,MR1668954,MR1701136,MR1799274}.
We will reformulate the HS equation in new variables, the aim being to identify variables where the singularities disappear. 

Rewriting the HS equation, using characteristics, yields a linear system of differential equations \cite{BHR}, 
\begin{align} 
\check y_t(t,\xi)&= \check U(t,\xi), \nn \\
\check U_t(t,\xi)&= \frac12 (\check H(t,\xi)-\frac12 C), \label{eq:char_intro}\\ 
\check H_t(t,\xi)&=0, \nn
\end{align}
where $C=\mu(0,\Real)=\mu(t,\Real)$.  Here $\check U(t,\xi)=u(t,\check y(t,\xi))$ and 
$\check H(t,\xi)=\int_{-\infty}^{\check y(t,\xi)}u_x^2(t,x)dx$. This system describes weak, conservative solutions and can be integrated to yield
\begin{align*} 
\check y(t,\xi)&= \check y(0,\xi)+t\big(\check U(0,\xi)+ \frac{t}4 (\check H(0,\xi)-\frac12 C)\big),\\
\check U(t,\xi)&=\check U(0,\xi)+ \frac{t}2 (\check H(0,\xi)-\frac12 C),\\ 
\check H(t,\xi)&=\check H(0,\xi).
\end{align*}
Here we may recover $(u,\mu)$ from $u(t,x)=\check U(t,\xi)$ for some $\xi$ such that $x=\check y(t,\xi)$ and $\mu=\check y_\#(\check H_\xi)d\xi$. In particular, it has been shown in \cite{BHR}, that given any initial data $(u_0,\mu_0)\in \D$ there exists at least one conservative solution and this solution satisfies \eqref{eq:char_intro}. On the other hand, the question of uniqueness of  conservative solutions has not been addressed. This question can also be rephrased as: Do all conservative solutions satisfy \eqref{eq:char_intro}?

In \cite{CGH} we introduced a new set of coordinates, which allowed us, in contrast to \cite{BHR}, to construct a Lipschitz metric $d$, which is not based on equivalence classes. The underlying system of differential equations, which has been derived using pseudo-inverses and the system \eqref{eq:char_intro}, is surprisingly simple, but forced us to impose an additional
 first moment condition\footnote{Condition (2.15) in \cite{CGH}, here stated in terms of the measure $\mu$.}, $\int_\Real (1+\abs{x})d\mu_0<\infty$, on the measure. To be more specific, $d$ satisfies 
\begin{equation*}
d((u_1,\mu_1)(t), (u_2,\mu_2)(t))\le \big(1+t+\frac18 t^2 \big)d((u_1,\mu_1)(0), (u_2,\mu_2)(0))
\end{equation*}
for any two weak, conservative solutions $(u_i,\mu_i)\in \D$, which satisfy the additional condition $\int_\Real (1+\abs{x})d\mu_{i,0}<\infty$. The new coordinates are defined as follows. Let  $\chi(t,\eta)=\sup\{x\mid \mu(t, (-\infty,x))<\eta\}$ and $\U(t,\eta)=u(t,\chi(t,\eta))$ and introduce $\hat\chi(t,\eta)=\chi(t,C \eta)$ and $\hat\U(t,\eta)=\U(t,C \eta)$ where $C=\mu(t,\Real)$.  Then we define
\begin{equation*}
d((u_1,\mu_1)(t), (u_2,\mu_2)(t))=\norm{\hat\U_1(t)-\hat\U_2(t)}_{L^\infty}+\norm{\hat\chi_1(t)-\hat\chi_2(t)}_{L^1}+\abs{C_1-C_2}.
\end{equation*}

However, a closer look reveals that one explicitly associates to any initial data the weak conservative solution computed using \eqref{eq:char_intro}. Thus  the question if all weak conservative solutions satisfy \eqref{eq:char_intro} is never addressed.

Furthermore, to study stability questions for conservative solutions the coordinates from \cite{CGH} seem to be favorable, but not for investigating the uniqueness. The main difficulty stems from the fact that for each $t\in \Real$, the function $F(t,x)=\mu(t,(-\infty,x))$, where $\mu$ denotes a positive, finite Radon measure, is increasing but not necessarily strictly increasing. This means, in particular, that its spatial inverse $\chi(t,\eta)$ might have jumps. Time evolution of increasing functions with possible jumps can lead to the same problems as for conservation laws. What happens to jumps as time evolves? Do they satisfy some kind of Rankine--Hugoniot condition or do they behave more like rarefaction waves? In \cite{CGH} this issue has been resolved by using the system \eqref{eq:char_intro} to show that any jump preserves position and  height. 
Thus the associated system for $(\chi(t,\eta), \U(t,\eta))$ cannot be treated using classical ODE theory, but only in a weak sense with some additional constraints. Hence these new variables would not simplify the study of uniqueness questions. 

\medskip
Given a conservative solution $(u,\mu)$, define the quantities 
\begin{align*}
y(t,\xi)&=\sup \{x \mid x+\mu(t,(-\infty,x))<\xi\}, \\
U(t,\xi)&=u(t,y(t,\xi)), \\
\tilde H(t,\xi)&=\xi-y(t,\xi).
\end{align*}
Then one can derive, see Theorem~\ref{thm:main_old}, that these quantities satisfy
\begin{align*}
y_t(t,\xi)+Uy_\xi(t,\xi)&=U(t,\xi),\\
\tilde H_t(t,\xi)+U\tilde H_\xi(t,\xi)&=0,\\
U_t(t,\xi) +UU_\xi(t,\xi)&= \frac12 (\tilde H(t,\xi)-\frac12C).
\end{align*}
In contrast to $u(t,x)$, the function $U(t,\xi)$ is Lipschitz continuous and hence the above system can be solved uniquely using the method of characteristics, which is sufficient to ensure the uniqueness of conservative solutions. In particular, it can be shown that by applying the method of characteristics the above system turns into \eqref{eq:char_intro}, see Remark~\ref{rem:relsys}. 

Although the uniqueness question is successfully addressed, the above system has one main drawback: The definition of the function $y(t,\xi)$ is far from unique. On the other hand, the above system can be used to find other equivalent formulations of the Hunter--Saxton equation, which might be advantageous for addressing, e.g., stability questions.  As an illustration, we here introduce a novel set of coordinates, which can be studied on its own, without relying on special properties of solutions to \eqref{eq:char_intro} and which avoids the formation of jumps but requires to impose additional moment conditions. The main idea is to introduce an auxiliary measure $\nu$, such that $G(t,x)=\nu(t,(-\infty,x))$ is strictly increasing for each $t\in \Real$. To that end 
define the auxiliary function (the power $n$ to be fixed later)
\begin{equation*}
p(t,x)= \int_{\Real}\frac{1}{(1+(x-y)^2)^n}d\mu(t,y),
\end{equation*}
which will be a smooth function for all Radon measures $\mu$ and let
\begin{align*}
\chi(t,\eta)&=\sup\{x\mid \nu(t,(-\infty,x))<\eta\} \\
&=\sup\{x\mid \int_{-\infty}^x p(t,y) dy +\mu(t,(-\infty,x))<\eta\},\\
 \U(t,\eta)&=u(t,\chi(t,\eta)), \\
 \P(t,\eta)&=p(t,\chi(t,\eta)).
\end{align*}
Provided $(u,\mu)$ is a weak, conservative solution of the HS equation, which satisfies an additional moment condition, see \eqref{cond:extramug}, we show, cf.~Theorem \ref{thm:systemnew}, that the triplet $(\chi(t,\eta), \U(t,\eta), \P(t,\eta))$ satisfies
\begin{subequations}\label{systemnew:secA}
\begin{align}
\chi_t+h\chi_\eta&=\U, \label{systemnew:sec1A}\\[2mm] 
\U_t+h\U_\eta &=\frac12 \Big(\eta-\int_0^\eta \P\chi_\eta(t,\tilde \eta) d\tilde\eta\Big)-\frac14C,\label{systemnew:sec2A}\\[2mm] 
\P_t+h\P_\eta&=\R\label{systemnew:sec3A}
\end{align}
\end{subequations}
where 
\begin{align}\label{fdef:hA}
h(t,\eta)&=\U\P(t,\eta) -\int_0^{B+C} \U(t,\tilde\eta)K(\chi(t,\eta)-\chi(t,\tilde\eta))\big(1-\P\chi_\eta(t,\tilde\eta)\big)d\tilde\eta, \\
\R(t,\eta)&=-\int_0^{B+C} \U(t,\tilde\eta)K'(\chi(t,\eta)-\chi(t,\tilde\eta))\big(1-\P\chi_\eta(t,\tilde\eta)\big)d\tilde\eta \nn \\ \label{def:slashR1A} 
& \quad +\U(t,\eta)\int_0^{B+C} K'(\chi(t,\eta)-\chi(t,\tilde\eta))\big(1-\P\chi_\eta(t,\tilde\eta)\big)d\tilde\eta.
\end{align}
In particular, $h(t,\eta)$ is continuous w.r.t.~time and Lipschitz continuous w.r.t.~space, so that the above system has a unique solution and can be solved by applying the method of characteristics. This is sufficient to ensure the uniqueness of conservative solutions that satisfy an additional moment condition, cf.~Theorem~\ref{thm:main1}.

\section{Background}\label{sec:back}

In this section we introduce the concept of weak conservative solutions for the Hunter--Saxton equation. Afterwards we show that there indeed exists at least one weak conservative solution to every admissible initial data. We use $\C_c^\infty$ to denote smooth functions with compact support and $\C_0^\infty$ to denote smooth functions that vanish at infinity.

As a starting point we introduce the spaces we work in. 

\begin{definition}
Let $E$ be the vector space defined by 
\begin{equation}\label{def:Ep}
E=\{f\in L^\infty(\Real) \mid f'\in L^2(\Real)\}
\end{equation}
equipped with the norm $\norm{f}_{E}=\norm{f}_{L^\infty} + \norm{f'}_{L^2}$.
\end{definition}
Furthermore, let 
\begin{equation*}
H_1^1(\Real)=H^1(\Real)\times\Real \quad \text{and} \quad H_2^1(\Real)=H^1(\Real)\times \Real^2.
\end{equation*}
Write $\Real$ as $\Real=(-\infty,1)\cup(-1,\infty)$ and consider the corresponding partition of unity $\psi^+$ and $\psi^-$, i.e., $\psi^+$ and $\psi^-$ belong to $C^\infty(\Real)$, $\psi^-+\psi^+\equiv 1$, $0\leq \psi^\pm\leq 1$, $\supp(\psi^-)\subset (-\infty,1)$, and $\supp(\psi^+)\subset (-1,\infty)$. Furthermore, introduce the linear mapping $\mathcal{R}_1$ from $H_1^1(\Real)$ to $E$ defined as
\begin{equation*}
(\bar f, a)\mapsto f=\bar f+a\psi^+,
\end{equation*}
and the linear mapping $\mathcal{R}_2$ from $H_2^1(\Real)$ to $E$ defined as 
\begin{equation*}
(\bar f, a, b)\mapsto f=\bar f+a\psi^++b\psi^-.
\end{equation*}
The mappings $\mathcal{R}_1$ and $\mathcal{R}_2$ are linear, continuous, and injective. Accordingly introduce $E_1$ and $E_2$, the images of $H_1^1(\Real)$ and $H_2^1(\Real)$ by $\mathcal{R}_1$ and $\mathcal{R}_2$, respectively, i.e.,
\begin{equation}\label{def:Eps}
E_1=\mathcal{R}_1(H_1^1(\Real)) \quad \text{ and } \quad E_2=\mathcal{R}_2(H_2^1(\Real)).
\end{equation}
The corresponding norms are given by 
\begin{equation*}
\norm{f}_{E_1}=\norm{\bar f+a\psi^+}_{E_1}=\Big(\norm{\bar f}_{L^2}^2+\norm{\bar f'}_{L^2}^2+a^2\Big)^{1/2}
\end{equation*}
and
\begin{equation*}
\norm{f}_{E_2}=\norm{\bar f+a\psi^++b\psi^-}_{E_2}=\Big(\norm{\bar f}_{L^2}^2+ \norm{\bar f'}_{L^2}^2+a^2+b^2\Big)^{1/2}.
\end{equation*}
Note that the mappings $\mathcal{R}_1$ and $\mathcal{R}_2$ are also well-defined for all $(\bar f, a)\in L^2_1(\Real)=L^2(\Real)\times \Real$ and $(\bar f, a, b)\in L_2^2(\Real)=L^2(\Real)\times\Real^2$. Accordingly, let 
\begin{equation*}
E^0_1=\mathcal{R}_1(L^2_1(\Real)) \quad \text{ and }\quad E^0_2=\mathcal{R}_1(L^2_2(\Real))
\end{equation*}
equipped with the norms
\begin{equation*}
\norm{f}_{E^0_1}=\norm{\bar f+a \psi^+}_{E^0_1}=\big(\norm{\bar f}^2_{L^2}+a^2\big)^{1/2}
\end{equation*}
and 
\begin{equation*}
\norm{f}_{E^0_2}=\norm{\bar f+ a\psi^++b\psi^-}_{E^0_2}=\big(\norm{\bar f}^2_{L^2}+ a^2+ b^2\big)^{1/2},
\end{equation*}
respectively.

With these spaces in mind, we can define next the admissible set of initial data.
\begin{definition}[Eulerian coordinates]\label{def:euler}
The space $\D$ consists of all pairs $(u,\mu)$ such that 
\begin{itemize}
\item $u\in E_2$,
\item $\mu\in \mathcal{M}^+(\Real)$,
\item $\mu((-\infty,\dott))\in E_1^0$,
\item $d\mu_{\rm ac}=u_x^2 dx$,
\end{itemize}
where $\mathcal{M}^+(\Real)$ denotes the set of positive, finite Radon measures on $\Real$.
\end{definition}

A weak conservative solution is not only a weak solution of the Hunter--Saxton equation, but has to satisfy several additional conditions, which make it possible to single out a unique, energy preserving solution.

\begin{definition}\label{def:loes}
We say that $(u,\mu)$ is a weak conservative solution of the Hunter--Saxton equation with initial data $(u(0,\dott),\mu(0,\dott))\in \D$ if 
\begin{enumerate}
\item At each fixed $t$ we have $u(t,\dott)\in E_2$.
\item At each fixed $t$ we have $\mu(t,(-\infty, \dott))\in E^0_1$ and $d\mu_{\rm ac}(t)=u_x^2(t,\dott)dx$.
\item The pair $(u,\mu)$  satisfies for any $\phi\in \C_c^\infty([0,\infty)\times\Real)$
\begin{subequations}
\begin{align} \label{svak1}
 \int_0^\infty \int_{\Real}\Big[ u\phi_t+\frac12 u^2\phi_x+\frac14 \left(\int_{-\infty}^x d\mu(t)-\int_x^\infty d\mu(t)\right) \phi\Big] dx dt&= -\int_{\Real} u\phi|_{t=0} dx,\\ \label{svak2}
\int_0^\infty \int_{\Real} (\phi_t+u\phi_x)d\mu(t) dt&=-\int_{\Real}\phi |_{t=0}d\mu(0).
\end{align}
\end{subequations}
\item The function $u(t,x)$ defined on $[0,T]\times \Real$ is H{\"o}lder continuous and the map $t\mapsto u(t,\dott)$ is Lipschitz continuous from $[0,T]$ into $E^0_2$.
\item There exists a set $\mathcal{N}\subset \Real$ with $\meas(\mathcal{N})=0$ such that for every $t\not \in \mathcal{N}$ the measure $\mu(t)$ is absolutely continuous and has density $u_x^2(t,\dott)$ w.r.t.~the Lebesgue measure.
\item The family of Radon measures $\{ \mu(t)\mid t\in \Real\}$ depends continuously on time w.r.t.~the topology of weak convergence of measures.  
\end{enumerate}
\end{definition}

Note that the family $\{\mu(t) \mid t\in \Real\}$ provides a measure-valued solution $w$ to the linear transport equation 
\begin{equation*}
w_t+(uw)_x=0. 
\end{equation*}
Thus one has that $\mu(t,\Real)=\mu(0,\Real)$ for all $t\in \Real$.

In \cite{BHR} weak conservative solutions in $\D$ have been constructed. A closer look at their construction reveals that the following theorem holds. 

\begin{theorem}\label{back}
For any initial data $(u_0,\mu_0)\in \D$ the Hunter--Saxton equation has a global conservative weak solution $(u,\mu)$ in the sense of Definition~\ref{def:loes}.
\end{theorem}

In other words, all the properties stated in Definition~\ref{def:loes} are satisfied for the conservative solutions constructed in \cite{BHR}. However, some of them are better hidden than others. This is especially true for (iv) and (vi), which we show here. 

We start by recalling the set of Lagrangian coordinates.

\begin{definition}[Lagrangian coordinates]
The set $\F$ consists of all triplets $(\check y,\check U,\check H)$ such that 
\begin{itemize}
\item $(\check y-\id, \check U,\check H)\in E_2\times E_2\times E_1$,  
\item $(\check y-\id, \check U,\check H)\in [W^{1,\infty}(\Real)]^3$,
\item $\displaystyle{\lim_{\xi\to-\infty}} \check H(\xi)=0$,
\item $\check y_\xi\geq 0$, $\check H_\xi \geq 0$ a.e.,
\item there exists $c>0$ such that $\check y_\xi+\check H_\xi\geq c>0$ a.e.,
\item $\check U_\xi^2=\check y_\xi \check H_\xi$ a.e..
\end{itemize}
\end{definition}

Note that there cannot be a one-to-one correspondence between Eulerian and Lagrangian coordinates. Instead, one has that each element in Eulerian coordinates corresponds to an equivalence class in Lagrangian coordinates. Furthermore, all elements belonging to one and the same equivalence class can be identified using so-called relabeling functions. 

\begin{definition}[Relabeling functions] We denote by $\G$ the group of homeomorphisms $f$ from $\Real$ to $\Real$ such that 
\begin{align}
f-\id  \text{ and } &f^{-1}-\id  \text{ both belong to } W^{1,\infty}(\Real),\\
&f_\xi-1 \text{ belongs to } L^2(\Real),
\end{align}
where $\id$ denotes the identity function.
\end{definition}

Let $X_1=(\check y_1,\check U_1,\check H_1)$ and $X_2=(\check y_2,\check U_2,\check H_2)$ in $\F$.  Then $X_1$ and $X_2$ belong to the same equivalence class if there exists a relabeling function $f\in \G$ such that 
\begin{equation*}
X_1\circ f=(\check y_1\circ f, \check U_1\circ f, \check H_1\circ f)=(\check y_2,\check U_2,\check H_2)=X_2.
\end{equation*}

Furthermore, let
\begin{equation*}
\F_0=\{(\check y,\check U,\check H) \in \F\mid \check y+\check H=\id\}.
\end{equation*}
Then $\F_0$ contains exactly one representative of each equivalence class in $\F$. 
Note that if $X=(\check y,\check U,\check H)\in \F_0$ and $f\in \G$, then one has 
\begin{equation*}
\check y\circ f+\check H\circ f=f.
\end{equation*}
This implies that for each $X=(\check y,\check U,\check H)\in \F$ one has that $\check y+\check H\in \G$.

Whether or not a function is a relabeling function, can be checked using the following lemma, which is taken from \cite{HR}.
\begin{lemma}[Identifying relabeling functions]\label{lem:rel}
If $f$ is absolutely continuous, $f-\id\in W^{1,\infty}(\Real)$, $f_\xi-1\in L^2(\Real)$, and there exists $c\geq 1$ such that $\frac1c\leq f_\xi\leq c$ almost everywhere, then $f\in \G$.
\end{lemma}

\vspace{0.5cm}
In \cite{BHR}, one rewrites the Hunter--Saxton equation, with the help of a generalized method of characteristics as a linear system of differential equations, cf.~\eqref{eq:char_intro},
\begin{subequations}\label{ODEsys}
\begin{align} \label{ODEsys1}
\check y_t(t,\xi)&= \check U(t,\xi),\\
\check U_t(t,\xi)&= \frac12 (\check H(t,\xi)-\frac12 C),\\ \label{ODEsys3}
\check H_t(t,\xi)&=0,
\end{align}
\end{subequations}
where $C=\mu(0,\Real)=\mu(t,\Real)$. This system of differential equations does not preserve $\F_0$, but respects equivalence classes.  
It can be integrated to yield
\begin{subequations}\label{ODEsysSOLVE}
\begin{align} \label{ODEsys1A}
\check y(t,\xi)&= \check y(0,\xi)+t\big(\check U(0,\xi)+ \frac{t}4 (\check H(0,\xi)-\frac12 C)\big),\\
\check U(t,\xi)&=\check U(0,\xi)+ \frac{t}2 (\check H(0,\xi)-\frac12 C),\\ \label{ODEsys3A}
\check H(t,\xi)&=\check H(0,\xi),
\end{align}
\end{subequations}
with initial data determined as introduced next in \eqref{eq:DFdef}. 

The connection between the pairs $(u,\mu)\in \D$ and the triplet $(\check y,\check U,\check H)\in \F$ is given by the following definitions.
\begin{definition} Let the mapping $L\colon\D\to \F_0$ be defined by $L(u,\mu)=(\check y,\check U,\check H)$, where 
\begin{subequations}\label{eq:DFdef}
\begin{align} \label{DF1}
\check y(\xi)&=\sup \{ x\mid x+\mu((-\infty,x))<\xi\},\\ \label{DF2}
\check H(\xi)&=\xi-\check y(\xi),\\
\check U(\xi)&=u\circ \check y(\xi).
\end{align}
\end{subequations}
\end{definition}
\begin{definition}
Let the mapping $M\colon \F\to \D$ be defined by $M(\check y,\check U,\check H)=(u,\mu)$, where\footnote{The push-forward  of a measure $\nu$ by a measurable function $f$ is the measure $f_\#\nu$ defined by
    $f_\#\nu(B)=\nu(f\inv(B))$ for all Borel sets $B$.}
\begin{subequations}
\begin{align}\label{FD1}
u(x)&=\check U(\xi) \quad \text{ for some }\xi \text{ such that } x=\check y(\xi),\\ \label{FD2}
\mu&=\check y_{\#}(\check H_\xi) d\xi.
\end{align}
\end{subequations}
\end{definition}
Now we can finally focus on showing that the weak conservative solutions constructed in \cite{BHR} satisfy Definition~\ref{def:loes} (iv) and (vi).

\subsection{On the H{\"o}lder continuity in the definition of weak conservative solutions} In \cite{BHR} a generalized method of characteristics was used to construct weak conservative solutions as outlined above. This ansatz yields solutions $u$ that are H{\"o}lder continuous with respect to space and time, but not Lipschitz continuous. Indeed, assume we are given a solution $(u,\mu)$ with corresponding Lagrangian coordinates $(\check y, \check U, \check H)$ satisfying \eqref{ODEsys}. Choose two points $(t_1,x_1)$ and $(t_2,x_2)$. Then we can find $\xi_1$ and $\xi_2$ such that 
\begin{equation}
\check y(t_1,\xi_1)=x_1\quad \text{ and }\quad \check y(t_2,\xi_2)=x_2.
\end{equation}
Thus we have 
\begin{align}
\vert u(t_1,x_1)-u(t_2,x_2)\vert \notag
& \leq \vert u(t_1,\check y(t_1,\xi_1))-u(t_2,\check y(t_2,\xi_1))\vert\\ 
& \quad  +\vert u(t_2,\check y(t_2,\xi_1))-u(t_2,\check y(t_1,\xi_1))\vert \\
& \quad +\vert u(t_2, x_1)-u(t_2,x_2)\vert.\notag
\end{align}
As far as the first term on the right-hand side is concerned, we have 
\begin{equation}
\vert u(t_1,\check y(t_1,\xi_1))-u(t_2, \check y(t_2,\xi_1))\vert=\vert \check U(t_1,\xi_1)-\check U(t_2,\xi_1)\vert \leq \frac14 C\vert t_2-t_1\vert,
\end{equation}
where we have used that $-\frac14 C\leq \frac12\check H(t,\xi)-\frac14 C\leq \frac14 C$. For the second term observe that the time variable is the same, but not the space variable. In particular, we have 
\begin{align}
\vert u(t_2, \check y(t_2,\xi_1))-u(t_2,\check y(t_1,\xi_1))\vert & \leq C^{1/2}\big(\check y(t_2,\xi_1)-\check y(t_1,\xi_1)\big)^{1/2}  \notag\\ \nn
&\leq C^{1/2}\Big(\vert \check U(t_1,\xi_1)\vert \vert t_2-t_1\vert +\frac18 C( t_2-t_1)^2\Big)^{1/2}\\
&\leq C^{1/2}\Big(\vert \check U(0,\xi_1)\vert +\frac38CT\Big)^{1/2}\vert t_2-t_1\vert^{1/2}. 
\end{align}
Here we used that $u_x\in L^2(\Real)$ and $d\mu_{\rm ac}=u_x^2dx$ combined with the Cauchy--Schwarz inequality.
Similar considerations yield 
\begin{equation}\label{Holder2}
\vert u(t_2,x_1)-u(t_2,x_2)\vert \leq C^{1/2}\vert x_2-x_1\vert^{1/2},
\end{equation}
thus one ends up with H{\"o}lder continuity with H{\"o}lder exponent $\frac12$.

An important consequence of the above observation is the following. The solution to the ODE
\begin{equation*}
\check y_t(t,\xi)=u(t,\check y(t,\xi))
\end{equation*}
 would be unique if the function $u(t,\dott)$ were Lipschitz continuous. According to \eqref{Holder2} this function is H{\"o}lder continuous with exponent $\frac12$, which leads to the possibility that there might exist several weak conservative solutions to one and the same initial data. 

Moreover, one has, in general, that 
\begin{equation*}
\vert u(t,x)-u(t,y)\vert \leq \norm{u_x(t,\dott)}_{L^\infty}\vert y-x\vert,
\end{equation*}
and hence every time wave breaking occurs, the Lipschitz continuity is lost. 

\subsection{On the Lipschitz continuity in the definition of weak conservative solutions} 

In \cite{BHR} a generalized method of characteristics was used to construct weak conservative solutions. The same approach has been used in \cite{anders}, see also \cite{CGH}, in the case of the two-component Hunter--Saxton system, which generalizes the HS equation. However, there is a slight, but important difference in the solution spaces.

The one used in \cite{anders} is bigger, since one only assumes $u(t,\dott)\in L^\infty(\Real)$ and $F(t,\dott)\in L^\infty(\Real)$ instead of $u(t,\dott)\in E_2^0$ and $F(t,\dott)\in E_1^0$ . Thus one would expect that the mapping $t\mapsto u(t,\dott)$ is Lipschitz continuous from $[0,T]$ into $L^\infty(\Real)$. Yet, a closer look at 
\begin{equation}\label{eq:Lip}
u_t=-uu_x+\frac{1}{2} (F-\frac12 C),
\end{equation}
where 
\begin{equation}\label{def:F}
F(t,x)=\mu(t,(-\infty,x)),
\end{equation}
reveals that $u_t(t,\dott)$ cannot be uniformly bounded in $L^\infty(\Real)$, since $u_x(t,\dott)$ does not belong to $L^\infty(\Real)$ and hence $t\mapsto u(t,\dott)$ is not Lipschitz continuous from $[0,T]$ into $L^\infty(\Real)$. 

The smaller solution space used in \cite{BHR} and here, on the other hand, is the correct choice, since the right-hand side of \eqref{eq:Lip} belongs to $E_2^0$ and hence the mapping $t\mapsto u(t,\dott)$ is Lipschitz continuous from $[0,T]$ into $E_2^0$.

\subsection{On the continuity in the topology of weak convergence of measures  in the definition of weak conservative solutions} In \cite{BHR} a generalized method of characteristics was used to construct weak conservative solutions as outlined above. This ansatz yields measures $\mu$, such that the mapping $t\mapsto \mu(t,\dott)$ is locally Lipschitz continuous if we equip the set of positive Radon measures with the Kantorovich--Rubinstein norm, which generates the weak topology \cite{B}. 

Denote by ${\rm BL}(\Real)$ the space of all bounded and Lipschitz continuous functions equipped with the norm
\begin{equation*}
\norm{\phi}=\max \left\{\norm{\phi}_{L^\infty},\sup_{x\not =y} \frac{\vert \phi(x)-\phi(y)\vert }{\vert x-y\vert}\right \}.
\end{equation*}
Then the Kantorovich--Rubinstein norm of $\mu\in \mathcal{M}^+(\Real)$ is given by 
\begin{equation}\label{KRnorm}
\norm{\mu}_0=\sup\left\{\int_{\Real} \phi d\mu \mid \phi\in {\rm BL}(\Real), \norm{\phi}\leq 1\right\}.
\end{equation}

Given a solution $(u,\mu)$ with corresponding Lagrangian coordinates $(\check y, \check U, \check H)$, which satisfy \eqref{ODEsys}. Let $\phi\in {\rm BL}(\Real)$ such that $\norm{\phi}\leq 1$.  Then we have 
\begin{align}
\left\vert \int_{\Real} \phi(x)d(\mu(t)-\mu(s))\right \vert &=\left\vert \int_{\Real} \big(\phi(\check y(t,\xi))\check H_\xi(t,\xi)-\phi(\check y(s,\xi))\check H_\xi(s,\xi)\big)d\xi\right\vert \nn \\
& = \left\vert \int_{\Real} (\phi(\check y(t,\xi))-\phi(\check y(s,\xi)))\check H_\xi(s,\xi)d\xi\right\vert \nn\\
& \leq \norm{\check y(t,\dott)-\check y(s,\dott)}_{L^\infty} C. \label{eq:skift}
\end{align}
Recalling \eqref{ODEsys}, we have 
\begin{equation*}
\norm{\check y(t,\dott)-\check y(s,\dott)}_{L^\infty} \leq \vert t-s\vert (\norm{\check U(s,\dott)}_{L^\infty}+\frac 18 \vert t-s\vert C).
\end{equation*}
Thus
\begin{equation*}
\left\vert \int_{\Real} \phi(x)d(\mu(t)-\mu(s))\right \vert \leq \vert t-s\vert (\norm{\check U(s,\dott)}_{L^\infty}+ \frac18 \vert t-s\vert C)C
\end{equation*}
for all $\phi\in {\rm BL}(\Real)$ such that $\norm{\phi}\leq 1$, and, in particular,
\begin{equation*}
\norm{\mu(t)-\mu(s)}_0\leq \vert t-s\vert (\norm{\check U(s,\dott)}_{L^\infty}+ \frac18\vert t-s\vert C)C,
\end{equation*}
which proves the local Lipschitz continuity, since $\norm{\check U(s,\dott)}_{L^\infty}$ can be uniformly bounded on any bounded time interval.

Note that we cannot expect global Lipschitz continuity in time due to the last term in the above inequality.

\section{Uniqueness of weak conservative solutions via Lagrangian coordinates}\label{sec:uniqueLag}

The main goal of this section is to present the proof of Theorem~\ref{thm:main2}. To be a bit more precise, we will show that the characteristic equation
\begin{equation*}
\check y_t(t,\xi)=u(t,\check y(t,\xi))
\end{equation*}
has a unique solution and thereby establish rigorously that each weak conservative solution satisfies the system of ordinary differential equations \eqref{ODEsys} in Lagrangian coordinates.  The pair $(u,\mu)$ will be a solution in the sense of Definition \ref{def:loes}. In particular, this means that  the function $u(t,x)$  is H{\"o}lder continuous in $(t,x)$, and the map $t\mapsto u(t,\dott)$ is Lipschitz continuous from $[0,T]$ into $E^0_2$, the set of locally square integrable functions with possible non-vanishing asymptotics at $\pm\infty$.
The measure $\mu(t)$ is  finite, $\mu(t,\Real)=C$, absolutely continuous and has density $u_x^2(t,\dott)$ w.r.t.~the Lebesgue measure, except on a set $\mathcal{N}$ of zero measure. 
Furthermore, no moment condition is assumed on the measure here.

Given a weak conservative solution $(u,\mu)\in \D$, let
\begin{equation}\label{def:y}
y(t,\xi)=\sup \{x \mid x+F(t,x)<\xi\},
\end{equation}
where $F(t,x)$ is given by \eqref{def:F}. 
Then $y(t,\dott)\colon\Real\to \Real$ is non-decreasing, $y(t,\dott)\le \id$, and Lipschitz continuous with Lipschitz constant at most one \cite[Thm.~3.8]{HR}. Furthermore, define
\begin{equation}\label{def:H}
\tilde H(t,\xi)=\xi-y(t,\xi).
\end{equation}
Note that $\tilde H(t,\dott): \Real \to [0,C]$ is non-decreasing (since $y(t,\dott)$ has Lipschitz constant at most one) and continuous.
For completeness, we introduce for later use 
\begin{equation}\label{def:U}
U(t,\xi)=u(t,y(t,\xi)).
\end{equation}

For each time $t$, we have that 
\begin{equation*}
y(t,\xi)=\sup\{x \mid x+F(t,x)<y(t,\xi)+\tilde H(t,\xi)\},
\end{equation*}
which implies that 
\begin{equation}\label{lasr2}
y(t,\xi)+F(t,y(t,\xi)-)\leq  y(t,\xi)+\tilde H(t,\xi)\leq y(t,\xi)+F(t,y(t,\xi)+).
\end{equation}
Subtracting $y(t,\xi)$ in the above inequality, we end up with 
\begin{equation}\label{lasr}
F(t,y(t,\xi)-)\leq \tilde H(t,\xi)\leq F(t,y(t,\xi)+) \quad \text{ for all }\xi\in \Real.
\end{equation}

\begin{remark}\label{rem:relabel}
Note that we made a particular choice in the above calculations, 
\begin{equation*}
y(t,\xi)+\tilde H(t,\xi)=\xi \quad \text{ for all }(t,\xi), 
\end{equation*}
i.e., $X(t, \dott)=(y(t,\dott),U(t,\dott), \tilde H(t,\dott))\in \F_0$ for all $t$. However, we could have chosen any representative of the corresponding equivalence class. Indeed, pick $f(t,x)$ such that $f(t,\dott)\in \G$ for all $t$ and replace $y(t,\xi)$, $U(t,\xi)$, and $\tilde H(t,\xi)$ by
\begin{equation*}
y_1(t,\xi)=\sup\{x\mid x+F(t,x)<f(t,\xi)\},
\end{equation*}
$\tilde H_1(t,\xi)=f(t,\xi)-y_1(t,\xi)$, 
and $U_1(t,\xi)=u(t,y_1(t,\xi))$, respectively. Then 
\begin{equation*}
y_1(t,\xi)+\tilde H_1(t,\xi)=f(t,\xi) \quad \text{ for all } (t,\xi), 
\end{equation*}
i.e., $X_1(t,\dott)=(y_1(t,\dott), U_1(t,\dott), \tilde H_1(t,\dott))\in \F$ for all $t$. In particular, one has
\begin{equation*}
X_1(t,\xi)=X(t,f(t,\xi)) \quad \text{ for all } (t,\xi).
\end{equation*}
\end{remark}

\subsection{The differential equation satisfied by the characteristics $y(t,\xi)$}

Recall that $u(t,x)$ is a weak solution to 
\begin{equation*}
u_t+uu_x=\frac12(F-\frac12C),
\end{equation*}
and hence we obtain, by computing $u(t,\dott)$ along characteristics, that 
\begin{equation}\label{normu}
\norm{u(t,\dott)}_{L^\infty}\leq \norm{u(0,\dott)}_{L^\infty}+\frac14 CT \quad \text{ for all }t\in [0,T].
\end{equation}
Thus for every characteristic $x(s)$ given by
\begin{equation}\label{char}
\dot x(t)=u(t,x(t)),
\end{equation}
we have 
\begin{equation}\label{charrough}
\vert x(t)-x(s)\vert \leq \big(\norm{u(0,\dott)}_{L^\infty} +\frac14 CT\big)\vert t-s\vert.
\end{equation}
Recall that due to the H{\"o}lder continuity of $u$ the equation \eqref{char} will in general not have a unique solution.
This estimate together with the H{\"o}lder continuity of the weak conservative solution, helps us to refine the estimate for $\vert x(t)-x(s)\vert$. Indeed, by assumption we know that there exists a constant $D$ such that 
\begin{equation}\label{Holder}
\vert u(t,x)-u(s,y)\vert \leq D\big(\vert t-s\vert + \vert x-y\vert\big)^{1/2} \quad \text{ for all } (t,x), (s,y)\in [0,T]\times \Real.
\end{equation}
Thus every characteristic $x(t)$ given through \eqref{char}, satisfies 
\begin{equation*}
\vert \dot x(t)-u(s,x(s))\vert \leq D\big(\vert t-s\vert +\vert x(t)-x(s)\vert\big)^{1/2}.
\end{equation*}
Recalling \eqref{charrough} we end up with 
\begin{equation*}
\vert \dot x(t)-u(s,x(s))\vert \leq D\vert t-s\vert^{1/2} \big(1+\norm{u(0,\dott)}_{L^\infty}+\frac14 CT\big)^{1/2}.
\end{equation*}
Integration then yields for all $s$ and $t$ in $[0,T]$ that 
\begin{equation}\label{fineLipx} 
\abs{x(t)-x(s)-u(s,x(s))(t-s)}\le M\vert t-s\vert^{3/2}.
\end{equation}
Here $M$ denotes some positive constant, which is independent of $s$ and $t$.
Furthermore, using \eqref{charrough}, there exists a positive constant $N$ such that 
\begin{equation*}
\abs{x(t)-x(s)}\leq N\vert t-s\vert \quad \text{ for all } s,t\in [0,T].
\end{equation*}

\medskip
We are now ready to turn our attention towards the equation 
\begin{equation}\label{mueq}
\mu_t+(u\mu)_x=0.
\end{equation}
In the case of a classical solution, one has that $F(t,x(t))=F(s,x(s))$. In our more general case, one has 
\begin{equation}\label{normF}
F(t, x(s)-N\vert t-s\vert)\leq F(s,x(s)\pm)\leq F(t, x(s)+ N\vert t-s\vert +).
\end{equation}
We will  show this estimate in the next lemma. For simplicity we let $s=0$ and only consider the right inequality. 
\begin{lemma} \label{lem:hjelp1}  In the above notation, we have the following result
\begin{equation*}
F(0,x(0)+)\leq F(t, x(0)+ N t+) \quad \text{ for all } t\in [0,T].
\end{equation*}
\end{lemma}

\begin{proof}  We have that $\mu(t,\Real)=C$ for all $t$.  We will first show that for a given $\epsi>0$, we can find an $M>0$ such that 
\begin{equation}\label{error}
\mu(t, (-M,M)) \ge C- \epsi \quad \text{ for all } t\in [0,T].
\end{equation}
To that end we first observe that for $t=0$ we can find an $\tilde M_0>0$ such that
\begin{equation*}
F(0,\tilde M_0)-F(0,-\tilde M_0+)=\mu(0, (-\tilde M_0,\tilde M_0)) \ge C- \frac12\epsi.
\end{equation*}
Let $\delta$ be a small positive number to be decided later. Since $u(t,\dott)\in E_2$ for all $t\in[0,T]$, we can find $u_{\pm\infty}(t)$, $x_\delta$, and  $\tilde x_\delta$ such that
\begin{subequations}\label{asymp}
\begin{align}
\abs{u_{-\infty}(t)- u(t,x)}&\le \delta \quad \text{ for all } (t,x)\in [0,T]\times(-\infty,x_\delta], \\
\abs{u_{\infty}(t)- u(t,x)}&\le \delta \quad \text{ for all }(t,x)\in [0,T]\times[\tilde x_\delta,\infty).
\end{align}
\end{subequations} 
Next choose $x_1<x_2<x_3<x_4$ such that
\begin{align*}
x_2+\int_0^t u_{-\infty}(s) ds&\le x_\delta \quad \text{ for all } t\in [0,T], \\
x_3+\int_0^t u_{\infty}(s) ds&\ge \tilde x_\delta \quad \text{ for all } t\in [0,T],
\end{align*}
and pick functions $\psi_1,\psi_2\in C^\infty(\Real)$ with 
\begin{align*}
\psi_1'&\ge0, \quad   \psi_1|_{(-\infty, x_1]}=-\frac12, \quad  \psi_1|_{[x_2,\infty)}=\frac12, \\
\psi_2'&\le0, \quad  \psi_2|_{(-\infty, x_3]}=\frac12, \quad  \psi_2|_{[x_4,\infty)}=-\frac12.
\end{align*} 
Define 
\begin{equation*}
\psi(t,x)=\psi_1\big(x-\int_0^t u_{-\infty}(s) ds\big)+\psi_2\big(x-\int_0^t u_{\infty}(s) ds\big),
\end{equation*}
which clearly satisfies $\psi(t,\dott)\in C^{\infty}_c(\Real)$ for $t\in[0,T]$.  We then find
\begin{align*}
F(t, x_3&+\int_0^t u_{\infty}(s) ds)-F(t, x_2+\int_0^t u_{-\infty}(s) ds+) \\
& \le \int_\Real \psi(t,x)d\mu(t)\\
&\le
F(t, x_4+\int_0^t u_{\infty}(s) ds)-F(t, x_1+\int_0^t u_{-\infty}(s) ds+).
\end{align*} 
Furthermore, using that $(u,\mu)$ is a weak solution and $\psi$ is a test function, we get
\begin{align*}
F(t, x_4&+\int_0^tu_{\infty}(s) ds)-F(t, x_1+\int_0^t u_{-\infty}(s) ds+)\\
&\ge  \int_\Real \psi(t,x)d\mu(t) =\int_\Real \psi(0,x)d\mu(0)+ \int_0^t \int_\Real(\psi_t+u\psi_x) d\mu(s)\, ds \\
&\qquad\qquad\qquad\qquad \,\ge F(0, x_3)-F(0, x_2+)+ \int_0^t \int_\Real(\psi_t+u\psi_x) d\mu(s)\, ds.
\end{align*} 
Direct computations yield
\begin{align*}
(\psi_t+u\psi_x)(t,x)
&= \psi_1'\big(x-\int_0^t u_{-\infty}(s) ds\big)(-u_{-\infty}(t)+u(t,x)) \indicator_{A_1(t)}(x) \\
&\quad +\psi_2'\big(x-\int_0^t u_{\infty}(s) ds\big)(-u_{\infty}(t)+u(t,x)) \indicator_{A_2(t)}(x)
\end{align*} 
with
\begin{align*}
A_1(t)&=\{x\mid x_1+\int_0^t u_{-\infty}(s) ds<x<x_2+\int_0^t u_{-\infty}(s) ds \},\\
A_2(t)&=\{x\mid x_3+\int_0^t u_{\infty}(s) ds<x<x_4+\int_0^t u_{\infty}(s) ds \}.
\end{align*} 
This implies for all $t\in [0,T]$, using \eqref{asymp},
\begin{align*}
\abs{\int_0^t \int_\Real(\psi_t+u\psi_x) d\mu(s)\, ds}&\le\norm{\psi_1'}_{L^\infty} \int_0^t \int_\Real \abs{-u_{-\infty}(s)+u(s,x)}\indicator_{A_1(s)}(x)d\mu(s)\, ds \\
&\quad +\norm{\psi_2'}_{L^\infty} \int_0^t \int_\Real \abs{-u_{\infty}(s)+u(s,x)}\indicator_{A_2(s)}(x)d\mu(s)\, ds \\
&\le \delta T C\big(\norm{\psi_1'}_{L^\infty}+  \norm{\psi_2'}_{L^\infty}\big).
\end{align*} 
Choosing $\delta\le \epsi/(2TC(\norm{\psi_1'}_{L^\infty}+\norm{\psi_2'}_{L^\infty}))$ and $M=\max\{\tilde M_0, \abs{x_\epsi}, \abs{\tilde x_\epsi}, \abs{x_2}, \abs{x_3}\}$,
we see that
\begin{equation*}
\mu(t, (-M,M)) \ge C- \epsi \quad \text{ for all } t\in [0,T].
\end{equation*}

\medskip
Let $\phi\in C^\infty(\Real)$ with $\supp(\phi_x)\subseteq [-\tilde M,\tilde M]$ and $\phi_x\ge0$, such that $\phi(x)=\phi(-\infty)$ for $x\le -\tilde M$ with
$\phi(-\infty)=-\int_{\Real}\phi_x(z) dz$, and $\phi(x)=0$ for $x\ge \tilde M$.
 Introduce $\bar\phi\in C^\infty([0,T]\times\Real)$ such that 
\begin{equation*}
\bar\phi|_{ [0,T]\times[- M, M]}=0,  \quad \bar\phi_x\ge 0, \quad
\bar\phi(t,x)=\begin{cases} 
\phi(-\infty), & \text{for $x\le -1-M$},\\
0, & \text{for $x\ge 1+M$}.
\end{cases}
\end{equation*}
Fix a constant $b$. 
Then we find, since $\phi(x-bt)-\bar\phi(t,x)\in C^\infty_c([0,T]\times\Real)$, that
\begin{align*}
\int_\Real \phi(x-bt) d\mu(t)&= \int_\Real \bar\phi(t,x) d\mu(t)+ \int_\Real (\phi(x-bt)-\bar\phi(t,x)) d\mu(t)\\
&=\int_\Real \bar\phi(t,x) d\mu(t)+ \int_\Real (\phi(x)-\bar\phi(0,x)) d\mu(0)\\
&\quad + \int_0^t \int_\Real\Big((\phi(x-bs)-\bar\phi(s,x))_t \\
&\qquad\qquad\qquad\qquad+u(s,x)(\phi(x-bs)-\bar\phi(s,x))_x  \Big) d\mu(s) ds\\
&=\int_\Real \bar\phi(t,x) d\mu(t)- \int_\Real \bar\phi(0,x) d\mu(0)\\
&\quad- \int_0^t \int_\Real\big(\bar\phi_t +u\bar\phi_x  \big)(s,x) d\mu(s) ds  \\
&\quad+  \int_\Real \phi(x) d\mu(0)+  \int_0^t \int_\Real(u(s,x)-b)\phi'(x-bs)  d\mu(s) ds .
\end{align*}
By \eqref{error} the terms in the next to last line can be made arbitrarily small by increasing $M$, so that 
\begin{equation*}
\int_\Real \phi(x-bt) d\mu(t)= \int_\Real \phi(x) d\mu(0)+  \int_0^t \int_\Real(u(s,x)-b)\phi'(x-bs)  d\mu(s) ds.
\end{equation*}

If we choose $b\geq u$ a.e., then $\int_0^t \int(u-b)\phi_x  d\mu(s) ds\le 0$, which together with 
\begin{align*}
\int_\Real \phi(x-bt) d\mu(t)&=- \int_\Real \phi'(x-bt) F(t,x\pm)dx=- \int_\Real \phi'(x) F(t,x+bt\pm)dx
\end{align*}
yields
\begin{equation*}
\int_\Real \phi'(x)\big(F(t,x+bt\pm)-F(0,x\pm) \big) dx\ge 0,
\end{equation*}
from which we conclude
\begin{equation*}
F(0,x+)\le F(t,x+bt+). 
\end{equation*}
\end{proof}

For any $t\in \Real$, introduce the strictly increasing function $L(t,\dott)\colon\Real \to \Real$ given by
\begin{equation}\label{def:L}
L(t,x)=x+F(t,x),
\end{equation}
which satisfies, cf.~\eqref{def:H} and \eqref{lasr2},
\begin{equation}\label{prop:L}
L(t,y(t,\xi)-)\leq \xi \leq L(t, y(t,\xi)+), \quad t\in[0,T].
\end{equation}
Then 
\begin{equation*}
L(t, x(s)-N\vert t-s\vert)\leq L(s,x(s)\pm)\leq L(t, x(s)+ N\vert t-s\vert +),
\end{equation*}
and choosing $x(s)=y(s,\xi)$, we get
\begin{equation*}
L(t, y(s,\xi)-N\vert t-s\vert )\leq L(t,y(t,\xi)) \leq L(t, y(s,\xi)+N\vert t-s\vert +).
\end{equation*}
Recalling that $L(t,\dott)$ is strictly increasing we end up with 
\begin{equation}\label{dervy}
\abs{y(t,\xi)-y(s,\xi)}\leq  N\vert t-s\vert.
\end{equation}
Since $\xi\mapsto y(t,\xi)$ is Lipschitz with Lipschitz constant at most one, it follows that $y(t,\xi)$ is Lipschitz and hence differentiable almost everywhere in $[0,T]\times \Real$. 

\medskip
Next we aim at computing  $y_t(t,\xi)$,  using \eqref{fineLipx}, and deriving the differential equation for $y(t,\xi)$. One has, combining \eqref{fineLipx} and the analysis used to derive \eqref{normF},
\begin{align*}
F(t,x(s)+&u(s,x(s))(t-s)-M\vert t-s\vert^{3/2})\\
&\leq F(s, x(s)\pm)
\leq F(t,x(s)+u(s,x(s))(t-s)+M\vert t-s\vert^{3/2}+).
\end{align*}
We can derive this estimate as follows. For simplicity let $s=0$ and consider only the right estimate. 

\begin{lemma} \label{lem:hjelp2} In the above notation, we have the following result
\begin{equation*}
F(0,x(0)+)\leq F(t, x(0)+ u_0(x(0))t+Mt^{3/2}+) \quad \text{ for all } t\in [0,T].
\end{equation*}
\end{lemma}
\begin{proof}  Given an $\epsi>0$, we can find, as in the proof of Lemma \ref{lem:hjelp1}, an $M>0$ such that 
\begin{equation}\label{error2}
\mu(t, \Real\backslash(-M,M))<\epsi \quad \text{ for all } t\in [0,T].
\end{equation}
Let $\psi\in C^\infty(\Real)$ with $\supp(\psi_x)\subseteq [-\tilde M,\tilde M]$ and $\psi_x\ge0$. Let $\phi=\phi(t,x)$ satisfy $\phi_t+g \phi_x=0$, with initial data $\phi|_{t=0}=\psi$ for some given continuous function $g=g(t,x)$. Introduce $\bar\phi\in C^\infty([0,T]\times\Real)$ such that 
\begin{equation*}
\bar\phi|_{ [0,T]\times[-M,M]}=0,  \quad \bar\phi_x\ge 0, \quad
\bar\phi(t,x)=\begin{cases} 
-\int_{\Real}\phi_x(t,z) dz, & \text{for $x\le -1-M$},\\
0, & \text{for $x\ge 1+M$}.
\end{cases}
\end{equation*}
Then we find, since $\phi-\bar\phi\in C^\infty_c([0,T]\times\Real)$, that
\begin{align*}
\int_\Real \phi(t,x) d\mu(t)&= \int_\Real \bar\phi(t,x) d\mu(t)+ \int_\Real (\phi-\bar\phi)(t,x) d\mu(t)\\
&=\int_\Real \bar\phi(t,x) d\mu(t)+ \int_\Real (\phi-\bar\phi)(0,x) d\mu(0)\\
& \quad + \int_0^t \int_\Real\big((\phi-\bar\phi)_t +u(\phi-\bar\phi)_x  \big)(s,x) d\mu(s) ds\\
&=\int_\Real \bar\phi(t,x) d\mu(t)- \int_\Real \bar\phi(0,x) d\mu(0)- \int_0^t \int_\Real \big(\bar\phi_t +u\bar\phi_x  \big)(s,x) d\mu(s) ds  \\
&\quad+  \int_\Real \phi(0,x) d\mu(0)+ \int_0^t \int_\Real \big(\phi_t +u\phi_x  \big)(s,x) d\mu(s) ds.
\end{align*}
By \eqref{error2} the terms in the next to last line can be made arbitrarily small by increasing $M$, so that 
\begin{align*}
\int_\Real \phi(t,x) d\mu(t)&=\int_\Real \phi(0,x) d\mu(0)+ \int_0^t \int_\Real \big(\phi_t +u\phi_x  \big)(s,x) d\mu(s) ds\\
& =\int_\Real \phi(0,x) d\mu(0)+ \int_0^t \int_\Real \big(u-g\big)\phi_x (s,x) d\mu(s) ds.
\end{align*}
Let $\bar x\in \Real$ and consider the H{\"o}lder continuous function $g(t,x)=u(0,\bar x)+ D(t+\vert x-\bar x\vert)^{1/2}$, which satisfies $g\geq u$ a.e.~by \eqref{Holder}. 
Furthermore, let $\xi=\xi(t,z)$ solve $\xi_t=g(t,\xi)$ with initial condition $\xi(0,z)=z$.   Then $\phi(t, \xi(t, z))=\phi(0, \xi(0,z))=\psi(z)$  
if $\xi(t,\dott)$ is a strictly increasing function. To see this, observe that one has if $\xi(t,z_2)>\xi(t,z_1)$ for $z_2>z_1$,
 \begin{align*}
 \xi_t(t,z_2)-\xi_t(t,z_1)&=g(t,\xi(t,z_2))-g(t,\xi(t,z_1))\\
 & = D\frac{\vert \xi(t,z_2)-\bar x\vert -\vert \xi(t,z_1)-\bar x\vert}{(t+\vert \xi(t,z_2)-\bar x\vert)^{1/2}+(t+\vert \xi(t,z_1)-\bar x\vert)^{1/2}}\\
 &\geq -\frac{D}{2t^{1/2}} (\xi(t,z_2)-\xi(t,z_1)),
 \end{align*}
 which implies 
 \begin{equation*}
\xi(t,z_2)-\xi(t,z_1)\geq e^{-Dt^{1/2}}(\xi(0,z_2)-\xi(0,z_1))=e^{-Dt^{1/2}}(z_2-z_1),
 \end{equation*}
 and thus $\xi(t,\dott)$ and $\phi(t,\dott)$ are strictly increasing functions.

Since $g\ge u$ a.e., we have $\int_0^t \int(u-g)\phi_x  d\mu(s) ds\le 0$, which together with 
\begin{align*}
\int_\Real \phi(t,x) d\mu(t)&=- \int_\Real \phi_x(t,x) F(t,x\pm)dx\\
&=- \int_\Real \phi_x(t,\xi(t,z)) F(t,\xi(t,z)\pm)\xi_z(t,z)dz\\
&=  - \int_\Real \frac{d}{dz}\phi(t,\xi(t,z)) F(t,\xi(t,z)\pm)dz \\
&=- \int_\Real \frac{d}{dz}\phi(0,\xi(0,z)) F(t,\xi(t,z)\pm)dz  \\
& = -\int  \psi'(z) F(t,\xi(t,z)\pm)dz,
\end{align*}
implies
\begin{equation*}
\int_\Real \psi'(x)\big(F(t,\xi(t,x)\pm)-F(0,x\pm) \big) dx\ge 0,
\end{equation*}
from which we conclude
\begin{equation}\label{est:transport}
F(0,x+)\le F(t,\xi(t,x)+). 
\end{equation}

It remains to estimate $\xi(t,z)-\xi(0,z)$. Integrating the differential equation for $\xi(t,z)$ we find 
\begin{align*}
\vert \xi(t,z)-\xi(0,z)\vert & \leq \vert u(0,\bar x)\vert t+D\int_0^t (\tau+\vert \xi(\tau,z)-\bar x\vert)^{1/2}d\tau\\
& \leq \vert u(0,\bar x)\vert t+ Dt^{1/2} \left(\int_0^t (\tau+\vert \xi(\tau,z)-\bar x\vert) d\tau\right)^{1/2}\\
& \leq \vert u(0,\bar x)\vert t +\frac{D^2}{4}t +\int_0^t (\tau+\vert \xi(\tau,z)-\xi(0,z)\vert +\vert \xi(0,z)-\bar x\vert) d\tau\\
& \leq Lt +\int_0^t \vert \xi(\tau,z)-\xi(0,z)\vert d\tau
\end{align*}
where $L=\vert u(0,\bar x)\vert +\vert \xi(0,z)-\bar x\vert + \frac{D^2}{4}+\frac{T}2$. Thus
\begin{equation*}
L+\vert  \xi(t,z)-\xi(0,z)\vert \leq L+\int_0^t (L+\vert \xi(\tau,z)-\xi(0,z)\vert)d\tau,
\end{equation*}
which by the Gronwall inequality implies 
\begin{equation*}
L+\vert \xi(t,z)-\xi(0,z)\vert \leq Le^t
\end{equation*}
or 
\begin{equation*}
\vert \xi(t,z)-\xi(0,z)\vert \leq Lte^T, \quad \text{ for all } t\in [0,T].
\end{equation*}
Plugging this estimate into the integral representation of the solution, we find
\begin{align*}
\xi(t,z)-\xi(0,z)&\leq\int_0^t (u(0,\bar x)+D(\tau+\vert \xi(\tau,z)-\xi(0,z)\vert+\vert \xi(0,z)-\bar x\vert )^{1/2}) d\tau\\
& \leq u(0,\bar x)t+D\vert \xi(0,z)-\bar x\vert^{1/2} t+D(1+Le^T)t^{3/2}.
\end{align*}
Introducing $\tilde M=D(1+Le^T)$, \eqref{est:transport} reads
\begin{equation*}
F(0,\xi(0,z)+)\leq F(t, \xi(0,z)+u(0,\bar x)t +D\vert \xi(0,z)-\bar x\vert^{1/2} t + \tilde Mt^{3/2}+).
\end{equation*}
A close look reveals that $\tilde M=D(1+(\vert u(0,\bar x)\vert +\vert \xi(0,z)-\bar x\vert +\frac{D^2}{4}+\frac{T}2)e^T)$, which means that 
$\tilde M$ depends linearly on $\vert \xi(0,z)-\bar x\vert $. On the other hand, one has for $z=\xi(0,z)=\bar x$, 
\begin{equation*}
F(0,\bar x+)\leq F(t, \xi(0,\bar x)+u(0,\bar x)t+\tilde Mt^{3/2}+)\leq F(t, \bar x+u(0,\bar x)t+Mt^{3/2}+),
\end{equation*}
where $M=D(1+(\norm{u(0,\dott)}_{L^\infty}+\frac{D^2}{4}+\frac{T}2)e^T)$. Since the above argument holds for any choice of $\bar x\in \Real$ we end up with 
\begin{equation*}
F(0,x+)\leq F(t, x+u(0,x)t+Mt^{3/2}+) \quad \text{ for all }x\in \Real.
\end{equation*}
\end{proof}

\medskip
Recalling \eqref{def:L}, \eqref{prop:L}, and choosing $x(s)=y(s,\xi)$, we get
\begin{align*}
L(t, y(s,\xi)+&u(s,y(s,\xi))(t-s)-M\vert t-s\vert^{3/2})\\
&\leq \xi+u(s,y(s,\xi))(t-s)\\
&\leq L(t, y(s,\xi)+u(s,y(s,\xi))(t-s)+M\vert t-s\vert^{3/2}+),
\end{align*}
and, applying \eqref{prop:L} once more,
\begin{align*}
L(t,y(s,\xi)+&u(s,y(s,\xi))(t-s)-M\vert t-s\vert^{3/2})\\
&\leq L(t,y(t,\xi+u(s,y(s,\xi))(t-s)))\\
&\leq L(t,y(s,\xi)+u(s,y(s,\xi))(t-s)+M\vert t-s\vert^{3/2}+).
\end{align*}
Since $L(t,\dott)$ is strictly increasing we end up with
\begin{equation}\label{fineestimatey}
\abs{y(t,\xi+u(s,y(s,\xi))(t-s))-y(s,\xi)-u(s,y(s,\xi))(t-s)}\leq M\vert t-s\vert^{3/2}.
\end{equation}
Note that the above inequality implies that 
\begin{equation*}
\lim_{s\to t}\frac{y(t,\xi+u(s,y(s,\xi))(t-s))-y(s,\xi)}{t-s}=u(t,y(t,\xi)),
\end{equation*}
since combining \eqref{Holder} and \eqref{dervy} yields
\begin{equation}\label{cont:u}
 \vert u(s,y(s,\xi))-u(t,y(t,\xi))\vert \leq D\big(1+N\big)^{1/2}\vert t-s\vert^{1/2}.
\end{equation}
Thus, one has
\begin{align*}
y_t(t,\xi)& =\lim_{s\to t}\frac{y(t,\xi)-y(s,\xi)}{t-s}\\
&=\lim_{s\to t}\Big(\frac{y(t,\xi)-y(t,\xi+u(s,y(s,\xi))(t-s))}{t-s}\\
&\qquad\qquad\qquad+\frac{y(t,\xi+u(s,y(s,\xi))(t-s))-y(s,\xi))}{t-s}\Big)\\
&=\lim_{s\to t}\frac{y(t,\xi)-y(t,\xi+u(s,y(s,\xi))(t-s))}{t-s}+u(t,y(t,\xi)),
\end{align*}
and it is left to compute
\begin{equation}\label{eq:above}
\lim_{s\to t}\frac{y(t,\xi)-y(t,\xi+u(s,y(s,\xi))(t-s))}{t-s}=\lim_{s\to t}\frac{y(t,\xi)-y(t,\xi+u(t,y(t,\xi))(t-s))}{t-s}.
\end{equation}
Recalling \eqref{cont:u}, the above equality \eqref{eq:above} holds since $y(t,\dott)$ is Lipschitz continuous with Lipschitz constant at most one.
Moreover, note that for $u(t,y(t,\xi))\not =0$, one has 
\begin{align*}
\lim_{s\to t}&\frac{y(t,\xi)-y(t,\xi+u(t,y(t,\xi))(t-s))}{t-s}\\
&\qquad\qquad\qquad =u(t,y(t,\xi))\lim_{s\to t}\frac{y(t,\xi)-y(t,\xi+u(t,y(t,\xi))(t-s))}{u(t,y(t,\xi))(t-s)}\\
& \qquad\qquad\qquad =-u(t,y(t,\xi))y_\xi(t,\xi).
\end{align*}
This result also remains valid in the case $u(t,y(t,\xi))=0$. Hence we conclude that $y(t,\xi)$ satisfies
\begin{equation}\label{difflig:y}
y_t(t,\xi)+u(t,y(t,\xi))y_\xi(t,\xi)=u(t,y(t,\xi)).
\end{equation}
Furthermore, recalling \eqref{def:H}, direct computations yield
\begin{equation}\label{difflig:tH}
\tilde H_t(t,\xi)+u(t,y(t,\xi))\tilde H_\xi(t,\xi)=0.
\end{equation}

\subsection{The differential equation satisfied by $U(t,\xi)$}\label{diff:U}
To begin with we have a closer look at the system of differential equations, given by \eqref{difflig:y} and \eqref{difflig:tH}, which reads, using \eqref{def:U}
\begin{subequations}\label{ODEsys:2}
\begin{align}
y_t(t,\xi)+Uy_\xi(t,\xi)&=U(t,\xi),\\
\tilde H_t(t,\xi)+U\tilde H_\xi(t,\xi)&=0.
\end{align}
\end{subequations}
This systems of equations can be solved (uniquely) by the method of characteristics, if the differential equation
\begin{equation}\label{karak}
k_t(t,\zeta)=U(t,k(t,\zeta))
\end{equation} 
has a unique solution and $k_\zeta(t,\dott)$ is strictly positive for all $t\in [0,T]$. According to classical ODE theory, \eqref{karak} has for each fixed $\zeta$ a unique solution if the function $U(t,\xi)=u(t,y(t,\xi))$ is continuous with respect to time and Lipschitz with respect to space. The continuity with respect to time is an immediate consequence of \eqref{cont:u}. To establish the Lipschitz continuity with respect to space is a bit more involved. A closer look at \eqref{def:H} and \eqref{lasr2} reveals that one has 
\begin{equation*}
y(t,\xi)+\sigma F(t,y(t,\xi)-)+(1-\sigma)F(t,y(t,\xi)+)=\xi \quad \text{ for some } \sigma\in [0,1],
\end{equation*}
and 
\begin{equation*}
\tilde H(t,\xi)=\sigma F(t,y(t,\xi)-)+(1-\sigma)F(t,y(t,\xi)+).
\end{equation*}
This means especially, given $\xi\in \Real$, there exist $\xi^-\leq \xi\leq \xi^+$ such that 
\begin{equation*}
y(t,\xi^-)=y(t,\xi)=y(t,\xi^+)
\end{equation*} 
and
\begin{equation*}
\tilde H(t,\xi^-)=F(t,y(t,\xi)-) \quad \text{ and }\quad \tilde H(t,\xi^+)=F(t,y(t,\xi)+).
\end{equation*}
In view of Definition~\ref{def:loes}~(v), we then have for $\xi_1< \xi_2$ such $y(t,\xi_1)\not =y(t,\xi_2)$ that 
\begin{align*}
\vert u(t,y(t,\xi_2))-&u(t,y(t,\xi_1))\vert \\
& = \vert u(t,y(t,\xi_2^-))-u(t,y(t,\xi_1^+))\vert\\
& \leq \big(y(t,\xi_2^-)-y(t,\xi_1^+)\big)^{1/2}\big(F(t,y(t,\xi_2)-)-F(t,y(t,\xi_1)+)\big)^{1/2}\\
& = \big(y(t,\xi_2^-)-y(t,\xi_1^+)\big)^{1/2}\big(\tilde H(t,\xi_2^-)-\tilde H(t,\xi_1^+)\big)^{1/2}\\
& \leq \vert \xi_2^ --\xi_1^+\vert \leq \vert \xi_2-\xi_1\vert,
\end{align*}
since both $y(t,\dott)$ and $\tilde H(t,\dott)$ are Lipschitz continuous  in space with Lipschitz constant at most one.
Thus \eqref{karak} has a unique solution. Furthermore, if $k(0,\zeta)=\zeta$ for all $\zeta\in \Real$ and $\zeta_1<\zeta_2$, we have, as long as the function $k(t,\dott)$ remains non-decreasing that 
\begin{equation*}
-(k(t,\zeta_2)-k(t,\zeta_1))\leq (k(t,\zeta_2)-k(t,\zeta_1))_t\leq k(t,\zeta_2)-k(t,\zeta_1),
\end{equation*}
which yields
\begin{equation*}
(k(0,\zeta_2)-k(0,\zeta_1))e^{-t}\leq k(t,\zeta_2)-k(t,\zeta_1)\leq (k(0,\zeta_2)-k(0,\zeta_1))e^t.
\end{equation*}
Thus $k(t,\dott)$ not only remains strictly increasing, it is also Lipschitz continuous with Lipschitz constant $e^t$ and hence according to Rademacher's theorem differentiable almost everywhere.
In particular, one has that 
\begin{equation}\label{bder:k}
e^{-t}\leq k_\zeta(t,\zeta)\leq e^t.
\end{equation}

Introducing 
\begin{equation*}
\bar y(t,\zeta)=y(t,k(t,\zeta)) \quad \text{ and }\quad \bar H(t,\zeta)=\tilde H(t,k(t,\zeta)),
\end{equation*}
we have from \eqref{ODEsys:2}
\begin{subequations}\label{syshr}
\begin{align}\label{syshr1}
\bar y_t(t,\zeta)&=u(t,\bar y(t,\zeta)),\\
\bar H_t(t,\zeta)&=0.
\end{align}
\end{subequations}
In particular, one has 
\begin{equation*}
\bar H(t,\zeta)=\bar H(s,\zeta)
\end{equation*}
and \eqref{def:H} turns into
\begin{equation}\label{rel:bar}
\bar y(t,\zeta)+\bar H(t,\zeta)=k(t,\zeta).
\end{equation}
Furthermore, note that $\bar y(t,\zeta)$ is a characteristic due to \eqref{syshr1}. Introducing 
\begin{equation*}
\bar U(t,\zeta)=u(t,\bar y(t,\zeta))=u(t,y(t,k(t,\zeta)))=U(t,k(t,\zeta)),
\end{equation*}
the system \eqref{syshr} reads
\begin{align*}
\bar y_t(t,\zeta)&=\bar U(t,\zeta),\\
\bar H_t(t,\zeta)&=0.
\end{align*}
The above system can be extended to the system \eqref{ODEsys}, which has been introduced in \cite{BHR} and which describes conservative solutions in the sense of Definition~\ref{def:loes}, if we can show that 
\begin{equation}\label{missing:difflig}
\bar U_t(t,\zeta)=\frac12 \left(\bar H(t,\zeta)-\frac12 C\right).
\end{equation}
As an immediate consequence, one then obtains the uniqueness of global weak conservative solutions.

The proof of \eqref{missing:difflig} is based on an idea that has been used in \cite{BCZ}. According to the definition of a weak solution, one has for all $\phi\in \C_c^\infty([0,\infty)\times\Real)$ that 
\begin{align*}
\int_s^t \int_\Real (u\phi_t+\frac12 u^2\phi_x+\frac12(F-\frac12 &C)\phi) (\tau,x)dx d\tau \\
&=\int_\Real u\phi(t,x)dx -\int_\Real u\phi(s,x)dx, \quad s<t.
\end{align*}
Note that in the above equality, one can replace $\phi\in \C_c^\infty([0,\infty)\times \Real)$ by $ \phi(t,x)$ such that $\phi(t,\dott) \in \C_c^\infty(\Real)$ for all $t$ and $\phi(\dott,x)\in \C^1(\Real)$ for all $x$. 

To prove that $\bar U(t,\zeta)$ is Lipschitz we have to make a special choice of $\phi(t,x)$. Let 
\begin{equation*}
\phi_\varepsilon(t,x)=\frac{1}{\varepsilon}\psi\Big(\frac{\bar y(t,\zeta)-x}{\varepsilon}\Big)
\end{equation*}
where $\psi$ is a standard Friedrichs mollifier. Our choice is motivated by the following observation,
\begin{equation*}
\lim_{\varepsilon\to 0} \int_\Real u\phi_\varepsilon(t,x)dx =u(t,\bar y(t,\zeta)),
\end{equation*}
and hence 
\begin{equation*}
u(t,\bar y(t,\zeta))-u(s,\bar y(s,\zeta))=\lim_{\varepsilon\to 0} \int_\Real (u\phi_\varepsilon(t,x)-u\phi_\varepsilon(s,x))dx.
\end{equation*}
Direct calculations then yield
\begin{align*}
\int_\Real \big(u\phi_{\varepsilon,t}+\frac12 u^2\phi_{\varepsilon,x}& +\frac12(F-\frac12 C)\phi_\varepsilon\big)(\tau,x)dx\\
&= \frac12u(\tau,\bar y(\tau,\zeta))^2 \int_\Real \frac{1}{\varepsilon^2}\psi'\Big(\frac{\bar y(\tau,\zeta)-x}{\varepsilon}\Big)dx\\
& \quad -\frac12 \int_\Real \big(u(\tau,x)-u(\tau,\bar y(\tau,\zeta))\big)^2\frac{1}{\varepsilon^2} \psi'\Big(\frac{\bar y(\tau,\zeta)-x}{\varepsilon}\Big)dx\\
& \quad +\frac12 \int_\Real \big(F(\tau,x)-F(\tau,\bar y(\tau,\zeta))\big)\frac1{\varepsilon}\psi\Big(\frac{\bar y(\tau,\zeta)-x}{\varepsilon}\Big)dx \\
& \quad + \frac12 F(\tau,\bar y(\tau,\zeta))-\frac14 C\\
&=  -\frac12 \int_\Real \big(u(\tau,x)-u(\tau,\bar y(\tau,\zeta))\big)^2\frac{1}{\varepsilon^2} \psi'\Big(\frac{\bar y(\tau,\zeta)-x}{\varepsilon}\Big)dx\\
& \quad +\frac12 \int_\Real \big(F(\tau,x)-F(\tau,\bar y(\tau,\zeta))\big)\frac1{\varepsilon}\psi\Big(\frac{\bar y(\tau,\zeta)-x}{\varepsilon}\Big)dx \\
& \quad + \frac12 F(\tau,\bar y(\tau,\zeta))-\frac14 C.
\end{align*}
Introduce the function
\begin{equation*}
F_{\rm ac}(\tau,x)=\int_{-\infty}^x u_x^2(\tau,y)dy
\end{equation*}
and note that $F_{\rm ac}(\tau,\dott)$ is absolutely continuous. Moreover, one has 
\begin{equation*}
\vert u(\tau,x)-u(\tau,y)\vert \leq \vert x-y\vert^{1/2}\vert F_{\rm ac}(\tau,x)-F_{\rm ac}(\tau,y)\vert^{1/2}
\end{equation*}
and 
\begin{align*}
\vert \int_\Real \big(u(\tau,x)-u(\tau,\bar y(\tau,\zeta))\big)^2& \frac{1}{\varepsilon^2}\psi'\Big(\frac{\bar y(\tau,\zeta)-x}{\varepsilon}\Big) dx\vert \\
& =\frac{1}{\varepsilon} \vert \int_{-1}^{1}\big(u(\tau,\bar y(\tau,\zeta)-\varepsilon\eta)-u(\tau,\bar y(\tau,\zeta))\big)^2 \psi'(\eta)d\eta\vert\\
& \leq  \vert F_{\rm ac}(\tau, \bar y(\tau,\zeta)+\varepsilon)-F_{\rm ac}(\tau,\bar y(\tau,\zeta)-\varepsilon)\vert,
\end{align*}
which implies
\begin{equation*}
\lim_{\varepsilon\to 0} \vert \int_\Real \big(u(\tau,x)-u(\tau,\bar y(\tau,\zeta))\big)^2 \frac{1}{\varepsilon^2}\psi'\Big(\frac{\bar y(\tau,\zeta)-x}{\varepsilon}\Big) dx\vert =0.
\end{equation*}
For the last two terms, observe that for every $\tau\not \in \N$, i.e., for almost every $\tau\in [0,T]$ one has 
\begin{equation*}
F(\tau,\bar y(\tau,\zeta))=F_{\rm ac}(\tau,\bar y(\tau,\zeta))=\bar H(\tau,\zeta)
\end{equation*}
and 
\begin{align*}
&\abs{ \int_\Real \big(F(\tau,x)-F(\tau,\bar y(\tau,\zeta))\big)\frac1{\varepsilon}\psi\Big(\frac{\bar y(\tau,\zeta)-x}{\varepsilon}\Big)dx} \\
 &\qquad=\abs{\int_{-1}^1 \big(F(\tau,\bar y(\tau,\zeta)-\varepsilon\eta)-F(\tau,\bar y(\tau,\zeta))\big)\psi(\eta)d\eta} \\
 &\qquad=\abs{\int_{-1}^1 \big(F_{\rm ac}(\tau,\bar y(\tau,\zeta)-\varepsilon\eta)-F_{\rm ac}(\tau,\bar y(\tau,\zeta))\big)\psi(\eta)d\eta}\\
 &\qquad\le\abs{F_{\rm ac}(\tau,\bar y(\tau,\zeta)-\varepsilon)-F_{\rm ac}(\tau,\bar y(\tau,\zeta))} \to 0, \quad \text{as $\varepsilon\to 0$.}
\end{align*}

Thus, the dominated convergence theorem yields  
\begin{align*}
\bar U(t,\zeta)-\bar U(s,\zeta)&=\int_{s}^t \frac12 \bar H(\tau,\zeta)d\tau -\frac14 C(t-s)\\
&= \frac12 \bar H(t,\zeta)(t-s)-\frac14 C(t-s)=\frac12 \bar H(s,\zeta)(t-s)-\frac14 C(t-s).
\end{align*}
In particular,
\begin{equation*}
\vert \bar U(t,\zeta)-\bar U(s,\zeta)\vert \leq \frac12 C\abs{t-s},
\end{equation*}
i.e., $\bar U(t,\zeta)$ is Lipschitz continuous with respect to time,
and
\begin{equation}\label{difflig:gU}
\bar U_t(t,\zeta)=\frac12 \bar H(t,\zeta)-\frac14C
\end{equation}
for all $(t,\zeta)$ in $[0,T]\times\Real$. Moreover, one has, using \eqref{bder:k} and \eqref{rel:bar}
\begin{align}
\vert \bar U(t,\zeta_1)-\bar U(t,\zeta_2)\vert& =\vert u(t, \bar y(t,\zeta_1))-u(t,\bar y(t,\zeta_2))\vert \nn \\ \nn
& \leq \vert \bar y(t,\zeta_1)-\bar y(t,\zeta_2)\vert^{1/2}\vert F_{\rm ac}(t, \bar y(t,\zeta_1))-F_{\rm ac}(t,\bar y(t,\zeta_2))\vert^{1/2}\\ \nn
&\le \abs{\bar y(t,\zeta_1)+F_{\rm ac}(t, \bar y(t,\zeta_1))-\bar y(t,\zeta_2)-F_{\rm ac}(t,\bar y(t,\zeta_2))}\\ \nn
&=\vert k(t,\zeta_1)-k(t,\zeta_2)\vert\\ \nn
& \leq \norm{k_\zeta(t,\dott)}_{L^\infty} \vert \zeta_1-\zeta_2\vert\\ 
& \leq e^t \vert \zeta_1-\zeta_2\vert,  \label{LIPU}
\end{align}
i.e., $\bar U(t,\zeta)$ is Lipschitz continuous with respect to space.

The final step is to derive the differential equation for $U(t,\xi)$ from \eqref{difflig:gU}. Recall that we have the relation
\begin{equation}\label{rel:UbU2}
\bar U(t,\zeta)=U(t, k(t,\zeta)).
\end{equation}
Since $k(t,\dott)$ is continuous and strictly increasing, there exists a continuous and strictly increasing function $o(t,\dott)$ such that 
\begin{equation}\label{inv:char2}
o(t,k(t,\zeta))=\zeta \quad \text{for all } (t,\zeta)\in [0,T]\times \Real.
\end{equation}
Now, given $\xi\in \Real$, there exists a unique $\zeta\in \Real$ such that  $k(t,\zeta)=\xi$, and thus
\begin{equation*}
o(t,\xi)-o(s,\xi)=o(t,k(t,\zeta))-o(s,k(t,\zeta))=o(s,k(s,\zeta))-o(s,k(t,\zeta)).
\end{equation*}
By definition, we have that $o(t,\dott)$ is continuous and strictly increasing and hence differentiable almost everywhere. Furthermore, one has that $o(t,\dott)$ satisfies
\begin{equation}\label{der:o}
e^{-t}\leq o_\xi(t,\xi)=\frac{1}{k_\zeta(t,o(t,\xi))}\leq e^t,
\end{equation}
by \eqref{bder:k}, which yields
\begin{equation*}
\vert o(t,\xi)-o(s,\xi)\vert \leq e^T \vert k(s,\zeta)-k(t,\zeta)\vert \leq e^T \int_{\min(s,t)}^{\max(s,t)} \vert u(\tau, y(\tau,k(\tau,\zeta)))\vert d\tau.
\end{equation*}
Since $u(t,x)$ can be uniformly bounded on $[0,T]\times \Real$, it follows that $o(t,\xi)$ is Lipschitz with respect to both space and time on $[0,T]\times \Real$ and by Rademacher's theorem differentiable almost everywhere. Moreover, direct calculations yield
\begin{equation}\label{difflig:o}
o_t(t,\xi)+Uo_\xi(t,\xi)=0 \quad \text{ almost everywhere.}
\end{equation}

For every $t\in [0,T]$, $o(t,\dott)$ is strictly increasing and continuous, and combining \eqref{rel:UbU2} and \eqref{inv:char2}, one has 
\begin{equation*}
U(t,\xi)=\bar U(t, o(t,\xi)) \quad \text{ for all } (t,\xi)\in [0,T]\times \Real.
\end{equation*}
Furthermore, both $\bar U$ and $o$ are Lipschitz with respect to both space and time on $[0,T]\times \Real$, and hence both $U$ and $\bar U$ are Lipschitz and hence differentiable almost everywhere on $[0,T]\times\Real$.
Using \eqref{difflig:gU} and \eqref{difflig:o} we finally end up with
\begin{align}\label{difflig:U}
U_t(t,\xi)& =\bar U_t(t, o(t,\xi))+\bar U_\zeta(t,o(t,\xi))o_t(t,\xi)\\ \nn
& = \frac12 (\tilde H(t,\xi)-\frac12C)-UU_\xi(t,\xi).
\end{align}

Following closely \cite{BHR} one can show that for each time $t$ the triplet $(y,U,\tilde H)$ belongs to $\F_0$. Furthermore, the system given by \eqref{ODEsys:2} and \eqref{difflig:U} can be uniquely solved in $\F_0$ with the help of the method of characteristics. Thus we have shown the following result.

\begin{theorem}\label{thm:main_old}
Given a weak conservative solution $(u,\mu)$ to the Hunter--Saxton equation. Then the functions $y(t,\xi)$, $U(t,\xi)$, and $\tilde H(t,\xi)$ defined in \eqref{def:y}, \eqref{def:H}, and \eqref{def:U}, respectively, satisfy the following system of differential equations
\begin{subequations} \label{eq:new:LAG}
\begin{align}
y_t(t,\xi)+Uy_\xi(t,\xi)&=U(t,\xi),\\
\tilde H_t(t,\xi)+U\tilde H_\xi(t,\xi)&=0,\\
U_t(t,\xi) +UU_\xi(t,\xi)&= \frac12 (\tilde H(t,\xi)-\frac12C),
\end{align}
\end{subequations}
which can be solved uniquely in $\F_0$ with the help of the method of characteristics. In particular, applying the method of characteristics yields the system of ordinary differential equations \eqref{ODEsys}, which describes the weak conservative solutions constructed in \cite{BHR}. 

\end{theorem}

\begin{remark}\label{rem:relsys}
Note that the above system of differential equations \eqref{eq:new:LAG} is related to the system \eqref{eq:char_intro} by relabeling. Indeed, denote by $(\check y, \check U, \check H)$ the classical Lagrangian solution with initial data in $\F_0$, which solves \eqref{eq:char_intro}, as introduced in \cite{BHR}.

Applying the method of characteristics to \eqref{eq:new:LAG}, will yield a unique solution, since the characteristic equation 
\begin{equation*}
k_t(t,\zeta)=U(t,k(t,\zeta)) \quad \text{ with } \quad k(0,\zeta)=\zeta
\end{equation*}
can be solved uniquely and $k(t,\dott)$ is strictly increasing, see the beginning of this subsection. Thus we are led to investigating the system 
\begin{subequations}\label{eq:new:LAGex}
\begin{align}
k_t(t,\zeta)&=\hat U(t,\zeta),\\
\hat y_t(t,\zeta)&=\hat U(t,\zeta),\\
\hat U_t(t,\zeta)&=\frac12 (\hat H(t,\zeta)-\frac12 C),\\
\hat H_t(t,\zeta)&=0,
\end{align}
\end{subequations}
where $(\hat y, \hat U, \hat H)(t,\zeta)=( y, U, \tilde H)(t,k(t,\zeta))$ with initial data
\begin{equation*}
(\hat y, \hat U, \hat H)(0,\zeta)= (y, U, \tilde H)(0,\zeta) =(\check y,\check U, \check H)(0,\zeta). 
\end{equation*}
While the last three equations coincide with the ones in \eqref{eq:char_intro}, which suggests 
\begin{equation}\label{con:rel}
(\hat y(t,\zeta), \hat U(t,\zeta), \hat H(t,\zeta))=(\check y(t,\zeta), \check U(t,\zeta), \check H(t,\zeta)),
\end{equation}
 the role of the first equation has to be clarified to prove \eqref{con:rel}. By construction one has that 
 \begin{equation*}
 (y+\tilde H)(t,\xi)=\xi \quad \text{ for all } (t,\xi),
 \end{equation*}
 which implies that 
 \begin{equation*}
 k(t,\zeta)=(y+\tilde H)(t,k(t,\zeta))=(\hat y+\hat H)(t,\zeta)
 \end{equation*}
 and in particular, $k(t,\dott)$ is the relabeling function connecting $(y, U, \tilde H)(t,\dott)\in \F_0$ with $(\hat y, \hat U, \hat H)(t,\dott)\in \F$ for every $t\in \Real$. Furthermore, since $k(t,\zeta)$ can be recovered using $\hat y(t,\zeta)$ and $\hat H(t,\zeta)$, 
 the system \eqref{eq:new:LAGex} can be reduced to \eqref{eq:char_intro} and hence one obtains the weak conservative solutions constructed in \cite{BHR}. 
\end{remark}

We have proved the following theorem. 
\begin{theorem}\label{thm:main2}
For any initial data $(u_0, \mu_0)\in \D$ the Hunter--Saxton equation has a unique global conservative weak solution $(u,\mu)\in \D$ in the sense of Definition~\ref{def:loes}.
\end{theorem}

\section{Introduction of an auxiliary function}\label{auxiliary}

As a preparation for rewriting the Hunter--Saxton equation in a set of coordinates, which shares the essential features with the Lagrangian coordinates, while at the same time avoiding equivalence classes, we introduce an, at the moment, auxiliary function $p(t,x)$.

According to the definition of a weak solution, one has for all $\phi\in\C_c^\infty([0,\infty)\times \Real)$ that
\begin{equation} \label{weak2} 
\int_s^t\int_\Real (\phi_t+u\phi_x)(\tau,x) d\mu(\tau,x)d\tau =\int_\Real \phi(t,x)d\mu(t,x)-\int_\Real\phi(s,x)d\mu(s,x).
\end{equation}
Furthermore, recall that $F(t,x)=\mu(t,(-\infty,x))$. 
If we would use the change of variables from \cite{CGH}, which is based on the pseudo inverse of $F$, one difficulty turns up immediately. The function $F(t,\dott)$, might have intervals, where it is constant and that would especially mean that its inverse would have jumps. The classical method of characteristics implies that these intervals, where $F(t,\dott)$ is constant, will change their position (i.e., they move to the right or to the left), but their length remains unchanged. This would imply that one has to deal with jumps in the inverse, and hence the involved change of variables does not simplify the problem we are interested in. Therefore, a change of variables, for proving the existence and uniqueness of conservative solutions of the HS equation, while at the same time avoiding equivalence classes, should not be based on the inverse of $F$, but the inverse of a strictly increasing and bounded function. This is where $p(t,x)$ will come into the play. 

Let $n\in \mathbb{N}$. Introduce the non-negative function 
\begin{equation}\label{defp}
p_n(t,x)=\big(K_n\star \mu(t)\big)(x)= \int_{\Real}\frac{1}{(1+(x-y)^2)^n}d\mu(t,y),
\end{equation}
with
\begin{equation}\label{eq:K}
K_n(x)=\frac{1}{(1+x^2)^n}.
\end{equation}
We note the following elementary result.
\begin{lemma}\label{lemma:Kp}
Let $n\in \mathbb{N}$ and $K_n$ be given by \eqref{eq:K}.   Then $K_n\in L^p(\Real)\cap \C^\infty_0(\Real)$ for all $1\leq p\leq\infty$.
Note that
\begin{equation*}
\abs{K_n'}\le nK_n, \quad 0< K_n\le K_1, \quad \norm{K_1}_{L^1(\Real)}=\pi.
\end{equation*}
The $j$th derivative $K_n^{(j)}$ satisfies
\begin{equation*}
K_n^{(j)}(x)=\frac{q_{n,j}(x)}{(1+x^2)^{2^{j}n}},
\end{equation*}
where $q_{n,j}(x)$ denotes a polynomial with degree not exceeding $2^{j+1}n-2n-j$.

Thus $p_n(t,\dott)=K_n\star \mu(t)\in \C^\infty_0(\Real)\cap H^s(\Real)$ for any integer $s\geq 1$.
\end{lemma}

\begin{proof}
Clearly $K_n^{(j)}=q_{n,j}r^{-2^j n}$ for some polynomial $q_{n,j}$ where $r=1+x^2$.  It remains to estimate the degree of $q_{n,j}$. 
We prove this by induction. Observe first that $\deg(q_{n,0})=0$ and $\deg(r^n)=2n$.
We find in general
\begin{equation*}
q_{n,j+1}=q_{n,j}' r^{2^{j}n}-2^{j}nq_{n,j}r^{2^{j}n-1}r', \quad j\in \mathbb{N}\cup\{0\}.
\end{equation*}
Assume that $\deg(q_{n,j})\le 2^{j+1}n-2n-j$ for some $j$. Then we note
\begin{align*}
\deg(q_{n,j+1})&\leq \max(\deg(q_{n,j}' r^{2^{j}n}),\deg(q_{n,j}r^{2^{j}n-1}r'))\\
&\leq \max(2^{j+1}n-2n-j-1+2^{j+1}n, 2^{j+1}n-2n-j+2(2^{j}n-1)+1)\\
& = 2^{j+2}n-2n-j-1.
\end{align*}

Furthermore, one has
\begin{equation*}
\deg(q_{n,j})\leq \deg(r^{2^j n})-j,
\end{equation*}
and therefore $p_n(t,x)$ not only belongs to $\C^\infty_0(\Real)$, but $p_n(t,\dott)$ in $H^s(\Real)$ for any integer $s\geq 1$.
\end{proof}

\begin{remark}\label{rem:n}
We will drop the subscript $n$ in the notation, and the value of $n$ will only be fixed later. Thus we write
\begin{equation*}
K(x)=\frac{1}{(1+x^2)^n}, \quad p(t,x)=\big(K\star \mu(t)\big)(x).
\end{equation*}
\end{remark}
What can we say about the time evolution of the function $p(t,x)$? We have two main ingredients: On the one hand the definition of $p(t,x)$. Consider $\psi\in \C_c^\infty(\Real)$ such that $\supp(\psi)\subset(-1,T+1)$, $\psi\equiv 1$ on $[0,T]$, and $0\leq \psi\leq 1 $. If we define $\phi(t,x)=\psi(t)K(x)$, then $\phi(t,x)\in \C^\infty([0,\infty)\times \Real)$ and it can be approximated by admissible test functions. On the other hand we have the definition of a weak solution \eqref{weak2}, which implies 
\begin{align} \nn
p(s,x)-p(t,x)&=\int_\Real \phi(s,x-y)d\mu(s,y)-\int_\Real\phi(t,x-y)d\mu(t,y)\\ \nn
& =-\int_t^s\int_\Real u(\tau,y)\phi_{x} (\tau,x-y) d\mu(\tau,y)d\tau \\ \label{LIPP}
& =-\int_t^s \int_\Real u(\tau,y) K'(x-y)d\mu(\tau,y)d\tau , \quad 0\le t<s\le T.
\end{align}
Note that the above equation implies that the function $p(t,x)$ is locally Lipschitz continuous with respect to time, if
\begin{equation*}
\int_\Real u(\tau,y)K'(x-y)d\mu(\tau,y)
\end{equation*}
can be uniformly bounded. Therefore observe that
\begin{equation*}
\vert \int_\Real u(\tau,y) K'(x-y)d\mu(\tau,y) \vert \leq n\norm{u(\tau,\dott)}_{L^\infty}p(\tau,x)\leq n \norm{u(\tau,\dott)}_{L^\infty}C.
\end{equation*} 
Furthermore, since $u(t,x)$ is a weak solution to 
\begin{equation*}
u_t+uu_x=\frac12 (F-\frac12 C),
\end{equation*} 
 we have that 
\begin{equation*}
\norm{u(t,\dott)}_{L^\infty}\leq \norm{u(0,\dott)}_{L^\infty}+\frac14 Ct \quad \text{ for all } \quad t\geq 0.
\end{equation*}
Thus Rademacher's theorem yields that the function
\begin{equation*}
t\mapsto p(t,x)=\int_\Real K(x-y)d\mu(t,y)
\end{equation*}
from $[0,T]$ to $\Real$ is Lipschitz continuous (or locally Lipschitz continuous on $[0,\infty)\times \Real$) 
and hence differentiable almost everywhere on the finite interval $[0,T]$. In fact, one has from \eqref{LIPP}
and Definition~\ref{def:loes} that 
\begin{equation}\label{evolp}
p_{t}(t,x)=-\int_\Real u(t,y)K'(x-y)d\mu(t,y) 
\end{equation}
for all $(t,x)\in [0,T]\times\Real$.
Indeed, use \eqref{LIPP} as a starting point, which reads in Lagrangian coordinates $(\check y, \check U, \check H)$,
\begin{equation}
p(s,x)-p(t,x)= -\int_t^s \int_\Real \check U(\tau,\xi)K'(x-\check y(\tau,\xi))\check H_\xi(\tau,\xi)d\xi d\tau.
\end{equation}
Recalling \eqref{ODEsys}, direct calculations yield
\begin{align*}
p(s,x)-p(t,x)&=-\int_t^s \int_\Real \check U(t,\xi)K'(x-\check y(t,\xi))\check H_\xi(t,\xi)d\xi d\tau+ \mathcal{O}( (t-s)^2)\\
& = -\int_\Real u(t,y)K'(x-y)d\mu(t,y)(s-t)+\mathcal{O}((t-s)^2)
\end{align*}
which implies 
\begin{equation}\label{improved:pt}
\lim_{s\to t} \frac{p(s,x)-p(t,x)}{s-t}=-\int_{\Real} u(t,y)K'(x-y)d\mu(t,y) 
\end{equation}
for all $(t,x)\in [0,T]\times \Real$.

Note that combining  \eqref{evolp}, \eqref{normu}, and Lemma~\ref{lemma:Kp}, one has 
\begin{equation}\label{eq:p_t}
\vert p_t(t,x)\vert \leq \norm{u(t,\dott)}_{L^\infty}n \norm{p(t,\dott)}_{L^\infty}\leq (\norm{u(0,\dott)}_{L^\infty}+\frac14 CT)nC
\end{equation}

For later use, note that, we have
\begin{equation}\label{Lup}
0\leq p(t,x) \leq \int_\Real d\mu(t,y)=C
\end{equation}
and
\begin{equation}\label{L1p}
\int_\Real p(t,x)dx =\int_\Real \int_\Real K(x-y)d\mu(t,y)dx=\int_\Real \int_\Real K(x-y)dx\, d\mu(t,y)=B,
\end{equation}
where $B$ is independent of time. A closer look reveals that 
\begin{equation*}
\int_\Real p(t,x)dx\leq \int_\Real \int_\Real \frac1{1+(x-y)^2} d\mu(t,y)dx=\pi C.
\end{equation*}
Furthermore, Lemma~\ref{lemma:Kp} implies that $p_x(t,\cdot)\in \C_0^\infty(\Real)$ and 
\begin{equation}\label{eq:p_x}
\vert p_x(t,x)\vert \leq np(t,x).
\end{equation}

Finally, note the following useful expression
\begin{align}
\int_{-\infty}^x p_{t}(t,z)dz&=-\int_{-\infty}^x\int_\Real u(t,y)K'(z-y)d\mu(t,y)dz  \notag\\
&=-\int_\Real \int_{-\infty}^x u(t,y)K'(z-y)dz\, d\mu(t,y) \notag\\
&=- \int_\Real u(t,y)K(x-y)d\mu(t,y). \label{evolp_int}
\end{align}
Thus the function
\begin{equation}\label{def:g}
g(t,\xi)=\int_{-\infty}^{y(t,\xi)} p(t,z)dz
\end{equation}
is differentiable almost everywhere on $[0,T]\times\Real$ and
satisfies 
\begin{align}
g_{\xi}(t,\xi)&=p(t,y(t,\xi))y_\xi(t,\xi),  \label{eq:g_xi}\\
g_{t}(t,\xi)&=\int_{-\infty}^{y(t,\xi)}p_{t}(t,z)dz+ p(t,y(t,\xi))y_t(t,\xi) \nn \\
&= pu(t,y(t,\xi))\tilde H_\xi(t,\xi)-\int_\Real u(t,z)K(y(t,\xi)-z)d\mu(t,z). \label{eq:g_tid}
\end{align}
Here we used, in the last step, \eqref{def:H} and \eqref{difflig:y}. Note, that $g(t,\xi)$ is not only differentiable on $[0,T]\times \Real$, but even Lipschitz continuous.

\section{Uniqueness in a new set of coordinates}

In this section we will rewrite the Hunter--Saxton equation in a set of coordinates, which shares the essential features with the Lagrangian coordinates, while at the same time avoids equivalence classes. However, there is a price to pay: we have to impose an additional moment condition. 

Given $\gamma>2$ and a weak conservative solution $(u,\mu)$ in the sense of Definition~\ref{def:loes}, such that 
\begin{equation}\label{cond:extramu}
\int_\Real (1+ \vert x\vert^\gamma) d\mu(0,x)<\infty.
\end{equation}
Using the reformulation of the Hunter--Saxton equation in Lagrangian coordinates, whose time evolution is given by \eqref{ODEsys} and that $f_\gamma(x)=\vert x\vert^\gamma$ is convex for $\gamma>2$, it follows that  
\begin{equation}\label{cond:extramug}
\int_\Real (1+\vert x\vert ^\gamma)d\mu(t,x)<\infty \quad \text{ for all } t\in \Real.
\end{equation}
We can see this as follows.
\begin{align*}
\int_\Real (1+\vert x\vert ^\gamma)d\mu(t,x)&= \int_\Real (1+\vert \check y(t,\xi)\vert ^\gamma)\check H_\xi(t,\xi) d\xi\\
&= \int_\Real (1+\vert \check y(t,\xi)\vert ^\gamma)\check H_\xi(0,\xi) d\xi \\
&\le 2^{\gamma-1}  \int_\Real (1+\vert \check y(0,\xi)\vert ^\gamma)\check H_\xi(0,\xi) d\xi \\
&\quad +2^{\gamma-1}(\| u_0\|_\infty+\frac{CT}4)^\gamma t^\gamma  
\int_\Real \check H_\xi(0,\xi) d\xi\\
&\le \tilde C\Big(\int_\Real (1+\vert x\vert ^\gamma)d\mu(0,x)+ t^\gamma \Big),
\end{align*}
for a constant $\tilde C$, using  \eqref{ODEsys3A}, \eqref{FD2} and the estimate
\begin{equation*}
\abs{\check y(t,\xi)}\le\abs{\check y(0,\xi)} +\int_0^t \abs{\check y_t(s,\xi)}ds \le \abs{\check y(0,\xi)} 
+(\| u_0\|_\infty+\frac{CT}4)t.
\end{equation*}
Moreover, recalling the definition of $p(t,x)$, cf.~\eqref{defp}, and introducing the non-negative measure 
\begin{equation}\label{def:nu_def2}
d\nu(t,x)=p(t,x)dx+d\mu(t,x),
\end{equation}
we have that \eqref{cond:extramug} implies for any $n\geq \gamma$,
\begin{equation}\label{cond:extranu}
\int_\Real (1+\vert x\vert^\gamma) d\nu(t,x)<\infty \quad \text{ for all } t\in \Real.
\end{equation}
For  details we refer to Lemma~\ref{lem:momp}.

Furthermore, let
\begin{equation}\label{eq:nu_def}
G(t,x)=\nu(t, (-\infty,x))=\int_{-\infty}^x p(t,y) dy +F(t,x),
\end{equation}
where $F(t,x)$ is given by \eqref{def:F}. Then for all $t\in\Real$ the function $G(t,\dott)$ is strictly increasing and satisfies
\begin{equation}\label{eq:nuasymp}
\lim_{x\to -\infty} G(t,x)=0 \quad \text{ and }\quad \lim_{x\to \infty} G(t,x)=B+C.
\end{equation}
Last but not least introduce $\chi(t,\dott)$ the (pseudo) inverse of $G(t,\dott)$, i.e., 
\begin{equation}\label{eq:chi_defi}
\chi(t,\eta)=\sup\{x\mid G(t,x)<\eta\}.
\end{equation}
Then $\chi(t,\dott)\colon[0,B+C]\to \Real$ is strictly increasing for every $t\not\in \mathcal{N}$. Furthermore, since the function $G(t,\dott)$ is of bounded variation, it can have at most countably many jumps, which implies that $\chi(t,\dott)$ can have at most countably many intervals where it is constant. On the other hand, the function $\chi(t,\dott)$ has no jumps since the function $G(t,\dott)$, in contrast to $F(t,\dott)$, is strictly increasing.

Our first goal is to show that 
\begin{equation}\label{rep:chi}
\chi(t,\eta)=y(t,\ell(t,\eta)),
\end{equation}
where $y(t,\xi)$ is given by \eqref{def:y} and $\ell(t,\dott)\colon[0,B+C]\to \Real$ is a strictly increasing function to be determined next. Note that combining \eqref{def:g} and \eqref{eq:nu_def} one has
\begin{equation*}
G(t,y(t,\xi)\pm)=g(t,\xi)+F(t,y(t,\xi)\pm),
\end{equation*}
and therefore \eqref{lasr} can be rewritten as
\begin{equation*}
G(t,y(t,\xi)-)\leq g(t,\xi)+\tilde H(t,\xi)\leq G(t,y(t,\xi)+).
\end{equation*}
Introducing (cf.~\eqref{def:H} and \eqref{def:g})
\begin{equation}\label{def:H2}
H(t,\xi)=g(t,\xi)+\tilde H(t,\xi),
\end{equation}
we end up with 
\begin{equation}\label{lasrr}
G(t,y(t,\xi)-)\leq H(t,\xi)\leq G(t,y(t,\xi)+) \quad \text{ for all } \xi \in \Real. 
\end{equation}
Since $H(t,\dott)\colon\Real\to [0, B+C]$ is strictly increasing (as $\tilde H$ is non-decreasing and $g$ is strictly increasing) and continuous, it is invertible with inverse $\ell(t,\dott)$, i.e., 
\begin{equation}\label{def:l}
H(t,\ell(t,\eta))=\eta \quad \text{ for all } (t,\eta).
\end{equation}
This is the function sought in \eqref{rep:chi}. Furthermore,
\begin{equation}\label{interpolation}
G(t, y(t, \ell(t,\eta))-)\leq H(t,\ell(t,\eta))=\eta\leq G(t,y(t,\ell(t,\eta))+).
\end{equation}
Since $y(t,\dott)$ is surjective and non-decreasing, we end up with
\begin{equation*}
\chi(t,\eta)=\sup\{x\mid G(t,x)<\eta\}=y(t,\ell(t,\eta)).
\end{equation*}

\begin{remark} \label{rem:relabel2}
Recall Remark~\ref{rem:relabel} and the notation therein. Assume 
\begin{equation*}
X_1(t,\xi)=X(t,f(t,\xi)) \quad \text{ for all } (t,\xi),
\end{equation*}
where $X(t,\dott)\in \F_0$ and $f(t,\dott)\in \G$, i.e., $X(t, \dott)$ and $X_1(t,\dott)$ belong to the same equivalence class. Then, following the same lines as above and denoting the inverse of $f(t,\dott)$ by $f^{-1}(t,\dott)$, we would have ended up with 
\begin{equation*}
\chi(t,\eta)=y_1(t, \ell_1(t,\eta))=y(t, f(t,f^{-1}(t,\ell(t,\eta))))=y(t,\ell(t,\eta))
\end{equation*}
and 
\begin{equation*}
\ell_1(t,\eta)=f^{-1}(t,\ell(t,\eta)). 
\end{equation*}
Thus we obtain the same function $\chi(t,\eta)$ independent of which representative in the corresponding equivalence class we choose. This is in contrast to some of the following steps, where relabeling will play a crucial role. In particular, it is then of great importance, which representative we choose from the equivalence class.
\end{remark}

\subsection{Differentiability of $\chi(t,\eta)$ with respect to time}
Next, we want to study the time evolution of $\chi(t,\eta)$. On the one hand, we will see that $\chi(t,\eta)$ behaves like a characteristic. On the other hand, we expect $\chi_t(t,\dott)\in L^\gamma([0,B+C])$, since one has 
\begin{equation*}
\int_\Real \vert x\vert^\gamma d\nu(t,x)=\int_0^{B+C} \vert \chi(t,\eta)\vert^\gamma d\eta,
\end{equation*}
by \eqref{cond:extranu} and \eqref{eq:chi_defi}.

To begin with we aim at showing that $\chi(t,\eta)$ is differentiable with respect to time in the following sense: 
We establish that $\chi(\dott,\eta)$ is 
Lipschitz continuous and show that for each $t\in [0,T]$ one has that 
\begin{equation*}
\chi_t(t,\eta)=\lim_{s\to t} \frac{\chi(t,\eta)-\chi(s,\eta)}{t-s}
\end{equation*}
exists for almost every $\eta\in [0,B+C]$. The dominated convergence theorem then implies that $\chi_t(t,\dott)\in L^\gamma([0,B+C])$.

To establish that $\chi(t,\eta)$ is Lipschitz with respect to time, a relabeling argument will be the key. We will show that $(y+H)(t,\dott)$ is a relabeling
function denoted $v(t,\dott)$.
To that end observe that combining \eqref{def:H} and \eqref{def:H2} yields
\begin{equation}\label{rel:1}
y(t,\xi)+H(t,\xi)=\xi+g(t,\xi).
\end{equation}
Introducing the function 
\begin{equation}\label{def:v}
v(t,\xi)=\xi+g(t,\xi),
\end{equation}
where $g(t,\xi)$ is given by \eqref{def:g}, we end up with 
\begin{equation}\label{rel:3}
y(t,\xi)+H(t,\xi)=v(t,\xi).
\end{equation}
Combining \eqref{Lup}, \eqref{L1p}, and $0\leq y_\xi(t,\xi)\leq 1$, we have that $v(t,\dott)$ satisfies all assumptions of Lemma~\ref{lem:rel} and hence $v(t,\dott)\colon \Real\to \Real$ is a relabeling function. Thus $v(t,\dott):\Real\to \Real$ is strictly increasing and continuous, which implies that there exists a unique, strictly increasing and continuous function $w(t,\dott):\Real\to \Real$ such that 
\begin{equation}\label{rel:vw}
v(t,w(t,\xi))=\xi=w(t,v(t,\xi)) \quad \text{ for all } (t,\xi)\in [0,T]\times\Real.
\end{equation}
In particular, one has
\begin{equation*}
y(t,w(t,\xi))+H(t,w(t,\xi))=\xi \quad \text{ for all } (t,\xi)\in [0,T]\times \Real.
\end{equation*}
Introducing
\begin{equation}\label{def:hyH}
\hat y(t,\xi)=y(t,w(t,\xi)) \quad \text{ and }\quad \hat H(t,\xi)=H(t,w(t,\xi)),
\end{equation}
the relation \eqref{rel:3} rewrites as
\begin{equation}\label{rel:2}
\hat y(t,\xi)+\hat H(t,\xi)=\xi.
\end{equation}
Since both $H(t,\dott)$ and $w(t,\dott)$, and hence $\hat H(t,\dott)$, are strictly increasing and continuous, there exists a unique, strictly increasing and continuous function $\hat \ell(t,\dott)\colon[0,B+C]\to \Real$ such that 
\begin{equation}\label{rel:HinvH}
\hat H(t, \hat \ell(t,\eta))=\eta \quad \text{ for all } (t,\eta)\in [0,T]\times [0,B+C].
\end{equation}
Recalling \eqref{rel:vw} and \eqref{def:hyH}, we have that 
\begin{equation*}
H\big(t,w(t,v(t,\ell(t,\eta)))\big)=H(t,\ell(t,\eta))=\eta=\hat H(t,\hat \ell(t,\eta))=H\big(t,w(t,\hat \ell(t,\eta))\big).
\end{equation*} 
Thus, cf.~Remark~\ref{rem:relabel2},
\begin{equation*}
\hat \ell(t,\eta)=v(t,\ell(t,\eta))\quad \text{ for all }(t,\eta)\in [0,T]\times[ 0,B+C].
\end{equation*}
and 
\begin{equation}\label{chi:rel}
\chi(t,\eta)=y(t,\ell(t,\eta))=\hat y(t,\hat \ell(t,\eta)),
\end{equation}
which, together with \eqref{rel:2}, implies
\begin{equation}\label{eq:chi}
\chi(t,\eta)+\eta=\hat y(t,\hat \ell(t,\eta))+\hat H(t,\hat \ell(t,\eta))=\hat \ell(t,\eta)=\ell(t,\eta)+g(t,\ell(t,\eta)).
\end{equation}

An important consequence of the above equality is, that we can choose whether we want to study the differentiability of $\chi(t,\eta)=\hat\ell(t,\eta)-\eta$ or of $\hat \ell(t,\eta)$ with respect to time. Since $\hat \ell(t,\dott)$ is the inverse to $\hat H(t,\dott)$, it seems advantageous to study $\hat \ell(t,\eta)$ in detail. The basis will be a good understanding of the relabeling function $v(t,\eta)$ and its inverse $w(t,\eta)$.

\subsubsection*{The Lipschitz continuity of $w(t,\xi)$}
We proceed by showing that the function $w(t,\xi)$ is Lipschitz continuous,  which then implies that both $\hat y(t,\xi)$ and $\hat H(t,\xi)$ are differentiable almost everywhere. A closer look at \eqref{def:v} reveals that $v(t,\xi)$ is Lipschitz continuous and hence differentiable almost everywhere on $[0,T]\times \Real$. In particular, \eqref{eq:g_xi}, \eqref{Lup}, and $0\leq y_\xi(t,\xi)\leq 1$ yield
\begin{equation}\label{vx}
1\leq v_\xi(t,\xi)=1+p(t,y(t,\xi))y_\xi(t,\xi)\leq 1+C,
\end{equation}
and, using \eqref{def:v} and \eqref{eq:g_tid}, 
\begin{equation}\label{vt}
v_t(t,\xi)=g_t(t,\xi)=pu(t,y(t,\xi))\tilde H_\xi(t,\xi)-\int_\Real u(t,z)K(y(t,\xi)-z)d\mu(t,z),
\end{equation}
which satisfies
\begin{equation*}
\vert v_t(t,\xi)\vert \leq 2\big(\norm{u(0,\dott)}_{L^\infty}+\frac14 Ct\big)C,
\end{equation*}
by applying in addition \eqref{normu} and \eqref{def:H}. 
Now, given $\xi\in  \Real$, there exists a unique $\eta\in \Real$ by \eqref{rel:vw}, such that 
\begin{equation*}
w(t,\xi)=w(t,v(t,\eta))=\eta=w(s,v(s,\eta)).
\end{equation*}
Thus, we can write
\begin{equation}\label{L2}
w(t,\xi)-w(s,\xi)=w(t,v(t,\eta))-w(s,v(t,\eta))=w(s,v(s,\eta))-w(s,v(t,\eta)).
\end{equation}
By definition, we have that $w(t,\dott)$ is continuous and strictly increasing. In particular, one has that $w(t,\dott)$ is differentiable almost everywhere and 
\begin{align*}
w_\xi(t,\xi)=\frac{1}{v_\xi(t,w(t,\xi))}&=\frac{1}{1+p(t,y(t,w(t,\xi)))y_\xi(t,w(t,\xi))}\\
&=\frac{1}{1+p(t,\hat y(t,\xi))y_\xi(t,w(t,\xi))}\leq 1.
\end{align*}
Combined, with \eqref{L2}, this yields
\begin{equation*}
\vert w(t,\xi)-w(s,\xi)\vert \leq \vert v(s,\eta)-v(t,\eta)\vert \leq 2(\norm{u(0,\dott)}_{L^\infty}+\frac14 CT)C\vert t-s\vert.
\end{equation*}
Furthermore, one has 
\begin{equation*}
\vert w(t,\xi_1)-w(t,\xi_2)\vert \leq \vert \xi_1-\xi_2\vert.
\end{equation*}
This finishes the proof of the Lipschitz continuity of $w(t,\xi)$, which implies by Rademacher's theorem that $w(t,\xi)$ is differentiable almost everywhere on $[0,T]\times \Real$. 
Since also $y$, $\tilde H$, and $g$ are Lipschitz continuous on $[0,T]\times \Real$, \eqref{def:H2} and \eqref{def:hyH} imply that $\hat H$ and $\hat y$ are differentiable almost everywhere on $[0,T]\times \Real$.

\subsubsection*{The Lipschitz continuity of $\hat \ell(\dott,\eta)$}
We are now ready to show the Lipschitz continuity of $\hat \ell(\dott,\eta)$, which immediately implies that $\chi(\dott,\eta)$ is Lipschitz continuous by \eqref{eq:chi}. In view of \eqref{rel:HinvH}, we start by having a closer look at $\hat H(t,\xi)$. 
By \eqref{rel:vw} we have that,
\begin{equation}\label{eq:wt}
w_t(t,\xi)=-v_t(t,w(t,\xi))w_\xi(t,\xi)=-g_t(t,w(t,\xi))w_\xi(t,\xi),
\end{equation}
and, using \eqref{def:H2}, \eqref{difflig:tH}, and \eqref{eq:g_xi} that
\begin{align}
H_t(t,\xi)&+u(t,y(t,\xi))H_\xi(t,\xi)\nn \\
&=\tilde H_t(t,\xi)+u(t,y(t,\xi))\tilde H_\xi(t,\xi)+ g_t(t,\xi)+u(t,y(t,\xi))g_\xi(t,\xi)\nn \\
&= g_t(t,\xi)+u(t,y(t,\xi))g_\xi(t,\xi)\nn \\
& = g_t(t,\xi)+up(t,y(t,\xi))y_\xi(t,\xi). \label{eq:h+uh}
\end{align}
Combing the above equations and recalling \eqref{eq:g_tid}, \eqref{def:hyH}, \eqref{rel:2}, and \eqref{vx} we get
\begin{align*}
\hat H_t(t,\xi)&+u(t,\hat y(t,\xi))\hat H_\xi(t,\xi)\\
&=\big(H_t(t,w(t,\xi))+u(t,\hat y(t,\xi))H_\xi(t,w(t,\xi))\big)- u(t,\hat y(t,\xi))H_\xi(t,w(t,\xi))\\
&\quad + H_\xi(t,w(t,\xi))w_t(t,\xi)+ u(t,\hat y(t,\xi))\hat H_\xi(t,\xi)\\
&= g_t(t,w(t,\xi))+up(t,\hat y(t,\xi))y_\xi(t,w(t,\xi))-u(t,\hat y(t,\xi))\hat H_\xi(t,\xi)v_\xi(t,w(t,\xi))\\
& \quad -g_t(t,w(t,\xi))\hat H_\xi(t,\xi)+u(t,\hat y(t,\xi))\hat H_\xi(t,\xi) \\
&= g_t(t,w(t,\xi))\hat y_\xi(t,\xi)+up(t,\hat y(t,\xi))y_\xi(t,w(t,\xi))\\
&\quad-up(t,\hat y(t,\xi))\hat H_\xi(t,\xi) y_\xi(t,w(t,\xi))\\
& = (g_t(t,w(t,\xi))+up(t,\hat y(t,\xi))y_\xi(t,w(t,\xi)))\hat y_\xi(t,\xi)\\
& =\Big(up(t,\hat y(t,\xi))-\int_\Real u(t,z)K(\hat y(t,\xi)-z)d\mu(t,z)\Big)\hat y_\xi(t,\xi).
\end{align*}
If we can show that the term on the right-hand side can be bounded by a multiple of $\hat H_\xi$, then we are led to sub- and supersolutions, which solve a transport equation. 
Recalling \eqref{def:H2}, \eqref{def:g}, and \eqref{normu}, we have that  
\begin{align*}
\vert \Big(up(t,\hat y(t,\xi))& -\int_\Real u(t,z)K(\hat y(t,\xi)-z)d\mu(t,z)\Big)\hat y_\xi(t,\xi)\vert\\
& \leq 2\norm{u(t,\dott)}_{L^\infty}p(t,\hat y(t,\xi))\hat y_\xi(t,\xi)\\
& \leq 2\big(\norm{u(0,\dott)}_{L^\infty}+\frac14CT\big)\hat H_\xi(t,\xi).
\end{align*}
Here we have used that
\begin{align*}
\hat H_\xi(t,\xi)&=H_\xi(t,w(t,\xi))w_\xi(t,\xi)\\
& = \big(\tilde H_\xi(t,w(t,\xi))+p(t,y(t,w(t,\xi)))y_\xi(t,w(t,\xi))\big)w_\xi(t,\xi)\\
&\ge  p(t,\hat y(t,\xi))\hat y_\xi(t,\xi),
\end{align*}
which combines \eqref{def:g}, \eqref{def:H2}, \eqref{def:hyH},  $0\le \tilde H_\xi(t,\xi)\le 1$, and $w_\xi(t,\xi)\ge 0$. 
In particular, one obtains that 
\begin{equation*}
\hat H_t(t,\xi)-A\hat H_\xi(t,\xi)\leq 0\leq \hat H_t(t,\xi)+A\hat H_\xi(t,\xi),
\end{equation*}
where $A=3(\norm{u(0,\dott)}_{L^\infty}+\frac14CT)$.
Applying the method of characteristics, we end up with the following estimate
\begin{equation*}
\hat H(s, x-A\vert t-s\vert )\leq \hat H(t,x)\leq \hat H(s, x+A\vert t-s\vert).
\end{equation*}
Now, following the same lines as before, we have: Given $\eta\in [0,B+C]$, there exists a unique $\hat \eta\in \Real$ such that 
\begin{equation*}
\hat \ell(t,\eta)=\hat \ell(t,\hat H(t,\hat \eta))=\hat \eta=\hat \ell(s,\hat H(s,\hat\eta)).
\end{equation*}
Furthermore, we have
\begin{align*}
\vert \hat \ell(t,\eta)-\hat \ell(s,\eta)\vert &=\vert \hat \ell(t, \hat H(t,\hat\eta))-\hat \ell(s,\hat H(t,\hat\eta))\vert\\
&=\vert \hat \ell(s,\hat H(s,\hat \eta))-\hat \ell(s,\hat H(t,\hat \eta))\vert\\
& \leq \hat \ell(s,\hat H(s,\hat\eta+A\vert t-s\vert))-\hat \ell(s,\hat H(s,\hat \eta-A\vert t-s\vert))
= 2A\vert t-s\vert,
\end{align*}
since $\hat \ell(t,\dott)$ is strictly increasing,  
and (cf.~\eqref{eq:chi})
\begin{equation}\label{Lip:Chi}
\left\vert \frac{\chi(t,\eta)-\chi(s,\eta)}{t-s}\right\vert =\left\vert \frac{\hat \ell(t,\eta)-\hat \ell(s,\eta)}{t-s}\right\vert\leq 2A.
\end{equation}
This finishes the proof of the Lipschitz continuity of $\hat\ell(\dott,\eta)$.
Note that $\chi(t,\dott)\colon[0,B+C]\to \Real$ and hence $2A$ can be seen as a dominating function. This observation is essential since the differential equation for $\chi$ has to be considered in $L^\gamma([0,B+C])$.

\subsubsection*{The time derivative of $\chi_t(t,\eta)$}
It is left to compute, cf.~\eqref{eq:chi}, 
\begin{align}\nn
\chi_t(t,\eta)&=\lim_{s\to t}\frac{\chi(t,\eta)-\chi(s,\eta)}{t-s}\\ \label{derchi}
&=\lim_{s\to t} \frac{\ell(t,\eta)-\ell(s,\eta)}{t-s}+\lim_{s\to t}\frac{g(t,\ell(t,\eta))-g(s,\ell(s,\eta))}{t-s}
\end{align}
for almost every $\eta \in [0,B+C]$. 

\subsubsection*{The time derivative of $\ell(t,\eta)$}
To begin with, we will show that 
\begin{equation*}
\lim_{s\to t}\frac{\ell(t,\eta)-\ell(s,\eta)}{t-s}
\end{equation*}
exists for almost every $\eta\in [0,B+C]$ and compute its value. Since $\ell(t,\dott)$ is the inverse of $H(t,\dott)$, cf.~\eqref{def:l}, we start by having a closer look at $H(t,\xi)$. Recall \eqref{cont:u}, which implies that 
\begin{equation*}
\vert y\big(t,\xi+u(s, y(s,\xi))(t-s)\big)-y(t,\xi+u(t,y(t,\xi))(t-s))\vert \leq D(1+N)^{1/2} \vert t-s\vert^{3/2},
\end{equation*}
and, combined with \eqref{fineestimatey}, that there exists a positive constant $\tilde M$ such that  
\begin{equation}\label{finesty}
\abs{y(t,\xi+u(t,y(t,\xi))(t-s))-y(s,\xi)-u(t,y(t,\xi))(t-s)}\leq \tilde M\vert t-s\vert^{3/2}.
\end{equation}
Applying \eqref{def:H} the above inequality reads
\begin{equation}\label{finesttH}
\abs{\tilde H(s,\xi)-\tilde H\big(t,\xi+u(t,y(t,\xi))(t-s)\big)}\leq \tilde M\vert t-s\vert^{3/2}.
\end{equation}

As the following lemma shows a similar estimate holds for $H(t,\xi)$. The proof relies on a detailed investigation of the function $g(t,\xi)$ and can be 
found in Lemma~\ref{lemma:A1}.

\begin{lemma}\label{lem:appendix} (i):
Let $g(t,\xi)$ be given by \eqref{def:g}. Then there exists a positive constant $\bar M$ such that 
\begin{align}\label{finestg}
-\bar M&\vert t-s\vert^{3/2}   -p\big(t,y(t,\xi+u(t,y(t,\xi))(t-s)))u(t,y(t,\xi)\big)(t-s)\\ \nn
& \quad-\int_\Real u(t,y(t,\eta))K\big(y(t,\xi+u(t,y(t,\xi))(t-s))-y(t,\eta)\big)\tilde H_\xi(t,\eta)d\eta\, (s-t)\\ \nn
& \leq g(s,\xi)-g\big(t,\xi+u(t,y(t,\xi))(t-s)\big)\\ \nn
&\leq \bar M\vert t-s\vert^{3/2} -p\big(t,y(t,\xi+u(t,y(t,\xi))(t-s))\big)u(t,y(t,\xi))(t-s)\\ \nn
& \quad-\int_\Real u(t,y(t,\eta))K\big(y(t,\xi+u(t,y(t,\xi))(t-s))-y(t,\eta)\big)\tilde H_\xi(t,\eta)d\eta\, (s-t).
\end{align}
(ii): Let $H(t,\xi)$ be defined by \eqref{def:H2}.  Then 
\begin{align*}
-(\tilde M&+\bar M)\vert t-s\vert^{3/2}   -p\big(t,y(t,\xi+u(t,y(t,\xi))(t-s))\big)u(t,y(t,\xi))(t-s)\\
& \quad-\int_\Real u(t,y(t,\eta))K\big(y(t,\xi+u(t,y(t,\xi))(t-s))-y(t,\eta)\big)\tilde H_\xi(t,\eta)d\eta\, (s-t)\\
& \leq H(s,\xi)-H\big(t,\xi+u(t,y(t,\xi))(t-s)\big)\\
&\leq (\tilde M+\bar M)\vert t-s\vert^{3/2} -p\big(t,y(t,\xi+u(t,y(t,\xi))(t-s))\big)u(t,y(t,\xi))(t-s)\\
& \quad-\int_\Real u(t,y(t,\eta))K\big(y(t,\xi+u(t,y(t,\xi))(t-s))-y(t,\eta)\big)\tilde H_\xi(t,\eta)d\eta\, (s-t).
\end{align*}
\end{lemma}

Since both $H(t,\dott)$ and its inverse $\ell(t,\dott)$ are strictly increasing and continuous, we have that there exists a unique $\eta(s)$ such that 
\begin{equation}\label{conlts}
\ell(s,\eta)+u\big(t,y(t,\ell(s,\eta))\big)(t-s)=\ell(t,\eta(s)),
\end{equation}
and replacing $\xi$ by $\ell(s,\eta)$ in Lemma~\ref{lem:appendix} (ii), we have
\begin{align}\label{condlts}
-(&\tilde M+\bar M)\vert t-s\vert^{3/2} \\ \nn
&\quad  -p\big(t,y(t,\ell(s,\eta)+u(t,y(t,\ell(s,\eta)))(t-s))\big)u\big(t,y(t,\ell(s,\eta))\big)(t-s)\\ \nn
& \quad-\int_\Real u(t,y(t,\tilde\eta))\\ \nn
&\qquad\qquad\times K\big(y(t,\ell(s,\eta)+u(t,y(t,\ell(s,\eta)))(t-s))-y(t,\tilde \eta)\big)\tilde H_\xi(t,\tilde\eta)d\tilde\eta\, (s-t)\\ \nn
& \leq \eta-\eta(s)\\ \nn
&\leq (\tilde M+\bar M)\vert t-s\vert^{3/2} \\ \nn
&\quad-p\big(t,y(t,\ell(s,\eta)+u(t,y(t,\ell(s,\eta)))(t-s))\big)u\big(t,y(t,\ell(s,\eta))\big)(t-s)\\ \nn
& \quad-\int_\Real u(t,y(t,\tilde \eta))\\ \nn
&\qquad\qquad\times K\big(y(t,\ell(s,\eta)+u(t,y(t,\ell(s,\eta)))(t-s))-y(t,\tilde \eta)\big)\tilde H_\xi(t,\tilde \eta)d\tilde\eta \, (s-t).
\end{align}
Furthermore, note that the above inequality implies that $\eta(s)\to \eta$ as $s\to t$ and hence $\ell(s,\eta)\to \ell(t,\eta)$ as $s\to t$ by \eqref{conlts}.

Thus we have 
\begin{align}\label{tder:l}
\ell_t(t,\eta)&=\lim_{s\to t}\frac{\ell(s,\eta)-\ell(t,\eta)}{s-t}\\ \nn
& =\lim_{s\to t} \left(\frac{\ell(t,\eta(s))-\ell(t,\eta)}{\eta(s)-\eta}\frac{\eta(s)-\eta}{s-t}+u(t,y(t,\ell(s,\eta)))\right)\\ \nn
& =\lim_{s\to t}\frac{\ell(t,\eta(s))-\ell(t,\eta)}{\eta(s)-\eta}\lim_{s\to t}\frac{\eta(s)-\eta}{s-t}+\lim_{s\to t}u(t,y(t,\ell(s,\eta))),
\end{align}
if all the above limits exist.

The first limit is of the form 
\begin{equation*}
\lim_{\tilde \eta\to \eta}\frac{\ell(t,\tilde \eta)-\ell(t,\eta)}{\tilde \eta-\eta}
\end{equation*}
since $\eta(s)\to \eta$ as $s\to t$. 
Moreover, $\ell(t,\dott)$ is strictly increasing and continuous and hence differentiable almost everywhere. Thus, the above limit exists for almost every $\eta\in \Real$ and equals $\ell_\eta(t,\eta)$.

For the second limit, keep in mind that both $u(t,\dott)$ and $p(t,\dott)$ are continuous and that $\ell(s,\eta)\to \ell(t,\eta)$ as $s\to t$. The estimate \eqref{condlts}, then implies that 
\begin{align} \nn
\lim_{s\to t}\frac{\eta(s)-\eta}{s-t}& =-p(t,y(t,\ell(t,\eta)))u(t,y(t,\ell(t,\eta)))\\ \nn
&\quad+\int_\Real u(t,y(t,\xi))K(y(t,\ell(t,\eta))-y(t,\xi))\tilde H_\xi(t,\xi) d\xi\\ \nn
& = -p(t,\chi(t,\eta))u(t,\chi(t,\eta))\\ 
&\quad + \int_0^{B+C} u(t,\chi(t,\xi))K(\chi(t,\eta)-\chi(t,\xi))\big(1-p(t,\chi(t,\xi))\chi_\eta(t,\xi)\big)d\xi.
\end{align}
Here we used \eqref{rep:chi}, \eqref{def:g}, \eqref{def:H2}. Introducing
\begin{equation}\label{newvar}
\U(t,\eta)=u(t,\chi(t,\eta)) \quad \text{ and }\quad \P(t,\eta)=p(t,\chi(t,\eta)),
\end{equation}
we end up with 
\begin{equation}\label{lim:nsn}
\lim_{s\to t}\frac{\eta(s)-\eta}{s-t}=- h(t,\eta),
\end{equation}
where
\begin{equation}\label{def:h}
h(t,\eta)= \U\P(t,\eta)-\int_0^{B+C} \U(t,\tilde\eta)K(\chi(t,\eta)-\chi(t,\tilde\eta))(1-\P\chi_\eta(t,\tilde\eta))d\tilde\eta.
\end{equation}
Using once more that $u(t,\dott)$ is continuous and that $\ell(s,\eta)\to \ell(t,\eta)$ as $s\to t$, we get
\begin{equation}\label{lim:unsn}
\lim_{s\to t} u(t,y(t,\ell(s,\eta)))=u(t,y(t,\ell(t,\eta)))=\U(t,\eta).
\end{equation}

Combining \eqref{tder:l}--\eqref{lim:unsn} we have shown that for a  given $t$, one has for almost every $\eta\in [0,B+C]$ that  
\begin{equation}\label{difflig:l}
\ell_t(t,\eta)+h(t,\eta)\ell_\eta(t,\eta)=\U(t,\eta),
\end{equation}
where $h(t,\eta)$ is given by \eqref{def:h}.  This completes the computation of $\ell_t(t,\eta)$.

It is left to show, cf.~\eqref{derchi}, that 
\begin{equation*}
\lim_{s\to t}\frac{g(t,\ell(t,\eta))-g(s,\ell(s,\eta))}{t-s}
\end{equation*}
exists for almost every $\eta\in [0,B+C]$ and to compute its value. Recall \eqref{finestg} and \eqref{conlts}, which imply
\begin{align*}
g(t,\ell(t,\eta))-&g(s,\ell(s,\eta))\\
& =g(t,\ell(t,\eta))-g(t,\ell(t,\eta(s)))\\
& \quad + g\big(t,\ell(s,\eta)+u(t,y(t,\ell(s,\eta)))(t-s)\big)-g(s,\ell(s,\eta))\\
& = \frac{g(t,\ell(t,\eta))-g(t,\ell(t,\eta(s)))}{\eta-\eta(s)}(\eta-\eta(s))\\
& \quad +p\big(t,y(t,\ell(t,\eta(s)))\big)u\big(t,y(t,\ell(s,\eta))\big)(t-s)\\
&\quad  -\int_\Real u(t,y(t,\tilde\eta))K\big(y(t,\ell(t,\eta(s)))-y(t,\tilde\eta)\big)\tilde H_\xi(t,\tilde\eta) d\tilde\eta\, (t-s)\\
& \quad + \mathcal{O}(\abs{t-s}^{3/2}).
\end{align*}
Using \eqref{lim:nsn} and \eqref{def:h}, we thus end up with
\begin{align}\label{derg}
\lim_{s\to t}\frac{g(t,\ell(t,\eta))-g(s,\ell(s,\eta))}{t-s}&=\big(1-g_\xi(t,\ell(t,\eta))\ell_\eta(t,\eta)\big)h(t,\eta)\\ \nn
&= (1-\P(t,\eta)\chi_\eta(t,\eta))h(t,\eta).
\end{align}
 
Combining \eqref{derchi}, \eqref{difflig:l}, \eqref{derg}, and \eqref{eq:chi}, 
finally yields that for each $t\in [0,T]$
\begin{equation}\label{pointchi}
\chi_t(t,\eta)+h(t,\eta)\chi_\eta(t,\eta)=\U(t,\eta)
\end{equation}
for almost every $\eta\in [0,B+C]$ and the function $h(t,\eta)$ is given by \eqref{def:h}. This completes the computation of $\chi_t(t,\eta)$.

\bigskip
To summarize, we showed with the help of the dominated convergence theorem that 
\begin{align*}
\int_0^{B+C} \chi_t^\gamma(t,\eta)d\eta&=\lim_{s\to t}\int_0^{B+C}\left(\frac{\chi(t,\eta)-\chi(s,\eta)}{t-s}\right)^\gamma d\eta\\
& =\int_0^{B+C}\big( -h(t,\eta)\chi_\eta(t,\eta)+\U(t,\eta) \big)^\gamma d\eta.
\end{align*}
Actually we showed that
\begin{equation*}
\lim_{s\to t}\int_0^{B+C}\left\vert \frac{\chi(t,\eta)-\chi(s,\eta)}{t-s}+h(t,\eta)\chi_\eta(t,\eta)+\U(t,\eta) \right\vert^\gamma d\eta=0,
\end{equation*}
which is a limiting process in $L^\gamma([0,B+C])$. Thus the correct Banach space to work in, is $L^\gamma([0, B+C])$. This might be surprising at first sight since only $\chi_t(t,\dott)$ belongs to $L^\gamma([0,B+C])$ but not $\chi_\eta(t,\dott)$. On the other hand, one has that the function $g(t,\ell(t,\dott))$ is non-decreasing and hence differentiable almost everywhere. Furthermore, \eqref{def:H2} and \eqref{def:l}
imply that $g(t,\ell(t,\dott))$ is Lipschitz continuous with Lipschitz constant at most one and for almost every $\eta\in [0,B+C]$ it follows that 
\begin{equation}\label{bound:Pchin}
 0\leq \frac{d}{d\eta} g(t,\ell(t,\eta))=\P\chi_\eta(t,\eta)\leq 1. 
\end{equation}
Thus 
\begin{equation*}
\int_0^{B+C} (\P\chi_\eta)^\gamma(t,\eta)d\eta
 \leq \int_0^{B+C} \P\chi_\eta(t,\eta)d\eta= \int_\Real p(t,x)dx,
\end{equation*}
using \eqref{bound:Pchin}, 
which is finite, cf.~\eqref{L1p} and which implies that $h\chi_\eta(t,\dott)\in L^\gamma([0,B+C])$.
Thus we will not study \eqref{pointchi} pointwise, but as a differential equation in $L^\gamma([0,B+C])$. 

\begin{theorem}\label{thm:chi}
The function $\chi(t,\dott)$ satisfies for almost all $t$ the following differential equation in $L^\gamma([0,B+C])$
\begin{equation}\label{difflig:chi}
\chi_t(t,\eta)+h(t,\eta)\chi_\eta(t,\eta)=\U(t,\eta),
\end{equation}
where 
\begin{equation}\label{eq:def_h}
h(t,\eta)=\U\P(t,\eta)-\int_0^{B+C} \U(t,\tilde \eta)K(\chi(t,\eta)-\chi(t,\tilde\eta))(1-\P\chi_\eta(t,\tilde\eta))d\tilde\eta.
\end{equation}
Furthermore, the mapping $t\mapsto \chi(t,\dott)$ is continuous from $[0,T]$ into $L^\gamma([0,B+C])$.
\end{theorem}

\begin{remark}
The above computations are based on error estimates. This idea is also used when proving the chain rule in the classical setting for functions of several variables.  
However, one cannot use the chain rule here. Namely, writing 
\begin{equation*}
\chi(t,\eta)=\hat \ell(t,\eta)-\eta=v(t,\ell(t,\eta))-\eta,
\end{equation*}
one is tempted to look at $\chi(t,\eta)$ as a composition of two functions, since $v(t,\xi)$ is differentiable almost everywhere with respect to both time and space. However, $v_t(t,\xi)$ and $v_\xi(t,\xi)$ are not continuous with respect to time and space, and this fact prevents us from using the following splitting
\begin{align*}
\frac{v(t,\ell(t,\eta))-v(s,\ell(s,\eta))}{t-s}&=\frac{v(t,\ell(t,\eta))-v(s,\ell(t,\eta))}{t-s}\\
& \quad +\frac{v(s,\ell(t,\eta))-v(s,\ell(s,\eta))}{\ell(t,\eta)-\ell(s,\eta)}\frac{\ell(t,\eta)-\ell(s,\eta)}{t-s},
\end{align*}
together with a limiting process based on the existence of the partial derivatives of $v$. 
Furthermore, $\ell(t,\eta)$ is not differentiable almost everywhere, we only know that for fixed $t$, the derivatives $\ell_t(t,\eta)$ and $\ell_\eta(t,\eta)$ exist for almost every $\eta$.
\end{remark}

\begin{remark} An alternative derivation of the differential equation for $\chi(t,\eta)$: 
By the weak formulation we have, cf.~\eqref{svak2}, \eqref{evolp}, and \eqref{def:nu_def2}, that 
\begin{align*}
\int_0^T\int_\Real &(\phi_t+u\phi_x)(t,x) d\nu(t,x)dt  \nn \\ 
& \qquad - \int_0^T \int_\Real \int_\Real u(t,y)K'(x-y)d\nu(t,y) \phi(t,x) dxdt \nn \\
& \qquad  - \int_0^T\int_\Real \Big(up\phi_x(t,x)-\int_\Real u(t,y)K'(x-y)p(t,y)dy\phi(t,x)\Big) dx dt  \nn \\
&\quad =\int_\Real \phi(T,x)d\nu(T,x)-\int_\Real\phi(0,x)d\nu(0,x)
\end{align*}
for any test function $\phi\in \C_c^\infty([0,\infty)\times\Real)$. Using \eqref{eq:nu_def} and changing the coordinates according to \eqref{eq:chi_defi} and \eqref{newvar}, we end up with
\begin{align}\nn
\int_0^T& \int_0^{B+C} \big(\phi_t(t,\chi(t,\eta))+\U(t,\eta)\phi_x(t,\chi(t,\eta))\big)d\eta dt\\ \nn
&  -\int_0^T \int_0^{B+C} \int_0^{B+C} \U(t,\tilde\eta)\\ \nn
&\qquad\qquad\times K'(\chi(t,\eta)-\chi(t,\tilde\eta))(1-\P\chi_\eta(t,\tilde\eta))d\tilde\eta\,\phi(t,\chi(t,\eta))\chi_\eta(t,\eta) d\eta dt\\ \nn
& -\int_0^T\int_0^{B+C} \U\P(t,\eta)\phi_x(t,\chi(t,\eta))\chi_\eta(t,\eta)d\eta dt\\ \label{weak:chi}
&= \int_0^{B+C} \phi(T,\chi(T,\eta)) d\eta -\int_0^{B+C} \phi(0,\chi(0,\eta))d\eta.
\end{align}
First of all note that using integration by parts (for the integral with respect to $\eta$) in the triple integral we obtain 
\begin{align*}
& \int_0^T \int_0^{B+C} \int_0^{B+C} \U(t,\tilde\eta)\\
&\qquad \qquad\qquad\times K'(\chi(t,\eta)-\chi(t,\tilde\eta))(1-\P\chi_\eta(t,\tilde\eta))d\tilde\eta\, \phi(t,\chi(t,\eta))\chi_\eta(t,\eta) d\eta dt\\
&=- \int_0^T \int_0^{B+C} \int_0^{B+C} \U(t,\tilde\eta)\\
&\qquad \qquad\qquad\times K(\chi(t,\eta)-\chi(t,\tilde\eta))(1-\P\chi_\eta(t,\tilde\eta))d\tilde\eta\, \phi_x(t,\chi(t,\eta))\chi_\eta(t,\eta) d\eta dt.
\end{align*}
Furthermore, the right-hand side of \eqref{weak:chi} can be rewritten as 
\begin{align*}
\int_0^{B+C} \phi(T,\chi(T,\eta))& d\eta -\int_0^{B+C} \phi(0,\chi(0,\eta))d\eta\\
&= \int_0^T \int_0^{B+C} \big(\phi_t(t,\chi(t,\eta))+\phi_x(t,\chi(t,\eta))\chi_t(t,\eta)\big) d\eta dt.
\end{align*}
Combining the last two equation with \eqref{weak:chi}, we have
\begin{align*}
\int_0^T& \int_0^{B+C} \U(t,\eta)\phi_x(t,\chi(t,\eta))d\eta dt\\ \nn
&  +\int_0^T \int_0^{B+C} \int_0^{B+C} \U(t,\tilde\eta)\\ \nn
&\qquad \qquad\times K(\chi(t,\eta)-\chi(t,\tilde\eta))(1-\P\chi_\eta(t,\tilde\eta))d\tilde\eta\, \phi_x(t,\chi(t,\eta))\chi_\eta(t,\eta) d\eta dt\\ \nn
& -\int_0^T\int_0^{B+C} \U\P(t,\eta)\phi_x(t,\chi(t,\eta))\chi_\eta(t,\eta)d\eta dt\\ 
&= \int_0^T\int_0^{B+C} \phi_x(t,\chi(t,\eta))\chi_t(t,\eta)d\eta dt.
\end{align*}
Now we are ready to read off the differential equation for $\chi(t,\eta)$, since the above equality is equivalent to 
\begin{multline*}
\int_0^T\int_0^{B+C} \big(\chi_t(t,\eta)+h(t,\eta)\chi_\eta(t,\eta)\big)\phi_x(t,\chi(t,\eta)) d\eta dt\\
= \int_0^T\int_0^{B+C} \U(t,\eta)\phi_x(t,\chi(t,\eta)) d\eta dt,
\end{multline*}
where $h(t,\eta)$ is given by \eqref{def:h}.
Since the above equality must hold for any test function $\phi$, we end up with 
\begin{equation*}
\chi_t(t,\eta)+h(t,\eta) \chi_\eta(t,\eta)=\U(t,\eta).
\end{equation*}
\end{remark}

\subsection{Differentiability of $\U(t,\eta)$ with respect to time}\label{sec:U}

To begin with we have a closer look at the differential equation \eqref{difflig:chi}, which reads
 \begin{equation*}
 \chi_t(t,\eta)+h(t,\eta)\chi_\eta(t,\eta)=\U(t,\eta),
 \end{equation*}
 where $h(t,\eta)$ is given by \eqref{eq:def_h}.
 This equation can be solved (uniquely) by the method of characteristics, if the differential equation
 \begin{equation}\label{karak2}
 m_t(t,\theta)=h(t,m(t,\theta))
 \end{equation}
 has a unique solution and $m_\theta(t,\dott)$ is strictly positive for all $t\in[0,T]$. According to classical ODE theory, \eqref{karak2} has for each fixed $\theta$ a unique solution, if the function $h(t,\eta)$ is continuous with respect to time and Lipschitz with respect to space. This is the result of the next lemma, whose proof can be found in Lemma~\ref{lemma:A2}.
 Hence \eqref{karak2} has a unique solution.
 
 \begin{lemma} \label{lem:appendix2} Consider the function $h$ defined by \eqref{eq:def_h}. Then \\
 (i)  $t\mapsto h(t,\eta)$  is continuous; \\
 (ii) $\eta\mapsto h(t,\eta)$ is Lipschitz and satisfies
 \begin{equation}\label{Lcont_h}
\vert h(t,\eta_2)-h(t,\eta_1)\vert\leq \big(1+2n\norm{u(0,\dott)}_{L^\infty}+C+\frac{n}2 Ct\big)\vert \eta_2-\eta_1\vert.
\end{equation}
  \end{lemma}

For the method of characteristics to be well-defined, we must check that solutions to \eqref{karak2} are strictly increasing. If $m(0,\theta)=\theta$ for all $\theta \in [0,B+C]$ and $\theta_1<\theta_2$, we have, using \eqref{Lcont_h}, as long as the function $m(t,\dott)$ remains non-decreasing, that 
\begin{equation*}
\abs{(m(t,\theta_2)-m(t,\theta_1))_t}\leq B(t)(m(t,\theta_2)-m(t,\theta_1)),
\end{equation*}
where $B(t)=1+2n\norm{u(0,\dott)}_{L^\infty}+C+\frac{n}2Ct$,
which yields 
\begin{equation*}
(m(0,\theta_2)-m(0,\theta_1))e^{-\int_0^t B(s)ds}\leq m(t,\theta_2)-m(t,\theta_1)\leq (m(0,\theta_2)-m(0,\theta_1))e^{\int_0^t B(s)ds}.
\end{equation*}
Thus $m(t,\dott)$ not only remains strictly increasing, it is also Lipschitz continuous with Lipschitz constant $e^{\int_0^t B(s)ds}$ and hence according to Rademacher's theorem differentiable almost everywhere. In particular, one has that 
\begin{equation}\label{bder:m}
e^{-\int_0^t B(s)ds}\leq m_\theta(t,\theta)\leq e^{\int_0^t B(s)ds}.
\end{equation}
Thus $m(t,\dott)$ is strictly increasing.

 Introducing 
 \begin{equation} \label{eq:chibar}
 \bar \chi(t,\theta)=\chi(t,m(t,\theta)),
 \end{equation}
 we have from \eqref{difflig:chi}
 \begin{equation}\label{difflig:barchi}
\bar \chi_t(t,\theta)=u(t,\bar \chi(t,\theta)),
\end{equation}
and hence $\bar \chi(t,\eta)$ is a characteristic. Introducing 
\begin{equation}\label{def:barU}
\bar \U(t,\theta)=u(t,\bar\chi(t,\theta))=\U(t,m(t,\theta)),
\end{equation}
equation \eqref{difflig:barchi} reads
\begin{equation*}
\bar \chi_t(t,\theta)=\bar \U(t,\theta).
\end{equation*}
To take a next step towards deriving the differential equation for $\U(t,\eta)$, we want to show that $\bar \U(t,\theta)$ satisfies
\begin{equation}\label{savnet:difflig1}
\bar \U_t(t,\theta)=\frac12\Big(m(t,\theta)-\frac12 C-\int_0^\theta \bar \P(t,\alpha)\bar \chi_\theta(t,\alpha) d\alpha\Big),
\end{equation}
where 
\begin{equation}\label{def:barP}
\bar \P(t,\theta)=p(t,\bar \chi(t,\theta))=\P(t,m(t,\theta)).
\end{equation}

The proof of \eqref{savnet:difflig1} is again based on an idea that has been used in \cite{BCZ}. According to the definition of a weak solution, one has, see \eqref{svak1} and \eqref{eq:nu_def},  
for all $\phi\in \C_c^\infty([0,\infty)\times\Real)$ that 
\begin{align}
\int_s^t \int_\Real & \Big(u(\tau,x)(\phi_t+\frac12u\phi_x)(\tau,x) +\frac12 \big(G(\tau,x)-\int_{-\infty}^x p(\tau,y)dy-\frac12 C\big)\phi(\tau,x)\Big)dx d\tau \nn \\ 
&\qquad \qquad \qquad = \int_\Real u\phi(t,x)dx-\int_\Real u\phi(s,x)dx. \label{eq:svaktid}
\end{align}
In the above equality one can replace $\phi\in \C_c^\infty([0,\infty)\times \Real)$ by $\phi(t,x)$ such that $\phi(t,\dott)\in \C_c^\infty(\Real)$ for all $t$ and $\phi(\dott,x)\in \C^1(\Real)$. 

\subsubsection*{The Lipschitz continuity of $\bar \U(t,\dott)$}
To prove that $\bar \U(t,\theta)$ is Lipschitz with respect to time, we have to make a special choice of $\phi(t,x)$. Let
\begin{equation*}
\phi_\varepsilon(t,x)=\frac{1}{\varepsilon}\psi\Big(\frac{\bar \chi(t,\theta)-x}{\varepsilon}\Big),
\end{equation*}
where $\psi$ is a standard Friedrichs mollifier. Our choice is motivated by the following observation,
\begin{equation*}
\lim_{\varepsilon\to 0}\int_\Real u\phi_\varepsilon(t,x)dx=u(t,\bar\chi(t,\theta)),
\end{equation*}
and hence
\begin{equation*}
u(t,\bar\chi(t,\theta))-u(s,\bar\chi(s,\theta))=\lim_{\varepsilon\to 0}\int_\Real \big(u\phi_\varepsilon(t,x)-u\phi_\varepsilon(s,x)\big)dx.
\end{equation*}
Direct calculations then yield, using \eqref{difflig:barchi}
\begin{align}
\int_\Real \Big(u\phi_{\varepsilon,t}+\frac12 u^2\phi_{\varepsilon,x}& +\frac12 (G-\int_{-\infty}^x p(\tau,y)dy-\frac12C)\phi_\varepsilon(\tau,x)\Big)dx\nn\\
& = \frac12 u(\tau,\bar\chi(\tau,\theta))^2\int_\Real \frac{1}{\varepsilon^2}\psi'\left(\frac{\bar\chi(\tau,\theta)-x}{\varepsilon}\right)dx\nn\\
& \quad -\frac12 \int_\Real (u(\tau,x)-u(\tau,\bar\chi(\tau,\theta)))^2\frac{1}{\varepsilon^2}\psi'\left(\frac{\bar\chi(\tau,\theta)-x}{\varepsilon}\right)dx\nn\\
& \quad + \frac12 \int_\Real (G(\tau,x)-G(\tau,\bar\chi(\tau,\theta)))\frac{1}{\varepsilon}\psi\left(\frac{\bar\chi(\tau,\theta)-x}{\varepsilon}\right)dx\nn\\
& \quad -\frac12 \int_\Real \int_{\bar\chi(\tau,\theta)}^xp(\tau,y)dy\frac{1}{\varepsilon}\psi\left(\frac{\bar\chi(\tau,\theta)-x}{\varepsilon}\right)dx\nn\\
& \quad + \frac12 \left(G(\tau,\bar\chi(\tau,\theta))-\int_{-\infty}^{\bar\chi(\tau,\theta)} p(\tau,y)dy\right)-\frac14 C\nn\\
& = -\frac12 \int_\Real (u(\tau,x)-u(\tau,\bar\chi(\tau,\theta)))^2\frac{1}{\varepsilon^2}\psi'\left(\frac{\bar\chi(\tau,\theta)-x}{\varepsilon}\right)dx\nn\\
& \quad -\frac12 \int_\Real \int_{\bar\chi(\tau,\theta)}^xp(\tau,y)dy\frac{1}{\varepsilon}\psi\left(\frac{\bar\chi(\tau,\theta)-x}{\varepsilon}\right)dx\nn\\
& \quad + \frac12 \int_\Real (G(\tau,x)-G(\tau,\bar\chi(\tau,\theta)))\frac{1}{\varepsilon}\psi\left(\frac{\bar\chi(\tau,\theta)-x}{\varepsilon}\right)dx\nn\\
& \quad +\frac12 \left(G(\tau,\bar\chi(\tau,\theta))-\int_{-\infty}^{\bar\chi(\tau,\theta)} p(\tau,y)dy\right)-\frac14 C. \label{eq:Ggreie}
\end{align}
Introduce the function 
\begin{equation*}
G_{\rm ac}(\tau,x)=\int_{-\infty}^x (u_x^2+p)(\tau,y)dy
\end{equation*}
and note that $G_{\rm ac}(\tau,\dott)$ is absolutely continuous. Moreover, one has 
\begin{equation*}
\vert u(\tau,x)-u(\tau,y)\vert \leq \vert x-y\vert^{1/2}\vert G_{\rm ac}(\tau,x)-G_{\rm ac}(\tau,y)\vert^{1/2},
\end{equation*}
and 
\begin{align*}
\vert \int_\Real (u(\tau,x)-u(\tau,\bar\chi(\tau,\theta)))^2 &\frac{1}{\varepsilon^2}\psi'\left(\frac{\bar\chi(\tau,\theta)-x}{\varepsilon}\right) dx\vert\\
&= \frac{1}{\varepsilon}\vert \int_{-1}^{1} (u(\tau,\bar\chi(\tau,\theta)-\varepsilon\eta)-u(\tau,\bar\chi(\tau,\theta)))^2 \psi'(\eta)d\eta\vert\\
&\leq \vert G_{\rm ac}(\tau,\bar\chi(\tau,\theta)+\varepsilon)-G_{\rm ac}(\tau,\bar\chi(\tau,\theta)-\varepsilon)\vert,
\end{align*}
which implies 
\begin{equation*}
\lim_{\varepsilon\to 0}\vert \int_\Real (u(\tau,x)-u(\tau,\bar\chi(\tau,\theta)))^2\frac{1}{\varepsilon^2}\psi'\left(\frac{\bar\chi(\tau,\theta)-x}{\varepsilon}\right)dx\vert =0.
\end{equation*}
For the next term observe
\begin{align*}
\vert \int_\Real \int_{\bar\chi(\tau,\theta)}^x p(\tau,y)dy& \frac{1}{\varepsilon}\psi\left(\frac{\bar\chi(\tau,\theta)-x}{\varepsilon}\right)dx\vert\\
& \leq\frac{1}{\varepsilon}\int_\Real \vert G_{\rm ac}(\tau,x)-G_{\rm ac}(\tau,\bar\chi(\tau,\theta))\vert \psi\left(\frac{\bar\chi(\tau,\theta)-x}{\varepsilon}\right) dx \\
& \leq \vert G_{\rm ac}(\tau,\bar\chi(\tau,\theta)+\varepsilon)-G_{\rm ac}(\tau,\bar\chi(\tau,\theta)-\varepsilon)\vert,
\end{align*}
which implies 
\begin{equation*}
\lim_{\varepsilon\to 0}\vert \int_\Real \int_{\bar\chi(\tau,\theta)}^x p(\tau,y)dy \frac{1}{\varepsilon}\psi\left(\frac{\bar\chi(\tau,\theta)-x}{\varepsilon}\right)dx\vert =0.
\end{equation*}
For the last two terms, observe that for every $\tau\not \in \mathcal{N}$, i.e., for almost every $\tau\in [0,T]$ one has, using \eqref{eq:chi_defi}, that
\begin{equation*}
G(\tau,\bar\chi(\tau,\theta))=G_{\rm ac}(\tau,\bar\chi(\tau,\theta))=G_{\rm ac}(\tau,\chi(\tau,m(t,\theta)))=m(\tau,\theta)
\end{equation*}
and 
\begin{equation*}
\lim_{\varepsilon\to 0} \frac12 \int_\Real \big(G(\tau,x)-G(\tau,\bar\chi(\tau,\theta))\big)\frac{1}{\varepsilon}\psi\left(\frac{\bar\chi(\tau,\theta)-x}{\varepsilon}\right)dx=0.
\end{equation*}
Thus, the dominated convergence theorem, using \eqref{eq:Ggreie} and \eqref{eq:svaktid}, yields that 
\begin{align*}
\bar \U(t,\theta)-\bar \U(s,\theta)&=\int_s^t \frac12 \Big(m(\tau,\theta)-\int_{-\infty}^{\bar\chi(\tau,\theta)} p(\tau,y)dy\Big)d\tau -\frac14 C(t-s).
\end{align*}
Furthermore, the continuity of $m(t,\theta)$, $\bar \chi(t,\theta)$, and $p(t,y)$ with respect to time, imply that 
\begin{equation}\label{difflig:bU}
\bar \U_t(t,\theta)=\frac12\Big(m(t,\theta)-\int_{-\infty}^{\bar \chi(t,\theta)} p(t,y)dy\Big)-\frac14 C
\end{equation}
for all $(t,\theta)\in [0,T]\times [0, B+C]$.
Thus $\bar\U(\dott,\eta)$ will be Lipschitz on $[0,T]$, if we can show that the right-hand side of \eqref{difflig:bU} can be uniformly bounded on $[0,T]$. Since $p(t,\dott) \in L^1(\Real)$, cf.~\eqref{L1p}, the claim follows if $m(t,\dott)$ can be bounded. Therefore recall that $m(t,\theta)$ satisfies \eqref{karak2} with initial condition $m(0,\theta)=\theta$, which implies that it suffices to show that $\norm{h(t,\dott)}_{L^\infty}$ can be bounded by a function, which is at most growing linearly. 
Using \eqref{eq:def_h}, \eqref{newvar}, \eqref{normu}, Lemma~\ref{lemma:Kp}, and \eqref{defp}, we have
\begin{align} \nn
\norm{h(t,\dott)}_{L^\infty}& \leq \norm{u(t,\dott)}_{L^\infty}\norm{p(t,\dott)}_{L^\infty}\\ \nn
& \quad + (B+C)\norm{u(t,\dott)}_{L^\infty}+\norm{u(t,\dott)}_{L^\infty}\norm{p(t,\dott)}_{L^\infty}\norm{K}_{L^1}\\ \nn
& \leq (B+(2+\pi)C)\norm{u(t,\dott)}_{L^\infty}\\ \label{boundu:h}
& \leq (B+(2+\pi)C)(\norm{u(0,\dott)}_{L^\infty}+\frac14 Ct),
\end{align}
and thus $\norm{h(t,\dott)}_{L^\infty}$ grows at most linearly.  This finishes the proof of the uniform Lipschitz continuity in time of  
$\bar \U(t,\theta)$.

As a closer look reveals the differential equation \eqref{difflig:bU} holds pointwise. Thus one can either look at $\bar \U(t,\dott)$ as a function in $L^\gamma([0,B+C])$ or in $L^\infty([0,B+C])$. Furthermore, the mapping $t\mapsto \bar \U(t,\dott)$ is continuous from $[0,T]$ into $L^\gamma([0,B+C])$.

\subsubsection*{The  differential equation for $\U(t,\eta)$}
In view of \eqref{difflig:chi} it remains to derive the differential equation for $\U(t,\eta)$ from \eqref{difflig:bU}.
Recall that we have by \eqref{def:barU} the relation 
\begin{equation*}
\bar \U(t,\theta)=\U(t,m(t,\theta)).
\end{equation*}
Since $m(t,\dott)$ is continuous and strictly increasing, cf.~\eqref{bder:m}, there exists a unique, continuous and strictly increasing function $n(t,\dott)$ such that 
\begin{equation}\label{inv:char}
n(t,m(t,\theta))=\theta \quad \text{ for all }(t,\theta)\in [0,T]\times [0,B+C].
\end{equation}
Now, given $\eta\in [0,B+C]$, there exists a unique $\theta\in [0,B+C]$, such that 
\begin{equation*}
n(t,\eta)-n(s,\eta)=n(t,m(t,\theta))-n(s,m(t,\theta))=n(s,m(s,\theta))-n(s,m(t,\theta)).
\end{equation*}
By definition, we have that $n(t,\dott)$ is continuous, strictly increasing, and hence differentiable almost everywhere. Furthermore, one has that $n(t,\dott)$ satisfies
\begin{align*}
e^{-\int_0^t B(\tau)d\tau}\leq n_\eta(t,\eta)=\frac{1}{m_\theta(t,n(t,\eta))}\leq e^{\int_0^t B(\tau)d\tau},
\end{align*}
by \eqref{bder:m},
which yields 
\begin{align*}
\vert n(t,\eta)-n(s,\eta)\vert &\leq e^{\int_0^{T} B(\tau)d\tau}\vert m(s,\theta)-m(t,\theta)\vert\\
&\leq e^{\int_0^T B(\tau)d\tau}\int_{\min(s,t)}^{\max(s,t)} \vert h(\tau, m(\tau,\theta))\vert d\tau .
\end{align*}
Since $h(t,\eta)$ can be uniformly bounded on $[0,T]\times [0, B+C]$, see \eqref{boundu:h}, it follows that $n(t,\eta)$ is Lipschitz with respect to both space and time on $[0,T]\times [0, B+C]$ and by Rademacher's theorem differentiable almost everywhere. Moreover, direct calculations yield
\begin{equation*}
n_t(t,\eta)+hn_\eta(t,\eta)=0 \quad \text{ almost everywhere}.
\end{equation*}

Since $n(t,\dott)$ is strictly increasing for every $t\in [0,T]$, combining \eqref{def:barU} and \eqref{inv:char}, yields
\begin{equation*}
\U(t,\eta)=\bar \U(t,n(t,\eta)) \quad \text{ for all } (t,\eta)\in [0,T]\times [0,B+C].
\end{equation*}
Furthermore, using \eqref{newvar}, \eqref{chi:rel}, \eqref{def:U}, and \eqref{def:barU}, one has
\begin{equation*}
\U(t,\eta)=U(t,\ell(t,\eta))\quad \text{ and } \quad \bar \U(t,\theta)=U(t,\ell(t,m(t,\theta))).
\end{equation*}
In Section~\ref{diff:U} we showed, that for every $t\in [0,T]$, the function $U(t,\dott)$ is Lipschitz. Since $\ell(t,\dott)$ and $m(t,\dott)$ are strictly increasing, it follows that both $\U(t,\dott)$ and $\bar \U(t,\dott)$ are differentiable almost everywhere on $[0,B+C]$. 
Thus, for any $t\in [0,T]$ one ends up with 
\begin{align*}
\U_t(t,\eta)& =\bar \U_t(t,n(t,\eta))+\bar \U_\theta(t,n(t,\eta))n_t(t,\eta)\\
& = \frac12 \Big(\eta-\int_{-\infty}^{\chi(t,\eta)} p(t,y)dy\Big)-\frac14 C- h(t,\eta)\U_\eta(t,\eta) \\
& =\frac12 \Big(\eta-\int_0^\eta \P\chi_\eta(t,\tilde\eta) d\tilde\eta\Big)-\frac14C-h(t,\eta)\U_\eta(t,\eta)
\end{align*}
for almost every $\eta\in [0,B+C]$. This is the sought differential equation for
$\U(t,\eta)$.

The correct Banach space to work in is $L^\gamma([0,B+C])\cap L^\infty([0,B+C])$. To see this, introduce the function
\begin{equation*}
g_2(t,\eta)=\int_{-\infty}^{\chi(t,\eta)} pu_x(t,y) dy=\int_{0}^{\eta} \P\U_\eta(t,\tilde\eta)d\tilde\eta. 
\end{equation*}
Observe that \eqref{eq:nu_def} and \eqref{interpolation} imply that $g_2(t,\dott)$ is Lipschitz continuous with Lipschitz constant at most $1+C$, since  
\begin{align*}
\vert g_2(t,\eta_2)-g_2(t,\eta_1)\vert & \leq \vert \int_{\chi(t,\eta_1)}^{\chi(t,\eta_2)} pu_x(t,y) dy\vert \\
& \leq  \vert \int_{\chi(t,\eta_1)}^{\chi(t,\eta_2)} (u_x^2+p^2)(t,y) dy\vert\\
& \leq (1+C)\vert \int_{\chi(t,\eta_1)}^{\chi(t,\eta_2)} (u_x^2+p)(t,y) dy\vert\\
& \leq (1+C)\vert G_{\rm ac}(t,\chi(t,\eta_2))-G_{\rm ac}(t,\chi(t,\eta_1))\vert\\
& \leq (1+C)\vert \eta_2-\eta_1\vert.
\end{align*}
Thus it follows that for any fixed $t\in [0,T]$ 
\begin{equation*}
g_{2,\eta}(t,\eta)=\P\U_\eta(t,\eta)\quad \text{ for almost every } \eta\in [0,B+C].
\end{equation*}
In  particular, 
one has
\begin{equation*}
\vert \P\U_\eta(t,\eta)\vert \leq (1+C)
\end{equation*}
and 
\begin{equation*}
\int_0^{B+C} \vert\P\U_\eta\vert^\gamma(t,\eta)d\eta\leq (1+C)^{\gamma-1} \int_0^{B+C}\vert \P\U_\eta\vert (t,\eta)d\eta= (1+C)^{\gamma-1}\int_\Real \vert pu_x\vert (t,x)dx,
\end{equation*}
which is finite, cf.~\eqref{Lup} and \eqref{eq:nuasymp},  and which implies that $h\U_\eta(t,\dott)\in L^\gamma([0,B+C])\cap L^\infty([0,B+C])$.

To summarize we showed the following theorem. 

\begin{theorem}\label{thm:slashU}
The function $\U(t,\dott)$ satisfies the following differential equation in $L^\gamma([0,B+C])\cap L^\infty([0,B+C])$
\begin{equation}\label{difflig:barU}
\U_t(t,\eta)+h(t,\eta)\U_\eta(t,\eta)=\frac12 \Big(\eta-\int_0^\eta \P\chi_\eta(t,\tilde\eta) d\tilde\eta\Big)-\frac14 C,
\end{equation}
where 
\begin{equation*}
h(t,\eta)=\U\P(t,\eta)-\int_0^{B+C} \U(t,\tilde \eta)K(\chi(t,\eta)-\chi(t,\tilde\eta))(1-\P\chi_\eta(t,\tilde\eta))d\tilde\eta.
\end{equation*}
Furthermore, the mapping $t\mapsto \U(t,\dott)$ is continuous from $[0,T]$ into $L^\gamma([0,B+C])$.
\end{theorem}

\subsection{Differentiability of $\P(t,\eta)$ with respect to time}\label{sec:P}
To close the system of differential equations \eqref{difflig:chi} and \eqref{difflig:barU}, i.e.,
\begin{align*}
\chi_t(t,\eta)+h(t,\eta)\chi_\eta(t,\eta)&=\U(t,\eta),\\
\U_t(t,\eta)+h(t,\eta)\U_\eta(t,\eta)&=\frac12 \Big(\eta-\int_0^\eta \P\chi_\eta(t,\tilde\eta) d\tilde\eta\Big)-\frac14 C,
\end{align*}
where 
\begin{equation*}
h(t,\eta)=\U\P(t,\eta)-\int_0^{B+C} \U(t,\tilde \eta)K(\chi(t,\eta)-\chi(t,\tilde\eta))(1-\P\chi_\eta(t,\tilde\eta))d\tilde\eta,
\end{equation*}
it remains to derive the differential equation satisfied by $\P(t,\eta)$, given by \eqref{newvar}. 

\subsubsection*{The Lipschitz continuity of $\P(t,\eta)$}
To begin with, we show that $\P(t,\eta)$ is Lipschitz continuous on $[0,T]\times[0,B+C]$ and hence differentiable almost everywhere on $[0,T]\times [0, B+C]$. Direct calculations yield
\begin{align}
\P(t,\eta)-\P(s,\theta)&= p(t, \chi(t,\eta))-p(s,\chi(s,\theta)) \label{eq:LipP}\\
& = p(t,\chi(t,\eta))-p(s,\chi(t,\eta))\nn\\
& \quad +p(s,\chi(t,\eta))-p(s,\chi(s,\eta))+p(s,\chi(s,\eta))-p(s,\chi(s,\theta))\nn\\
& = \int_{s}^t p_t(\tau, \chi(t,\eta))d\tau \nn\\
& \quad + \int_0^1 p_x\big(s, \chi(s,\eta)+ k(\chi(t,\eta)-\chi(s,\eta))\big)(\chi(t,\eta)-\chi(s,\eta)) dk\nn\\
& \quad + \int_{\theta}^{\eta} p_x(t,\chi(t,\tilde \eta))\chi_\eta(t,\tilde \eta) d\tilde \eta, \nn
\end{align}
where we used that $p(t,\dott)\in \C_0^\infty(\Real)$ and $p(\dott, x)$ is differentiable almost every on the finite interval $[0,T]$. Furthermore, we took advantage of $\chi(t,\dott)$ being continuous and increasing and hence differentiable almost everywhere on $[0,B+C]$.

Combining \eqref{eq:LipP} with \eqref{Lip:Chi} and \eqref{bound:Pchin} as well as \eqref{eq:p_t} and \eqref{eq:p_x}, we end up with 
\begin{equation*}
\vert \P(t,\eta)-\P(s,\theta)\vert \leq \max((\norm{u(0,\dott)}_{L^\infty}+2A+\frac14CT)nC, n)(\vert t-s\vert +\vert \eta-\theta\vert),
\end{equation*}
i.e., $\P(t,\eta)$ is Lipschitz continuous on $[0,T]\times [0,B+C]$.

\bigskip
We are now ready to derive the differential equation satisfied by $\P(t,\eta)$. Therefore note that \eqref{pointchi} implies that for each $t\in [0,T]$ one has 
\begin{equation*}
\P_\eta(t,\eta)=p_x(t,\chi(t,\eta))\chi_\eta(t,\eta)
\end{equation*}
for almost every $\eta\in [0,B+C]$, since $p(t,\dott)\in \C_0^\infty(\Real)$.
Furthermore, one has  for almost every $(t,\eta)\in[0,T]\times [0, B+C]$ that
\begin{align*}
\P_t(t,\eta)&=\lim_{s\to t}\frac{\P(t,\eta)-\P(s,\eta)}{t-s}\\
&=\lim_{s\to t}\frac{p(t,\chi(t,\eta))-p(s,\chi(t,\eta))}{t-s}+\lim_{s\to t}\frac{p(s,\chi(t,\eta))-p(s,\chi(s,\eta))}{t-s},
\end{align*}
if the two limits on the right-hand side exist. The first one exists and is given by \eqref{improved:pt}. The second one, on the other hand, requires a closer look. Write
\begin{align*}
p(s,\chi(t,\eta))-&p(s,\chi(s,\eta))\\
& = \int_{\chi(s,\eta)}^{\chi(t,\eta)} (p_x(s,y)-p_x(t,y))dy \\
& \quad + \int_0^1 p_x\big(t,\chi(t,\eta)+k(\chi(s,\eta)-\chi(t,\eta))\big)dk\, (\chi(t,\eta)-\chi(s,\eta)).
\end{align*}
Following once more the argument leading to \eqref{LIPP}, recalling \eqref{normu}, \eqref{Lip:Chi}, Lemma~\ref{lemma:Kp}, and that $p_x(t,\dott)\in \C_0^\infty(\Real)$, we have 
\begin{equation*}
 \int_{\chi(s,\eta)}^{\chi(t,\eta)} (p_x(s,y)-p_x(t,y))dy=\mathcal{O}((t-s)^2)
 \end{equation*}
 and 
 \begin{equation*}
\lim_{s\to t}\frac{p(s,\chi(t,\eta))-p(s,\chi(s,\eta))}{t-s}=p_x(t,\chi(t,\eta))\lim_{s\to t} \frac{\chi(t,\eta)-\chi(s,\eta)}{t-s},
\end{equation*}
if the limit on the right-hand side exists. In view of \eqref{pointchi}, one has that for each $t\in [0,T]$
\begin{align*}
\P_t(t,\eta)&=p_t(t,\chi(t,\eta))+p_x(t,\chi(t,\eta))\chi_t(t,\eta)
\end{align*}
for almost every $\eta\in [0,B+C]$.  We find, using \eqref{difflig:chi}, (cf.~\eqref{FD2}), \eqref{evolp}, \eqref{rep:chi},
\eqref{def:H2}, \eqref{def:l}, \eqref{eq:g_xi}, that
\begin{align}
\P_t(t,\eta)+&h(t,\eta)\P_\eta(t,\eta)\\ \nn 
&= p_t(t,\chi(t,\eta))+p_x(t,\chi(t,\eta))\U(t,\eta) \nn \\
&= -\int_\Real u(t,z)K'(\chi(t,\eta)-z) d\mu(t,z)+\U(t,\eta)\int_\Real K'(\chi(t,\eta)-z) d\mu(t,z)\nn\\
&= -\int_\Real u(t,y(t,\xi))K'(\chi(t,\eta)-y(t,\xi)) \tilde H_\xi(t,\xi)d\xi\nn\\
&\quad+\U(t,\eta)\int_\Real K'(\chi(t,\eta)-y(t,\xi)) \tilde H_\xi(t,\xi)d\xi\nn\\
&= -\int_0^{B+C}  u(t,y(t,\ell(t,\tilde\eta)))K'(\chi(t,\eta)-y(t,\ell(t,\tilde\eta))) \tilde H_\xi(t,\ell(t,\tilde\eta))\ell_\eta(t,\tilde\eta)d\tilde\eta\nn\\
&\quad+\U(t,\eta)\int_0^{B+C}  K'(\chi(t,\eta)-y(t,\ell(t,\tilde\eta))) \tilde H_\xi(t,\ell(t,\tilde\eta))\ell_\eta(t,\tilde\eta)d\tilde\eta\nn\\
&=-\int_0^{B+C} \U(t,\tilde\eta)K'(\chi(t,\eta)-\chi(t,\tilde\eta)) (1-\P\chi_\eta(t,\tilde\eta)) d\tilde\eta\nn\\
& \quad +\U(t,\eta)\int_0^{B+C} K'(\chi(t,\eta)-\chi(t,\tilde\eta))(1-\P\chi_\eta(t,\tilde\eta))d\tilde\eta. \label{eq:PP}
\end{align}

The correct Banach space to work in is again $L^\gamma([0, B+C])\cap L^\infty([0,B+C])$, since 
\begin{equation*}
\vert \P_\eta(t,\eta)\vert \leq n\P\chi(t,\eta)\leq n
\end{equation*}
by \eqref{eq:p_x} and \eqref{bound:Pchin}.

To summarize we showed the following theorem.

\begin{theorem}\label{thm:slashP}
The function $\P(t,\dott)$, given by \eqref{newvar},  satisfies for almost all $t$ the following differential equation in $L^\gamma([0,B+C])\cap L^\infty([0,B+C])$
\begin{equation}\label{difflig:barP}
\P_t(t,\eta)+h(t,\eta)\P_\eta(t,\eta)=\R(t,\eta),
\end{equation}
where 
\begin{equation*}
h(t,\eta)=\U\P(t,\eta)-\int_0^{B+C} \U(t,\tilde \eta)K(\chi(t,\eta)-\chi(t,\tilde\eta))(1-\P\chi_\eta(t,\tilde\eta))d\tilde\eta
\end{equation*}
and 
\begin{align}
\R(t,\eta)&=-\int_0^{B+C} \U(t,\tilde\eta)K'(\chi(t,\eta)-\chi(t,\tilde\eta))\big(1-\P\chi_\eta(t,\tilde\eta)\big)d\tilde\eta \nn \\ \label{def:slashR} 
& \quad +\U(t,\eta)\int_0^{B+C} K'(\chi(t,\eta)-\chi(t,\tilde\eta))\big(1-\P\chi_\eta(t,\tilde\eta)\big)d\tilde\eta.
\end{align}
Furthermore, the mapping $t\mapsto \P(t,\dott)$ is continuous from $[0,T]$ into $L^\gamma([0,B+C])$.
\end{theorem}

\subsection{Summary}

We have shown the following result.
\begin{theorem}\label{thm:systemnew}
The functions $\chi(t,\eta)$, $\U(t,\eta)$, and $\P(t,\eta)$ defined by \eqref{eq:chi_defi} and \eqref{newvar}, respectively, satisfy the following system of equations:
\begin{subequations}\label{systemnew:sec}
\begin{align}
\chi_t(t,\eta)+h(t,\eta)\chi_\eta(t,\eta)&=\U(t,\eta), \label{systemnew:sec1}\\[2mm] 
\U_t(t,\eta)+h(t,\eta)\U_\eta(t,\eta) &=\frac12 \Big(\eta-\int_0^\eta \P\chi_\eta(t,\tilde \eta) d\tilde\eta\Big)-\frac14C,\label{systemnew:sec2}\\[2mm] 
\P_t(t,\eta)+h(t,\eta)\P_\eta(t,\eta)&=\R(t,\eta)\label{systemnew:sec3}
\end{align}
\end{subequations}
where 
\begin{align}\label{fdef:h}
h(t,\eta)&=\U\P(t,\eta) -\int_0^{B+C} \U(t,\tilde\eta)K(\chi(t,\eta)-\chi(t,\tilde\eta))\big(1-\P\chi_\eta(t,\tilde\eta)\big)d\tilde\eta, \\
\R(t,\eta)&=-\int_0^{B+C} \U(t,\tilde\eta)K'(\chi(t,\eta)-\chi(t,\tilde\eta))\big(1-\P\chi_\eta(t,\tilde\eta)\big)d\tilde\eta \nn \\ \label{def:slashR1} 
& \quad +\U(t,\eta)\int_0^{B+C} K'(\chi(t,\eta)-\chi(t,\tilde\eta))\big(1-\P\chi_\eta(t,\tilde\eta)\big)d\tilde\eta,
\end{align}
which can be uniquely solved using the method of characteristics in $L^\gamma([0,B+C])\times (L^\gamma([0,B+C])\cap L^\infty([0, B+C]))^2$. 
\end{theorem}

We remark that Definition~\ref{def:loes}, which lists all properties that have to be satisfied by weak conservative solutions, requires 
\begin{equation}\label{impl:moment}
u(t,\dott)\in E_2 \quad \text{ and }\quad \mu(t,(-\infty,\dott))=F(t,\dott)\in E_1^0.
\end{equation}
In our new coordinates these conditions can be formulated in terms of $\U\chi_\eta(t,\eta)$ and $\chi_\eta(t,\eta)$, whose time evolution cannot be described with the help of \eqref{systemnew:sec}. On the other hand, the imposed moment condition \eqref{cond:extramug}, which is preserved with respect to time, implies \eqref{impl:moment}. Indeed, let 
\begin{equation*}
u_{\pm\infty}(t)=\lim_{x\to \pm\infty}u(t,x), 
\end{equation*}
which both exist and are finite. Thus $u(t,\dott)\in E_2$, if and only if
\begin{multline*}
\int_{-\infty}^{-1}\big((u(t,x)-u_{-\infty}(t))^2 +u_x^2(t,x)\big)dx
+ \int_{-1}^1 \big(u^2+u_x^2\big)(t,x) dx\\+\int_1^\infty \big((u(t,x)-u_\infty(t))^2+u_x^2(t,x)\big)dx<\infty,
\end{multline*}
by the definition of $E_2$. Here we only consider one integral, since all of them can be investigated using similar ideas. One has
\begin{align*}
\int_{-\infty}^{-1} (u(t,x)-u_{-\infty}(t))^2 dx & = \int_{-\infty}^{-1} \left(\int_{-\infty}^x u_x(t,z) dz\right)^2dx\\
& =\int_{-\infty}^{-1}\left(\int_{-\infty}^x \frac{1}{(-z)^{\gamma/2}}(-z)^{\gamma/2}u_x(t,z)dz\right)^2 dx\\
& \leq \int_{-\infty}^{-1}\left(\int_{-\infty}^x \frac{1}{(-z)^\gamma}dz\right) \left(\int_{-\infty}^x (-z)^\gamma u_x^2(t,z) dz\right) dx\\
& \leq \frac{1}{\gamma-1}\int_{-\infty}^{-1}\frac{1}{(-x)^{\gamma-1}}dx \left (\int_\Real \vert z\vert ^\gamma u_x^2(t,z) dz \right)\\
& \leq \frac{1}{(\gamma-1)(\gamma-2)}\left(\int_\Real \vert z\vert^\gamma d\nu(t,z)\right),
\end{align*}
which is finite if $\gamma>2$ and \eqref{cond:extramug} is satisfied.

For the second condition, recall that 
\begin{equation*} 
\lim_{x\to -\infty} F(t,x)=0 \quad \text{ and }\quad \lim_{x\to \infty} F(t,x)=C. 
\end{equation*}
Thus $F(t,\dott)\in E_1^0$ if and only if 
\begin{equation*}
\int_{-\infty}^{-1} F^2(t,x)dx + \int_{-1}^1 F^2(t,x) dx +\int_1^\infty (F(t,x)-C)^2 dx <\infty,
\end{equation*}
by the definition of $E_0^1$. Again we only consider one integral, since all of them can be considered using similar ideas. One has 
\begin{align*}
\int_{-\infty}^{-1}& F^2(t,x) dx \\
& = \int_{-\infty}^{-1}\left(\int_{-\infty}^{x} \frac{1}{(-z)^{\gamma/2}}(-z)^{\gamma/2}d\mu(t,z)\right)^2 dx\\
& \leq \int_{-\infty}^{-1} \left(\int_{-\infty}^x \frac{1}{(-z)^\gamma} d\mu(t,z) dz \right)dx \left(\int_\Real \vert z\vert ^\gamma d\nu(t,z)\right) \\
& = \int_{-\infty}^{-1} \left(\frac{1}{(-x)^\gamma} F(t,x)-\gamma\int_{-\infty}^x \frac{1}{(-z)^{\gamma+1}} F(t,z) dz\right) dx \left(\int_\Real \vert z\vert ^\gamma d\nu(t,z) \right) \\
& \leq C\int_{-\infty}^{-1} \left(\frac{1}{(-x)^\gamma}+\gamma\int_{-\infty}^x \frac{1}{(-z)^{\gamma+1}} dz\right) dx \left(\int_\Real \vert z\vert ^\gamma d\nu(t,z) \right)\\
& \leq \frac{2C}{\gamma-1}\left(\int_\Real \vert z\vert ^\gamma d\nu(t,z) \right),
\end{align*}
which is finite if $\gamma>2$ and \eqref{cond:extramug} is satisfied.

We have proved the following theorem.  
\begin{theorem}\label{thm:main1}
Let $\gamma>2$, then for any initial data $(u_0, \mu_0)\in \D$ such that 
\begin{equation*}
\int_\Real (1+\vert x\vert^\gamma)d\mu_0<\infty 
\end{equation*}
 the Hunter--Saxton equation has a unique global conservative weak solution $(u,\mu)\in \D$ in the sense of Definition~\ref{def:loes}, which satisfies 
 \begin{equation*}
 \int_\Real (1+\vert x\vert^\gamma)d\mu(t)<\infty \quad \text{ for all } t\in \Real.
 \end{equation*}
\end{theorem}

\appendix

\section{Proof of Lemma~\ref{lem:appendix} and Lemma~\ref{lem:appendix2} }\label{sec:appendix}

\begin{lemma}[Lemma \ref{lem:appendix}] (i): \label{lemma:A1}
Let $g(t,\xi)$ be given by \eqref{def:g}. Then there exists a positive constant $\bar M$ such that 
\begin{align}\label{finestg2}
-\bar M&\vert t-s\vert^{3/2}   -p\big(t,y(t,\xi+u(t,y(t,\xi))(t-s))\big)u(t,y(t,\xi))(t-s)\\ \nn
& \quad-\int_\Real u(t,y(t,\eta))K\big(y(t,\xi+u(t,y(t,\xi))(t-s))-y(t,\eta)\big)\tilde H_\xi(t,\eta)d\eta\, (s-t)\\ \nn
& \leq g(s,\xi)-g\big(t,\xi+u(t,y(t,\xi))(t-s)\big)\\ \nn
&\leq \bar M\vert t-s\vert^{3/2} -p\big(t,y(t,\xi+u(t,y(t,\xi))(t-s))\big)u(t,y(t,\xi))(t-s)\\ \nn
& \quad-\int_\Real u(t,y(t,\eta))K\big(y(t,\xi+u(t,y(t,\xi))(t-s))-y(t,\eta)\big)\tilde H_\xi(t,\eta)d\eta\, (s-t).
\end{align}
(ii): Let $H(t,\xi)$ be defined by \eqref{def:H2}.  Then 
\begin{align*}
-(&\tilde M+\bar M)\vert t-s\vert^{3/2}   -p\big(t,y(t,\xi+u(t,y(t,\xi))(t-s))\big)u(t,y(t,\xi))(t-s)\\
& \quad-\int_\Real u(t,y(t,\eta))K\big(y(t,\xi+u(t,y(t,\xi))(t-s))-y(t,\eta)\big)\tilde H_\xi(t,\eta)d\eta\, (s-t)\\
& \leq H(s,\xi)-H\big(t,\xi+u(t,y(t,\xi))(t-s)\big)\\
&\leq (\tilde M+\bar M)\vert t-s\vert^{3/2} -p\big(t,y(t,\xi+u(t,y(t,\xi))(t-s))\big)u(t,y(t,\xi))(t-s)\\
& \quad-\int_\Real u(t,y(t,\eta))K\big(y(t,\xi+u(t,y(t,\xi))(t-s))-y(t,\eta)\big)\tilde H_\xi(t,\eta)d\eta\, (s-t).
\end{align*}
\end{lemma}

\begin{proof}
(i): We want to estimate
\begin{equation*}
g(s,\xi)-g(t,\xi+u(t,y(t,\xi))(t-s)).
\end{equation*}

Recalling \eqref{def:g}, we split the above difference into two terms as follows
\begin{align*}
g(s,\xi)-g(t,\xi+u(t,y(t,\xi))(t-s))
& = \int_{y(t,\xi+u(t,y(t,\xi))(t-s))}^{y(s,\xi)} p(s,z)dz\\
& \qquad+\int_{-\infty}^{y(t,\xi+u(t,y(t,\xi))(t-s))} \big(p(s,z)-p(t,z)\big) dz\\
& = \RN{1}_1+\RN{1}_2.
\end{align*}

$\RN{1}_1$: Since $p(t,x)$ is continuously differentiable on $[0,T]\times \Real$, we have by the mean value theorem that 
\begin{align*}
\RN{1}_1&=\int_{y(t,\xi+u(t,y(t,\xi))(t-s))}^{y(s,\xi)}p(s,z)dz\\
& = p(s,m)(y(s,\xi)-y(t,\xi+u(t,y(t,\xi))(t-s)))\\
& = \big(p(s,m)-p(t,y(t,\xi+u(t,y(t,\xi))(t-s)))\big)\\
&\qquad\qquad\qquad\qquad\qquad\times (y(s,\xi)-y(t,\xi+u(t,y(t,\xi))(t-s)))\\
&\quad+p(t,y(t,\xi+u(t,y(t,\xi))(t-s)))(y(s,\xi)-y(t,\xi+u(t,y(t,\xi))(t-s)))
\end{align*}
for some $m$ between $y(s,\xi)$ and $y(t,\xi+u(t,y(t,\xi))(t-s))$. Furthermore, we find
\begin{align*}
 &\abs{p(s,m)-p(t,y(t,\xi+u(t,y(t,\xi))(t-s)))}\abs{y(s,\xi)-y(t,\xi+u(t,y(t,\xi))(t-s))} \\
 &\quad\le   M\big( \abs{s-t}+\abs{m-y(t,\xi+u(t,y(t,\xi))(t-s))}\big) \\
 &\qquad\qquad\qquad\qquad\qquad\times \big( \abs{u(t,y(t,\xi))(t-s)}+\tilde M \abs{t-s}^{3/2} \big)  \\
 &\quad\le  M\big( \abs{s-t}+\abs{y(s,\xi)-y(t,\xi+u(t,y(t,\xi))(t-s))}\big)\abs{t-s}\\
 &\quad\le  M\big( \abs{s-t}+\abs{u(t,y(t,\xi))(t-s)}\big)\abs{t-s}\\
 &\quad\le  M \abs{s-t}^2,
\end{align*}
for some constant $M$ (that is increased during the calculation), using \eqref{eq:p_t}, \eqref{eq:p_x},  \eqref{finesty}, and \eqref{normu}.
In addition, we estimate
\begin{align*}
p(t,&y(t,\xi+u(t,y(t,\xi))(t-s)))(y(s,\xi)-y(t,\xi+u(t,y(t,\xi))(t-s)))\\
&=p(t,y(t,\xi+u(t,y(t,\xi))(t-s)))\Big(u(t,y(t,\xi))(t-s)\\
&\qquad+(y(s,\xi)-y(t,\xi+u(t,y(t,\xi))(t-s)))-u(t,y(t,\xi))(t-s)\Big).
\end{align*}
Finally, using \eqref{finesty} once more, there exists a positive constant $\hat M$ such that 
\begin{equation*}
\abs{\RN{1}_1+p(t,y(t,\xi+u(t,y(t,\xi))(t-s)))u(t,y(t,\xi))(t-s)}\leq \hat M\vert t-s\vert^{3/2}.
\end{equation*}

$\RN{1}_2$: Using \eqref{def:y}, \eqref{def:H}, and \eqref{evolp}, we can write   
\begin{align}
\RN{1}_2&=\int_{-\infty}^{y(t,\xi+u(t,y(t,\xi))(t-s))}\big(p(s,z)-p(t,z)\big)dz \nn\\
& = \int_{-\infty}^{y(t,\xi+u(t,y(t,\xi))(t-s))}\int_t^s p_t(\tau,z)d\tau  dz\nn\\
& = -\int_{-\infty}^{y(t,\xi+u(t,y(t,\xi))(t-s))}\int_t^s \int_\Real u(\tau,x)K'(z-x)d\mu(\tau,x)d\tau dz\nn\\
& = -\int_t^s \int_\Real u(\tau,x)K(y(t,\xi+u(t,y(t,\xi))(t-s))-x)d\mu(\tau,x) d\tau\nn \\
& =- \int_t^s \int_\Real u(\tau,y(\tau,\eta))K(y(t,\xi+u(t,y(t,\xi))(t-s))-y(\tau,\eta))\tilde H_\xi(\tau,\eta) d\eta d\tau\nn \\
& = -\int_t^s \int_\Real (u(\tau,y(\tau,\eta))-u(t,y(t,\eta)))\nn \\ 
&\qquad\qquad\qquad\qquad\times K(y(t,\xi+u(t,y(t,\xi))(t-s))-y(\tau,\eta))\tilde H_\xi(\tau,\eta) d\eta d\tau\nn \\
& \quad - \int_t^s \int_\Real u(t,y(t,\eta))(K(y(t,\xi+u(t,y(t,\xi))(t-s))-y(\tau, \eta))\nn\\
& \qquad \qquad \qquad\qquad \qquad  -K(y(t,\xi+u(t,y(t,\xi))(t-s))-y(t,\eta)))\tilde H_\xi(\tau,\eta) d\eta d\tau\nn \\
& \quad -\int_t^s \int_\Real u(t,y(t,\eta))\nn \\
& \qquad\qquad\qquad\times K(y(t,\xi+u(t,y(t,\xi))(t-s))-y(t,\eta))(\tilde H_\xi(\tau,\eta)-\tilde H_\xi(t,\eta)) d\eta d\tau\nn \\
& \quad -\int_t^s \int_\Real u(t,y(t,\eta))K(y(t,\xi+u(t,y(t,\xi))(t-s))-y(t,\eta))\tilde H_\xi(t,\eta) d\eta d\tau\nn \\
& = \RN{2}_1+\RN{2}_2+\RN{2}_3\nn\\
&  \quad -\int_t^s \int_\Real u(t,y(t,\eta))K(y(t,\xi+u(t,y(t,\xi))(t-s))-y(t,\eta))\tilde H_\xi(t,\eta) d\eta d\tau. \nn
\end{align}

Next we will show that each of the terms $\RN{2}_1$, $\RN{2}_2$, and $\RN{2}_3$ is of order $\vert t-s\vert^{3/2}$. Therefore it is important to keep in mind that $0\leq y_\xi$, $\tilde H_\xi\leq 1$.

$\RN{2}_1$: Since $u(t,x)$ is Hölder continuous with Hölder exponent $\frac12$ on $[0,T]\times \Real$, we have, by \eqref{charrough} and \eqref{Holder}, for all $\tau$ between $s$ and $t$ that
\begin{align*}
\vert u(\tau,y(\tau,\eta))-u(t,y(t,\eta))\vert &\leq D\big(\vert t-\tau\vert +\vert y(\tau,\eta)-y(t,\eta)\vert\big)^{1/2}\\
& \leq D\big(1+\norm{u(0,\dott)}_{L^\infty}+\frac14 CT\big)^{1/2}\vert t-\tau\vert^{1/2}\\
& \leq D\big(1+\norm{u(0,\dott)}_{L^\infty}+\frac14 CT\big)^{1/2}\vert t-s\vert^{1/2},
\end{align*}
and 
\begin{equation*}
\vert \RN{2}_1\vert  \leq DC\big(1+\norm{u(0,\dott)}_{L^\infty} +\frac14 CT\big)^{1/2} \vert t-s\vert^{3/2}.
\end{equation*}

$\RN{2}_2$: Since $K(\dott)$, given by \eqref{eq:K}, is smooth one has for all $\tau$ between $s$ and $t$ that 
\begin{align*}
\vert  K(y(t,\xi+u(t,y(t,\xi))(t-s))& -y(\tau,\eta))-K(y(t,\xi+u(t,y(t,\xi))(t-s))-y(t,\eta))\vert\\
 &= \vert \int_{y(t,\eta)}^{y(\tau,\eta)} K'(y(t,\xi+u(t,y(t,\xi))(t-s))-x) dx\vert \\
& \leq n\vert y(\tau,\eta)-y(t,\eta)\vert \\
& \leq n\Big(\norm{u(0,\dott)}_{L^\infty} +\frac14 CT\Big)\vert t-\tau\vert\\
& \leq n\Big(\norm{u(0,\dott)}_{L^\infty} +\frac14 CT\Big)\vert t-s\vert
\end{align*}
and 
\begin{equation*}
\vert \RN{2}_2\vert \leq nC\Big(\norm{u(0,\dott)}_{L^\infty}+\frac14 CT\Big)^2 \vert t-s\vert^2.
\end{equation*}

$\RN{2}_3$: Here integration by parts will be the key. Indeed, one has 
\begin{align*}
& \int_\Real u(t,y(t,\eta)) K(y(t,\xi+u(t,y(t,\xi))(t-s))-y(t,\eta))(\tilde H_\xi(\tau,\eta)-\tilde H_\xi(t,\eta))d\eta\\
&\quad =-\int_\Real u_x(t,y(t,\eta))\\
&\qquad\qquad \times K(y(t,\xi+u(t,y(t,\xi))(t-s))-y(t,\eta))(\tilde H(\tau,\eta)-\tilde H(t,\eta))y_\xi(t,\eta) d\eta\\
&\qquad  +\int_\Real u(t,y(t,\eta))\\
&\qquad\qquad \times K'(y(t,\xi+u(t,y(t,\xi))(t-s))-y(t,\eta))(\tilde H(\tau,\eta)-\tilde H(t,\eta))y_\xi(t,\eta) d\eta.
\end{align*}
Recalling \eqref{difflig:tH} and that $0\leq \tilde H_\xi\leq 1$, we have for all $\tau$ between $s$ and $t$ that 
\begin{equation*}
\vert \tilde H(t,\eta)-\tilde H(\tau,\eta)\vert \leq \Big(\norm{u(0,\dott)}_{L^\infty}+\frac14 CT\Big)\vert t-s\vert
\end{equation*}
and 
\begin{align*}
\vert &\int_\Real  u(t,y(t,\eta)) K(y(t,\xi+u(t,y(t,\xi))(t-s))-y(t,\eta))(\tilde H_\xi(\tau,\eta)-\tilde H_\xi(t,\eta))d\eta\vert\\
& \leq \Big(\norm{u(0,\dott)}_{L^\infty}+\frac14 CT\Big)\int_\Real \vert u_x(t,x)\vert K(y(t,\xi+u(t,y(t,\xi))(t-s))-x) dx\, \vert t-s\vert\\
& \quad + \Big(\norm{u(0,\dott)}_{L^\infty}+\frac14 CT\Big)\int_\Real \vert u(t,x)\vert K'(y(t,\xi+u(t,y(t,\xi))(t-s))-x)dx\, \vert t-s\vert\\
& \leq\Big( \big(\norm{u(0,\dott)}_{L^\infty} +\frac14 CT\big)C^{1/2}\pi^{1/2} +n\big(\norm{u(0,\dott)}_{L^\infty}+\frac14CT\big)^2 \pi\Big) \vert t-s\vert.
\end{align*}
Thus, we have
\begin{equation*}
\vert \RN{2}_3\vert \leq \Big(\big(\norm{u(0,\dott)}_{L^\infty} +\frac14 CT\big)C^{1/2}\pi^{1/2} +n\big(\norm{u(0,\dott)}_{L^\infty}+\frac14CT\big)^2 \pi\Big)\vert t-s\vert^2.
\end{equation*}

Combining now all these estimates, we have that there exists a positive constant $\check M$ such that 
\begin{align*}
&\abs{\RN{1}_2+\int_\Real u(t,y(t,\eta))K(y(t,\xi+u(t,y(t,\xi))(t-s))-y(t,\eta))\tilde H_\xi(t,\eta) d\eta\, (s-t)}\\
&\qquad\qquad\qquad\qquad\qquad\qquad\qquad\qquad\le \check M \vert t-s\vert^{3/2}
\end{align*}

(ii): Combine \eqref{def:H2}, \eqref{finesttH}, and \eqref{finestg2}.
\end{proof}

\begin{lemma}[Lemma~\ref{lem:appendix2}] \label{lemma:A2}
Consider the function $h$ defined by \eqref{eq:def_h}. Then \\
 (i)  $t\mapsto h(t,\eta)$  is continuous; \\
  (ii) $\eta\mapsto h(t,\eta)$ is Lipschitz and satisfies
 \begin{equation}\label{Lcont_h2}
\vert h(t,\eta_2)-h(t,\eta_1)\vert\leq (1+2n\norm{u(0,\dott)}_{L^\infty}+C+\frac{n}2Ct)\vert \eta_2-\eta_1\vert.
\end{equation}
  \end{lemma}
  
 \begin{proof}
 (i) First, we establish the continuity with respect to time. Recalling \eqref{normu}, \eqref{Holder}, \eqref{LIPP}, \eqref{Lup}, and \eqref{Lip:Chi}, we have 
 \begin{align*}
 \vert \U\P(t,\eta)-\U\P(s,\eta)\vert 
  &  \leq  \vert \U(t,\eta)\vert \vert \P(t,\eta)-\P(s,\eta)\vert + \vert \P(s,\eta)\vert \vert \U(t,\eta)-\U(s,\eta)\vert \\
  & \leq \big(\norm{u(0,\dott)}_{L^\infty} +\frac14 CT\big)\vert p(t,\chi(t,\eta))-p(s,\chi(s,\eta))\vert \\
  & \quad +C\vert u(t,\chi(t,\eta))-u(s,\chi(s,\eta))\vert\\
 & \leq \big(\norm{u(0,\dott)}_{L^\infty} +\frac14 CT\big) \big(\norm{u(0,\dott)}_{L^\infty}+\frac14 CT+2A\big)nC\vert t-s\vert \\
 & \quad + CD\big(1+2A\big)^{1/2}\vert t-s\vert^{1/2}.
 \end{align*}
 For the second term of $h(t,\eta)$, note that one can write 
 \begin {align*}
 \int_0^{B+C}& \U(t,\tilde \eta)K(\chi(t,\eta)-\chi(t,\tilde\eta))(1-\P\chi_\eta(t,\tilde\eta))d\tilde\eta \\
 & = \int_0^{B+C} \U(t,\tilde \eta)K(\chi(t,\eta)-\chi(t,\tilde \eta))d\tilde\eta
 -\int_\Real K(\chi(t,\eta)-y)up(t,y)dy.
 \end{align*}
 Using \eqref{Holder}, \eqref{Lip:Chi}, and Lemma~\ref{lemma:Kp}, one has 
 \begin{align*}
 \Big\vert \int_0^{B+C}&\Big( \U(t,\tilde\eta)K(\chi(t,\eta) -\chi(t,\tilde \eta)) -\U(s,\tilde\eta)K(\chi(s,\eta)-\chi(s,\tilde\eta))\Big)d\tilde\eta\Big\vert \\
 & \leq \Big\vert \int_0^{B+C} (\U(t,\tilde\eta)-\U(s,\tilde\eta))K(\chi(t,\eta)-\chi(t,\tilde\eta))d\tilde\eta\Big\vert\\
 & \quad +\Big\vert \int_0^{B+C} \U(s,\tilde\eta)\big(K(\chi(t,\eta)-\chi(t,\tilde\eta))-K(\chi(s,\eta)-\chi(s,\tilde \eta))\big)d\tilde\eta\Big\vert\\
 & \leq (B+C)\norm{\U(t,\dott)-\U(s,\dott)}_{L^\infty}\\
 & \quad + 2n\big(\norm{u(0,\dott)}_{L^\infty} +\frac 14CT\big) (B+C)\norm{\chi(t,\dott)-\chi(s,\dott)}_{L^\infty}\\
 & \leq (B+C)D\big(1+2A\big)^{1/2}\vert t-s\vert^{1/2}\\
 & \quad + 2n(\norm{u(0,\dott)}_{L^\infty}+\frac14 CT)(B+C)2A\vert t-s\vert.
 \end{align*}
 Last but not least, recalling \eqref{Holder}, Lemma~\ref{lemma:Kp}, \eqref{LIPP}, \eqref{L1p}, and \eqref{Lip:Chi}, we have
 \begin{align*}
 \Big\vert \int_\Real \big(K(\chi(t,\eta)-y)up(t,y)&-K(\chi(s,\eta)-y)up(s,y)\big)dy\Big\vert \\
 & \leq \Big\vert \int_\Real (K(\chi(t,\eta)-y)-K(\chi(s,\eta)-y))up(t,y)dy\Big\vert\\
 & \quad +\Big\vert \int_\Real K(\chi(s,\eta)-y)(up(t,y)-up(s,y))dy\Big\vert\\
&\leq  n\norm{\chi(t,\dott)-\chi(s,\dott)}_{L^\infty} \int_\Real \vert u\vert p(t,y)dy\\
 & \quad + \pi\norm{up(t,\dott)-up(s,\dott)}_{L^\infty}\\
 & \leq 2nA(\norm{u(0,\dott)}_{L^\infty}+\frac14CT)B\vert t-s\vert\\
 & \quad + \pi n(\norm{u(0,\dott)}_{L^\infty}+\frac14 CT)^2C\vert t-s\vert+\pi CD(t-s)^{1/2}.
 \end{align*}
 Thus we have that $h(t,\eta)$ is H{\"o}lder continuous with exponent $\frac12$ with respect to time.
 
\medskip 
 (ii) To establish the Lipschitz continuity with respect to space is a bit more involved, but follows pretty much the same lines. As a closer look at \eqref{interpolation} reveals, one has 
 \begin{equation*}
\sigma G(t,\chi(t,\eta)-)+ (1-\sigma)G(t,\chi(t,\eta)+)=\eta \quad \text{ for some } \sigma\in [0,1].
\end{equation*}
This means especially given $\eta\in [0,B+C]$, there exist $\eta^-\leq \eta\leq \eta^+$ such that 
\begin{equation*}
\chi(t,\eta^-)=\chi(t,\eta)=\chi(t,\eta^+)
\end{equation*}
and 
\begin{equation*}
\eta^-=G(t,\chi(t,\eta)-) \quad \text{ and }\quad \eta^+=G(t,\chi(t,\eta)+).
\end{equation*}
In view of Definition~\ref{def:loes} (v), \eqref{def:nu_def2}, \eqref{eq:p_x}, and \eqref{eq:nu_def}, we then have for $\eta_1<\eta_2$ such that 
$\chi(t,\eta_1)\not = \chi(t,\eta_2)$ that 
\begin{align*}
\vert \U\P(t,\eta_2)-&\U\P(t,\eta_1)\vert \\
& =\vert \U\P(t,\eta_2^-)-\U\P(t,\eta_1^+)\vert \\
& = \vert up(t,\chi(t,\eta_2^-))-up(t,\chi(t,\eta_1^+))\vert \\
& \leq  \int_{\chi(t,\eta_1^+)}^{\chi(t,\eta_2^-)} \vert u_xp+up_x\vert (t,y)dy\\
& \leq  \int_{\chi(t,\eta_1^+)}^{\chi(t,\eta_2^-)} (u_x^2+p^2+n\vert u\vert p)(t,y) dy\\
& \leq (1+\norm{p(t,\dott)}_{L^\infty}+n\norm{u(t,\dott)}_{L^\infty})\int_{\chi(t,\eta_1^+)}^{\chi(t,\eta_2^-)}(u_x^2+p)(t,y)dy\\
& \leq (1+ C+n(\norm{u(0,\dott)}_{L^\infty}+\frac14 Ct))(G(t,\chi(t,\eta_2^-))-G(t,\chi(t,\eta_1^+)))\\
& =(1+ C+n(\norm{u(0,\dott)}_{L^\infty}+\frac14 Ct))\vert \eta_2^--\eta_1^+\vert\\
& \leq (1+ C+n(\norm{u(0,\dott)}_{L^\infty}+\frac14 Ct))\vert \eta_2-\eta_1\vert.
\end{align*}
As far as the second part is concerned, observe that (cf.~\eqref{eq:PP})
\begin{equation*}
\int_0^{B+C} \U(t,\tilde \eta)K(\chi(t,\eta)-\chi(t,\tilde\eta))(1-\P\chi_\eta(t,\tilde \eta))d\tilde \eta
= \int_\Real u(t,y)K(\chi(t,\eta)-y)d\mu(t,y),
\end{equation*}
and we have 
\begin{align*}
\vert \int_0^{B+C} \U(t,\tilde\eta)& (K(\chi(t,\eta_2)-\chi(t,\tilde\eta))-K(\chi(t,\eta_1)-\chi(t,\tilde \eta)))(1-\P\chi_\eta(t,\tilde\eta))d\tilde\eta\vert\\
& =\vert \int_\Real u(t,y) (K(\chi(t,\eta_2)-y)-K(\chi(t,\eta_1)-y))d\mu(t,y)\vert\\
& = \vert\int_\Real u(t,y)\int_{\chi(t,\eta_1)}^{\chi(t,\eta_2)}K'(z-y)dz d\mu(t,y)\vert\\
& = \vert \int_{\chi(t,\eta_1^+)}^{\chi(t,\eta_2^-)} \int_\Real u(t,y)K'(z-y)d\mu(t,y)dz \vert\\
& \leq (\norm{u(0,\dott)}_{L^\infty}+\frac14 Ct) \int_{\chi(t,\eta_1^+)}^{\chi(t,\eta_2^-)} \int_\Real nK(z-y)d\mu(t,y)dz\\
& = n(\norm{u(0,\dott)}_{L^\infty}+\frac14 Ct)\int_{\chi(t,\eta_1^+)}^{\chi(t,\eta_2^-)} p(t,z)dz\\
& =n(\norm{u(0,\dott)}_{L^\infty}+\frac14 Ct) \big(G(t,\chi(t,\eta_2^-))-G(t,\chi(t,\eta_1^+))\big)\\
& \leq n(\norm{u(0,\dott)}_{L^\infty}+\frac14 Ct)\vert \eta_2-\eta_1\vert.
\end{align*}
Thus we have that $h(t,\eta)$ is Lipschitz continuous with respect to space
with
\begin{equation*}
\vert h(t,\eta_2)-h(t,\eta_1)\vert\leq \big(1+2n\norm{u(0,\dott)}_{L^\infty}+C+\frac{n}2Ct\big)\vert \eta_2-\eta_1\vert.
\end{equation*}
\end{proof}

\begin{lemma}\label{lem:momp}
Given $\gamma>2$ and a weak conservative solution $(u,\mu)$ in the sense of Definition~\ref{def:loes}, which satisfies 
\begin{equation*}
\int_\Real (1+\vert x\vert ^\gamma)d\mu(t)<\infty \quad \text{ for all }t\in \Real.
\end{equation*}
Then for any $n\in \mathbb{N}$ such that $n\geq \gamma$ 
\begin{equation*}
\int_\Real (1+\vert x\vert^\gamma)p(t,x)dx <\infty \quad \text{ for all }t\in \Real.
\end{equation*}
\end{lemma}

\begin{proof}
Due to \eqref{L1p} it suffices to show 
\begin{equation*}
 \int_{-\infty}^{-1} \vert x\vert^\gamma p(t,x)dx +\int_1^\infty \vert x\vert^\gamma p(t,x)dx<\infty.
\end{equation*}
Since $p(t,\dott)$ is non-negative and both integrals can be estimated using the same ideas, we only present the details for the first one. Direct computations yield
\begin{align*}
\int_{-\infty}^{-1} \vert x\vert^\gamma p(t,x)dx&= \int_{-\infty}^{-1} ( x^2) ^{\gamma/2} \int_\Real \frac{1}{(1+(x-y)^2)^n}d\mu(t,y) dx\\
& = \int_{-\infty}^{-1} ( ((x-y)+y)^2) ^{\gamma/2} \int_\Real \frac{1}{(1+(x-y)^2)^n}d\mu(t,y) dx\\
& \leq 2^{\gamma/2} \int_{-\infty}^{-1}\int_\Real \left(\frac{(x-y)^2+y^2}{1+(x-y)^2}\right)^{\gamma /2}\frac{1}{(1+(x-y)^2)^{n-\gamma/2}} d\mu(t,y) dx\\
& \leq 2^{\gamma/2} \int_{-\infty}^{-1}\int_\Real (1+y^2)^{\gamma/2}\frac{1}{(1+(x-y)^2)^{n-\gamma/2}} d\mu(t,y)dx\\
&\leq  2^{\gamma/2}\int_\Real \left(\int_\Real \frac{1}{(1+(x-y)^2)^{n-\gamma/2}}dx\right) (1+y^2)^{\gamma/2}d\mu(t,y)\\
& \leq 2^{\gamma/2}\int_\Real \left(\int_\Real \frac{1}{1+z^2}dz\right) (1+y^2)^{\gamma/2}d\mu(t,y)\\
& \leq 2^{\gamma/2}\pi \int_\Real (1+y^2)^{\gamma/2}d\mu(t,y)\\
& \leq 2^{\gamma}\pi\int_\Real (1+\vert y\vert^\gamma) d\mu(t,y).
\end{align*}
In the above calculations we used that $(a+b)^2\leq 2(a^2+b^2)$ and that $(a+b)^2\leq2\max (a^2, b^2)$. 
\end{proof}

%
%
%


\begin{thebibliography}{99}

\bibitem{B}
V. I.~Bogachev,
{\it Measure Theory. Vol. I, II.} Springer-Verlag, Berlin, 2007.


\bibitem{BCZ}
A.~Bressan, G.~Chen, and Q.~Zhang.
Uniqueness of conservative solutions to the Camassa--Holm equation via characteristics. 
{\it Discrete Contin. Dyn. Syst.} 35 (2015) 25--42.

\bibitem{BC}
A.~Bressan and A.~Constantin.
Global solutions of the Hunter--Saxton equation. {\em SIAM J. Math. Anal.} 37 (2005) 996--1026.

\bibitem{BHR}
A.~Bressan, H.~Holden, and X.~Raynaud.
Lipschitz metric for the Hunter--Saxton equation.  {\it J. Math. Pures Appl.} 94 (2010) 68--92.

\bibitem{BZZ}
A.~Bressan, P.~Zhang, and Y.~Zheng. Asymptotic variational wave equations.
{\em Arch. Ration. Mech. Anal.} 183 (2007), 163--185.

\bibitem{CGH}
J. A.~Carrillo, K.~Grunert, and H.~Holden.
A Lipschitz metric for the Hunter--Saxton equation. {\it Comm. Partial Differential Equations} 44 (2019) 309--334.

\bibitem{MR2796054}
C. M.~Dafermos. Generalized characteristics and the {H}unter--{S}axton equation.  {\it J. Hyperbolic Differ. Equ.} 8 (2011) 159--168.

\bibitem{GN}
K.~Grunert and A.~Nordli.
Existence and Lipschitz stability for $\alpha$-dissipative solutions of the two-component Hunter--Saxton system. 
{\it J. Hyperbolic Differ. Equ. } 15 (2018) 559--597.

\bibitem{GNS}
K.~Grunert, A.~Nordli, and S.~Solem.
Numerical conservative solutions of the Hunter--Saxton equation. 
{\it BIT Numerical Mathematics}  61 (2021) 441--471.

\bibitem{GT}
K.~Grunert and M.~Tandy.
Lipschitz stability for the Hunter--Saxton equation.
\arxiv{2103.10227}.

\bibitem{MR4151173}
H.~Holden, K.~H. Karlsen, and P.H.C.~Pang.
The {H}unter--{S}axton equation with noise.
 {\it J. Differential Equations} 270 (2021) 725--786.

\bibitem{HolKarRis:sub05}
H.~Holden, K.~H. Karlsen, and N.~H. Risebro. Convergent difference schemes for the
  Hunter--Saxton equation. {\it Math. Comp.} 76 (2007) 699--744.

\bibitem{HR}
H.~Holden and X.~Raynaud.
Global conservative solutions of the Camassa--Holm equation --- a Lagrangian point of view.
{\it Comm. Partial Differential Equations} 32 (2007) 1511--1549.

\bibitem{MR1135995}
J.~K. Hunter and R.~Saxton. Dynamics of director fields. {\em SIAM J. Appl. Math.}
  51~(6) (1991) 1498--1521.

\bibitem{MR1361013}
J.~K. Hunter and Y.~X. Zheng. On a nonlinear hyperbolic variational equation. {I}.
  {G}lobal existence of weak solutions  {\em Arch. Rational Mech. Anal.} 129~(4)
  (1995) 305--353.

\bibitem{MR1361014}
J.~K. Hunter and Y.~X. Zheng. On a nonlinear hyperbolic variational equation.
  {II}. {T}he zero-viscosity and dispersion limits.  {\em Arch. Rational Mech. Anal.}
  129~(4) (1995) 355--383.

\bibitem{MR1978343z}
B.~Khesin and G.~Misio\l ek.
Euler equations on homogeneous spaces and {V}irasoro orbits.
{\em Adv. Math.} 176 (2003) 116--144.


\bibitem {MR3945049}
J. M. Lee.    
Geometric approach on the global conservative solutions of the
              {C}amassa--{H}olm equation. {\it J. Geom. Phys.} 142 (2019) 137--150.

\bibitem {MR2403320}
J.~Lenells. The {H}unter--{S}axton equation: a geometric approach.
 {\em SIAM J. Math. Anal.} 40 (2008) 266--277.

\bibitem {MR2318260}
 J.~Lenells. Weak geodesic flow and global solutions of the
              {H}unter--{S}axton equation.  {\em Discrete Contin. Dyn. Syst.}
 18 (2007) 643--656. 

 \bibitem{MR2348278}
 J.~Lenells. The {H}unter--{S}axton equation describes the geodesic flow on
              a sphere. {\em J. Geom. Phys.}  57 (2007) 2049--2064.
 
\bibitem{LT}
H. Li and G. Toscani. Long--time asymptotics of kinetic  models of granular flows. {\em Arch. Ration. Mech. Anal.} 172 (3) (2004) 407--428.


\bibitem{anders}
 A.~Nordli.
 A Lipschitz metric for conservative solutions of the two-component Hunter--Saxton system.
{\it Methods Appl. Anal.} 23 (2016) 215--232.

\bibitem{MR2653251}
M.~Wunsch. The generalized {H}unter--{S}axton system. 
 {\it SIAM J. Math. Anal.} 42 (2010) 1286--1304.
 
\bibitem{MR2525162}
M.~Wunsch. On the {H}unter--{S}axton system. {\it Discrete Contin. Dyn. Syst. Ser. B} 12 (2009)
 647--656.
 
\bibitem{MR1668954}
P.~Zhang and Y.~Zheng. On oscillations of an asymptotic equation of a nonlinear
  variational wave equation.  {\em Asymptot. Anal.} 18~(3-4) (1998) 307--327.

\bibitem{MR1701136}
P.~Zhang and Y.~Zheng. On the existence and uniqueness of solutions to an
  asymptotic equation of a variational wave equation. {\em Acta Math. Sin. (Engl.
  Ser.)} 15~(1) (1999) 115--130.

\bibitem{MR1799274}
P.~Zhang and Y.~Zheng. Existence and uniqueness of solutions of an asymptotic
  equation arising from a variational wave equation with general data. {\em  Arch.
  Ration. Mech. Anal.} 155~(1) (2000) 49--83.


\end{thebibliography}
\end{document}